\documentclass[12pt,oneside,smallextended]{amsart}
\usepackage{amscd,amsmath,amssymb,amsfonts}
\usepackage[cmtip, all]{xy}
\usepackage{pstricks}
\usepackage[colorlinks]{hyperref}
\usepackage{enumerate}
\usepackage{microtype} 
\usepackage{colonequals}
\usepackage{fullpage}

\newlength{\rulebreite}

\def\timesover#1#2#3{\ \xymatrix@1@=0pt@M=0pt{ _{#1}&\times&_{#2} \\& ^{#3}&}\ }
\def\otimesover#1#2#3{\ \xymatrix@1@=0pt@M=0pt{ _{#1}&\otimes&_{#2} \\& ^{#3}&}\ }

\usepackage[all]{xy}
\theoremstyle{definition}
\newtheorem{thm}{Theorem}
\newtheorem{lem}[thm]{Lemma}
\newtheorem{cor}[thm]{Corollary}
\newtheorem{prop}[thm]{Proposition}
\theoremstyle{definition}
\newtheorem{defn}[thm]{Definition}

\newtheorem{rmk}[thm]{Remark}

\newtheorem{ex}[thm]{Example}

\numberwithin{thm}{subsection}
\numberwithin{equation}{subsection}

\newcommand{\e}{\epsilon}

\newcommand{\Spec}{{\rm Spec \,}}

\newcommand{\sF}{{\mathcal F}}
\newcommand{\sG}{{\mathcal G}}
\newcommand{\sH}{{\mathcal H}}
\newcommand{\sI}{{\mathcal I}}

\newcommand{\sP}{{\mathcal P}}

\newcommand{\A}{{\mathbb A}}

\newcommand{\C}{{\mathbb C}}

\newcommand{\F}{{\mathbb F}}

\renewcommand{\P}{{\mathbb P}}
\newcommand{\Q}{{\mathbb Q}}
\newcommand{\R}{{\mathbb R}}

\newcommand{\Z}{{\mathbb Z}}

\renewcommand{\bar}{\overline}

\def\tilde{\widetilde}
\numberwithin{equation}{subsection}

\begin{document}

\title{The (non-uniform) Hrushovski-Lang-Weil estimates}
\author{K. V. Shuddhodan}
\email{kvshud@ihes.fr}
\address{Institut des Hautes \'Etudes Scientifiques \\ Le Bois-Marie 35 rte de Chartres\\ 91440 Bures-sur-Yvette France}
\subjclass[2010]{14F05,~14F43,~14F45,~14G15,~14G17}

\begin{abstract}
In this article, we obtain a non-uniform version of Hrushovski's generalization of the Lang-Weil estimates using $\ell$-adic methods and without recourse to alterations.~Our method also implies that the associated generating function is rational.~Along the way we obtain a bound for the local terms for a class of non-proper correspondence which could be of independent interest.
\end{abstract}

\maketitle

\begingroup
\hypersetup{linkcolor=black}
\tableofcontents
\endgroup

\section{Introduction}

Let $k$ be an algebraically closed field of characteristic $p>0$.~For any $p$-primary number $q=p^r$,~and any scheme $X/k$,~~by $X^{(q)}$ we mean the base change of $X/k$ along the $r^{\mathrm{th}}$ iterate of the absolute Frobenius of $k$.~Let $F^{(q)}_{X/k}$ be the induced relative Frobenius morphism from $X$ to $X^{(q)}$.

%$F^{(q)}_{X/k}$ is a morphism of schemes over $k$.

Now suppose $X$ is a closed subvariety of $\A^n_k$.~Consider a \textit{correspondence} between $X$ and $X^{(q)}$ given by a closed subvariety $C$ of $X \times_k X^{(q)}$.~Let $c_1$ and $c_2$ denote the projections from $C$ to $X$ and $X^{(q)}$ respectively.~Suppose that $c_1$ and $c_2$ are dominant,~and at least one of them is quasi-finite. 

Note that $X$ and $X^{(q)}$ have natural compactifications (say $\bar{X}$ and $\bar{X^{(q)}}$ respectively) inside $\P^n_k$.~Similarly $C$ has a natural compactification (say $\bar{C}$) inside $\P^{n}_k \times \P^n_k$ which in turn can be embedded inside $\P^{n^2+2n}_k$ (the Segre embedding).~Let $\Delta^{(q)}_{X/k}$ be the graph of $F^{(q)}_{X/k}$ considered as a subscheme of $X \times_k X^{(q)}$.

Combining techniques from model theory and intersection theory, Hrushovski proved the following generalisation of the Lang-Weil estimates (\cite{Lang_Weil},~Theorem 1).

\begin{thm}[Hrushovski-Lang-Weil estimates]\label{Hrushovski_Lang_Weil}(\cite{Hrushovski},~Theorem 1.1 (1))

Let $n, d_1$ and $d_2$ be nonnegative integers.~There exists an integer $M(n,d_1,d_2)$ depending only on $n$,~$d_1$ and $d_2$ satisfying the following properties.

\begin{enumerate}

\item For any choice of $p, q, k, n, X$ and $C$ as above with $\text{deg}(\bar{X}) \leq d_1$,~$\text{deg}(\bar{C}) \leq d_2$ and $q \geq M(n,d_1,d_2)$,~the schematic intersection $C \cap \Delta^{(q)}_{X/k}$ is finite.

\item Moreover we have the following bound for the number of points in the intersection

\begin{equation}
|\#C \cap \Delta^{(q)}_{X/k}(k)-\frac{\delta}{\delta'}q^{\text{dim}(X)}| \leq M(n,d_1,d_2)q^{\text{dim}(X)-\frac{1}{2}}.
\end{equation}

\noindent Here $\delta$ and $\delta'$ are the degree and the inseparable degree of $c_1$ and $c_2$ respectively.

\end{enumerate}

\end{thm}

Now suppose that $k$ is an algebraic closure of a finite field $\F_q$.~For any scheme $X/k$ defined over $\F_q$,~let $F_{X/\F_q} \colon X \to X$ be the geometric Frobenius with respect to $\F_q$.~Let $\Delta^{(n)}$ be the graph of  $F^{n}_{X/\F_q}$,~considered as a closed subscheme of $X \times_k X$.~Recently Varshavsky (\cite{Varshavsky_Hrushovski}) gave a geometric proof of the following corollary to Hrushovski's result.

\begin{cor}\label{dense_correspondence}(\cite{Hrushovski},~Corollary 1.2,~\cite{Varshavsky_Hrushovski},~Theorem 0.1)

Let $c:C \to X \times_k X$ be a morphism of schemes finite type over $k$ such that,

\begin{enumerate}

\item $C$ and $X$ are irreducible.

\item $c_1$ and $c_2$ are dominant.

\item $X$ is defined over $\F_q$.

\end{enumerate}

Then for $n$ sufficiently large,~$c^{-1}(\Delta^{(n)})$ is non-empty .

\end{cor}

Corollary \ref{dense_correspondence} has applications to algebraic dynamics (\cite{Fakhruddin},\cite{Amerik}),~group theory (\cite{Borisov_Sapir}) and algebraic geometry (\cite{Esnault_Mehta},~\cite{Esnault_Srinivas_Bost}).

This article aims to prove a non-uniform avatar of Theorem \ref{Hrushovski_Lang_Weil} using geometric methods.~Our methods also imply that generating function keeping track of the fixed points is rational.~More precisely we obtain the following result.

\begin{thm}\label{non_uniform_estimate}

Let $c:C \to X \times_k X$ be a morphism of schemes finite type over $k$.~Suppose that $X$ is defined over $\F_q$ and that $c_2$ is quasi-finite.~Then 

\begin{enumerate}

\item there exists an integer $N$ such that,~$\text{Fix}(c^{(n)}):=c^{-1}(\Delta^{(n)})$ is finite (over $k$) for every $n \geq N$.~Further there exists a real number $M$ such that for any $n>>0$,

\begin{center}

$\#\text{Fix}(c^{(n)})(k) \leq Mq^{n\text{dim}(C)}$.

\end{center}

\item The formal series $Z(c,t):=\sum_{n \geq N} \#\text{Fix}(c^{(n)})(k)t^n \in \Z[[t]] \subset \Q((t))$ is a rational function in $t$,~that is it belongs to $\Q(t)$.

\item Moreover when $C$ and $X$ are irreducible with $c_1$ and $c_2$ dominant,~there exists a real number $M'$ such that for any $n >>0$

\begin{equation}\label{main_theorem_estimate}
|\#\text{Fix}(c^{(n)})(k)-\frac{\delta}{\delta'}q^{n \text{dim}(X)}| \leq M'q^{n(\text{dim}(X)-\frac{1}{2})},
\end{equation}

\noindent where $\delta$ and $\delta$' are the degree and inseparable degree of $c_1$ and $c_2$ respectively.

\end{enumerate}

\end{thm}

\begin{rmk}\label{Theorem_implies_Varshavsky}

The inequality (\ref{main_theorem_estimate}) under the assumptions of Corollary \ref{dense_correspondence} implies that

\begin{equation}\label{asymptotic_main_theorem}
\lim_{n \to \infty} \frac{\#\text{Fix}(c^{(n)})(k)}{q^{n\text{dim}(X)}}=\frac{\delta}{\delta'},
\end{equation}

\noindent and hence $\#\text{Fix}(c^{(n)})=c^{-1}(\Delta^{(n)})$ is nonempty for $n$ sufficiently large (see Corollary \ref{dense_correspondence}).

\end{rmk}

Our proof of Theorem \ref{non_uniform_estimate} is closely related to the circle of ideas around Deligne's conjecture on the Lefschetz-Verdier trace formula for non-proper varieties over an algebraic closure of a finite field.~The conjecture was first verified by Pink (\cite{Pink}) assuming the resolution of singularities and later by Fujiwara (\cite{Fujiwara}) unconditionally,~following an idea of Gabber.~Subsequently Varshavsky obtained an effective generalization of Fujiwara's trace formula (\cite{Varshavsky}).~The key notion in both Fujiwara and Varshavsky's approach is that of a \textit{contracting correspondence},~which ensures vanishing of local terms along the boundary of a compactification.

The connection between Theorem \ref{non_uniform_estimate} and Deligne's conjecture was already observed by Hrushovski (\cite{Hrushovski},~Section 1.1).~However as noted there,~the non-properness of $c_1$ (crucial to make sense of Deligne's conjecture) rules out a `direct' argument.~This connection was reestablished in the recent proof of Corollary \ref{dense_correspondence} by Varshavsky (\cite{Varshavsky_Hrushovski}).

Indeed,~as a first step in the proof of Corollary \ref{dense_correspondence} in \cite{Varshavsky_Hrushovski} it is shown that (at the cost of shrinking $X$) one can assume the natural compactification of $c$ is \textit{locally invariant} along the boundary.~That the boundary can be made only locally invariant and not globally,~is a manifestation of $c_1$ not being proper.~Then using a construction of Pink (\cite{Pink}),~Varshavsky obtains a trace formula which in turn implies an asymptotic growth of the form (\ref{asymptotic_main_theorem}) for a modified correspondence,~which is sufficient to show the desired nonemptiness.

In this article, we attempt to compute $\text{Fix}(c^{(n)})$ directly using the Lefschetz-Verdier trace formula.~In particular we do not use alterations or resolution of singularities \footnote{Though one is tempted to find a middle path combining the methods of this article with those in \cite{Varshavsky_Hrushovski},~it appears that a plausible proof of uniformity using this approach leads to many difficulties,~primarily among them being the lack of a suitable \textit{norm} on correspondences of $\ell$-adic sheaves.} unlike Hrushovski (in \cite{Hrushovski}) and Varshavsky (in \cite{Varshavsky_Hrushovski}).~As mentioned earlier the non-properness of $c_1$ immediately leads to technical difficulties,~most important of which is the possible nonvanishing of local terms along the boundary.~Hence an important step for us in the proof of Theorem \ref{non_uniform_estimate} is the following estimate (Theorem \ref{bound_local_term_invariant_subscheme_introduction}) for these local terms,~which could be of independent interest.~We now describe the various terms appearing in Theorem \ref{bound_local_term_invariant_subscheme_introduction}.

As before let $k$ be an algebraic closure of a finite field $\F_q$.~Let $c  \colon C \to X \times_k X$ be a correspondence (define over $\F_q$),~with $C$ and $X$ proper over $k$.~Let $c_1$ and $c_2$ be the induced maps from $C$ to $X$.~Denote by $c^{(n)}$ the correspondence $(F^{n}_{X/\F_q} \circ c_1,c_2) \colon C \to X \times_k X$.

Let $Z \subseteq X$ be a closed subset of $X$ defined over $\F_q$,~which is \textit{locally} $c$-invariant over $\F_q$.~That is there exists a cover of $Z$ by open sets $U_i$ of $X$,~ defined over $\F_q$,~such that $c_2^{-1}(U_i \cap Z) \cap c_1^{-1}(U_i) \subseteq c_1^{-1}(Z)$.

Let $\sF_0 \in D^{b}_{\leq w}(X_0,\bar{\Q}_{\ell})$ be a mixed sheaf of weight less than or equal to $w$ on $X_0$ (the chosen model of $X$ over $\F_q$).~Assume that $\sF_0|_{Z_0}$ belongs to $~^pD^{\leq a}(Z_0,\bar{\Q}_{\ell})$.~Let $\sF$ be the base change of $\sF_0$ to $k$.

Let $u$ be an element in $\text{Hom}_{D^b_c(C,\bar{\Q}_{\ell})}(c_1^*\sF,c_2^{!}\sF)$ (a \textit{cohomological correspondence} of $\sF$ lifting $c$).~Then for any $n \geq 1$ we have a cohomological correspondence $u^{(n)}$ of $\sF$ lifting $c^{(n)}$,~given by the natural structure of a Weil sheaf on $\sF$.

Moreover fix a field isomorphism (say $\tau$) of $\bar{\Q}_{\ell}$ with $\C$.

Since $\text{Fix}(c^{(n)})$ is proper,~and $Z$ a locally $c^{(n)}$-invariant closed subset,~we can make sense of the local terms $\text{LT}(u^{(n)}|_Z)$ (see Lemma \ref{neighbourhood_fixed_points} and Section \ref{local_term_subscheme_invariant_neighbourhood_fixed_points}).~In this setting we obtain the following bound on the local terms.

\begin{thm}\label{bound_local_term_invariant_subscheme_introduction}
For any $\e>0$,~there exists a natural number $N(\e)$ and a positive real number $M(\tau)$,~such that for any $n \geq N(\e)$,

\begin{equation}\label{bound_local_term_invariant_subscheme_introduction_1}
|\text{LT}(u^{(n)}|_Z)| \leq M(\tau)q^{n(\frac{(w+a+\text{dim}(Z))}{2}+\e)}.
\end{equation}

Here the norm on the left is with respect to the chosen isomorphism $\tau$.
\end{thm}

More generally we prove such a bound for a class of correspondences,~which we call \textit{essentially proper} over $\F_q$ (see Definition \ref{good_correspondences}).

Note that if $Z$ were $c$-invariant (and not just locally),~Theorem \ref{bound_local_term_invariant_subscheme_introduction} would be an immediate consequence of the Lefschetz-Verdier trace formula.~The main difficulty here is that even though the local terms along a locally invariant subset are defined,~they do not correspond to a \textit{global} term in a natural fashion.

Before we give a brief outline of the proof,~we describe the notations and conventions followed in this article.

\subsubsection{Acknowledgements:}

I am grateful to Prof.~Vasudevan Srinivas for his constant support and valuable insights at various stages of this project.~I would like to thank Prof.~H\'el\`ene Esnault for her consistent encouragement and stimulating discussions.~I am grateful to Prof.~Yakov Varshavsky for his constructive remarks and suggestions,~which have greatly improved the exposition in the article.~I would also like to thank Prof. Ehud Hrushovski for insightful discussions,~and in particular bringing the question of rationality to my attention.~I would also like to thank Prof.~Arvind Nair for helpful discussions.~This work has been supported by the Einstein Foundation Berlin and ISF grant 822/17 of Prof. Yakov Varshavsky.~Finally I am thankful
to the referees for a careful reading of the article and the many suggestions and corrections which have improved the exposition and readability of the article.

\section{Notations and conventions}\label{notations_convention}

\subsubsection{}All the schemes appearing in this article are assumed to be separated over $\Z$.~For any scheme $X$,~$\pi_0(X)$ denotes the set of its connected components.~A \textit{variety} over a field $k$ is a geometrically integral scheme of finite type over $k$.~For any integral scheme $X/k$,~by $k(X)$ we mean the function field of $X$.

\subsubsection{} Let $k$ be either a finite field or an algebraically closed field.~For a scheme $X$ of finite type over $k$,~by $D^b_c(X,\bar{\Q}_{\ell})$ we mean the bounded derived category of $\bar{\Q}_{\ell}$ sheaves with constructible cohomology (\cite{Deligne_Weil_II},~Section 1.1.2-3,~\cite{Bhatt_Scholze},~Section 6.6).~When we say sheaves on $X$,~we mean objects in $D^b_c(X,\bar{\Q}_{\ell})$. For any two sheaves $\sF$ and $\sG$ on $X$,~by $\text{Hom}(\sF,~\sG)$,~we mean $\text{Hom}_{D^b_c(X,\bar{\Q}_{\ell})}(\sF,\sG)$. The constant sheaf with coefficients in $\bar{\Q}_{\ell}$ on $X$ will be denoted by $\bar{\Q}_{\ell}$.

The usual $t$-structure on $D^b_c(X,\bar{\Q}_{\ell})$ will be denoted by $(D^{\leq 0},D^{\geq 0})$. The perverse $t$-structure will be denoted by $(^pD^{\leq 0},~^pD^{\geq 0})$.~We use $^pH^i$ to denote the corresponding perverse cohomology functors.~For any variety $X/k$ (possibly non-normal) we denote by 

\begin{equation}\label{definition_intermediate_extension}
\text{IC}_{X} \colonequals j_{!*}(\bar{\Q}_{\ell}[\text{dim}(X)]),
\end{equation}

\noindent where $j \colon X^{\text{reg}} \hookrightarrow X$ is the inclusion of the regular locus of $X$.~Here $j_{!*}$ is the intermediate extension functor (\cite{BBD},~D\'efinition 1.4.22).

\subsubsection{} We are mostly interested in the triangulated versions of the sheaf operations,~so we will denote them without the usual decorations of `R or `L'. For example, the derived direct image functors will be denoted by $f_*$.~The only exception to this is the derived (local) internal $\text{Hom}$ functor,~which will be denoted by $\mathcal{RH}\text{om}$.

For a morphism of schemes $f \colon X \to Y$,~finite type over $k$,~we have adjoint pairs $(f_!,f^!)$ and $(f^*,f_*)$.~Moreover when $f$ is proper we have an adjoint triple $(f^*,f_!=f_*,f^!)$.~For an embedding $f:Y \hookrightarrow X$ and any $\sF \in D^{b}_{c}(X,\bar{\Q}_{\ell})$ we write $\sF|_Y$ instead of $f^*\sF$.~We also follow a similar convention for a morphism of sheaves.

For any two sheaves $\sF$ and $\sG$ on $X$,~and a morphism $u$ from $\sF$ to $\sG$, $f_* u$ will denote the induced morphism from $f_* \sF$ to $f_*\sG$.~We will also follow a similar convention for the other functors.

\subsubsection{}When $k$ is an algebraically closed field,~we identify $D^{b}_{c}(\Spec(k),\bar{\Q}_{\ell})$ with the bounded derived category of finite dimensional $\bar{\Q}_{\ell}$ vector spaces.

\subsubsection{} For any scheme $X$ of finite type over $k$,~let $\pi_X \colon X \to \Spec(k)$ denote the structural morphism.~Let $K_{X} \colonequals \pi_X^{!}\bar{\Q}_{\ell}$,~be the dualizing complex of $X$.~We denote by $\mathbb{D}_X$ the Verdier duality functor.~For a morphism of sheaves $u \colon \sF \to \sG$,~we denote by $u^{\vee}$ the induced morphism from $\mathbb{D}_X \sG$ to $\mathbb{D}_X \sF$.

\subsubsection{} For schemes $X_1$ and $X_2$ over $k$,~let $\text{pr}_1$ and $\text{pr}_2$ denote the projections from $X_1 \times_k X_2$ onto $X_1$ and $X_2$ respectively.~Given any morphism $c \colon C \to X_1 \times_k X_2$,~by $c_1$ and $c_2$ we mean $\text{pr}_1 \circ c$ and $\text{pr}_2 \circ c$ respectively.

Let $\sF_1$ and $\sF_2$ be sheaves on $X_1$ and $X_2$ respectively.~Denote by $\sF_1 \boxtimes \sF_2$ the object $\text{pr}_1^* \sF_1 \otimes \text{pr}_2^{*} \sF_2$ in $D^{b}_{c}(X_1 \times_{k} X_2, \bar{\Q}_{\ell})$.~There is a canonical isomorphism (\cite{Illusie},~(1.7.6) and (2.2.4))
 
 \begin{equation}\label{duality_distributes_external_products}
  \mathbb{D}_{X_1} \sF_1 \boxtimes  \mathbb{D}_{X_2}\sF_2 \simeq \mathbb{D}_{X_1 \times_k X_2}(\sF_1 \boxtimes \sF_2).
 \end{equation}

\subsubsection{}\label{base_change_section_2} Consider a Cartesian (upto nilpotents) diagram of schemes (finite type over $k$)

\begin{equation}\label{notatation_base_change}
\xymatrix{X' \ar[r]^-{g'} \ar[d]^-{f'} & X \ar[d]^-{f} \\
             Y'                        \ar[r]^-{g}          &  Y }.
\end{equation}

We shall make repeated use of the proper base change theorem (\cite{SGA5},~Expos\'e VI, 2.2.3) either in the form $g^*f_! \simeq f'_{!}g^{'*}$ or its Verdier dual $g^!f_* \simeq g^{'!}f^{'}_{*}$.~Moreover there is also a natural transformation of functors $f^{'!}g^* \to g^{'*}f^!$.~In any case such morphisms will be simply denoted by $(\text{BC})$.

\subsubsection{}\label{cup_products_compactly_supported}

Let $X/k$ be a finite type scheme,~and $j:X \hookrightarrow \bar{X}$ a compactification.~Let $\sF$ and $\sG$ be sheaves on $X$.~Then there exists a natural commutative diagram

\begin{equation}\label{cup_product_compactly_supported_sheaves}
\xymatrix{ j_!\sF \otimes j_* \sG \ar[r]^-{\simeq} \ar[d] & j_!(\sF \otimes \sG) \ar[d] \\
                j_*\sF \otimes j_* \sG \ar[r] & j_*(\sF \otimes \sG) }.
\end{equation}

Here the isomorphism along the top row comes from the projection formula (\cite{SGA4},~XVII,~(5.2.9)).~The map on the bottom row is immediate from the adjunction between $(j^*,j_*)$.~The vertical maps are the usual forget support maps.~Applying $\text{Hom}(\hspace{0.5cm},\bar{\Q}_{\ell})$~we get 

\begin{equation}\label{cup_product_compactly_supported_cohomology}
\xymatrix{ H^0_c(X,\sF) \otimes H^0(X,\sG) \ar[r] \ar[d] &H^0_c(X,\sF \otimes \sG) \ar[d] \\
                H^0(X,\sF) \otimes H^0(X, \sG) \ar[r] & H^0(X,\sF \otimes \sG) }.
\end{equation}

Here the vertical maps are forget support maps,~and the horizontal ones are cup product maps.

\subsubsection{}\label{exceptional_cup_product}

Let $f:X \to Y$ be a morphism of schemes finite type over $k$.~We denote by $t_{f^!} \colon f^!\sF \otimes f^* \sG \to f^!(\sF \otimes \sG)$,~the natural map adjoint to the composition $f_!(f^! \sF \otimes f^* \sG) \simeq f_!f^!\sF \otimes \sG \to \sF \otimes \sG$.~Here the first isomorphism is due to the projection formula,~and the second arrow is due to the adjoint pair $(f_!,f^!)$.

\subsubsection{} For a closed subscheme $Z$ of $X$,~let $\sI_Z$ denote its ideal sheaf.~By $Z_{red}$ we mean the reduced closed subscheme underlying $Z$.~By $Z_d,~d \geq 1$,~we mean the closed subscheme of $X$ defined by the ideal sheaf $\sI_{Z}^{d}$.~In particular $Z_{1}=Z$,~and $Z_{r}$ is a closed subscheme of $Z_{s}$,~whenever $r \leq s$.~For a morphism of schemes $f \colon Y \to X$,~by $f_{\text{red}}$ we mean the induced morphism from $Y_{\text{red}}$ to $X_{\text{red}}$.

\subsubsection{}Let $k_0$ be an arbitrary finite field.~Let $k$ be an algebraic closure of $k_0$.~Objects over $k_0$ will be denoted by a subscript $0$ (for example $X_0,~f_0,~\sF_0$,~etc.).~The corresponding object over $k$ will be denoted without a subscript,~for example $X,~f,~\sF$,~etc.~For a scheme $X/k$ defined over $k_0$,~we denote by $F_{X/k_0}$ the geometric Frobenius morphism of $X$ (with respect to $k_0$ and the chosen model of $X$ over $k_0$).~By $F^{0}_{X/k_0}$ we mean the identity map on $X$.

We shall denote the graphs of $F^{n}_{X/k_0}$ by $\Delta^{(n)}_X$ (or $\Delta^{(n)}$,~if there is no scope of confusion) considered as subschemes of $X \times_k X$.~$\Delta^{(0)}_X$ will be simply denoted by $\Delta_X$ (or $\Delta$).

\subsubsection{}\label{twist_correspondence_notation} Let $c \colon C \to X \times_k X$ be a morphism of schemes finite type over $k$.~Suppose $X$ is defined over $k_0$.~Then for any $n \geq 0$,~we denote by $c^{(n)} \colon C \to X \times_k X$ the morphism induced by $(F^n_{X/k_0} \circ c_1,c_2)$.

Let $\sF$ be a Weil sheaf on $X$,~that is a sheaf $\sF$ and an isomorphism 

\begin{equation}\label{Weil_Sheaf_Notations}
F_{\sF}:F^*_{X/k_0} \sF \simeq \sF.
\end{equation}

Let $F^{n*}_{\sF}$ denote the induced isomorphism from $F^{n*}_{X/k_0}\sF$ to $\sF$ obtained by iterating (\ref{Weil_Sheaf_Notations}).~For any sheaf $\sG$ on $X$ and any element $u$ in $\text{Hom}(c_1^* \sF,~c_2^{!}\sG)$,~we denote by $u^{(n)}$ the element in $\text{Hom}(c^{(n)*}_1 \sF,~c_2^{!}\sG)$ obtained as follows

\begin{equation}\label{cohomological_correspondence_twist_notation}
\xymatrix{u^{(n)}:c_1^{*}F^{n*}_{X/k_0} \sF \ar[r]^-{\cong}_-{c_1^*F^{n}_{\sF}} & c_1^* \sF \ar[r]^-{u} & c_2^! \sG}.
\end{equation}

\subsubsection{} For a scheme $X_0$ of finite type over $k_0$,~we also have the category of mixed sheaves $D^b_m(X_0,\bar{\Q}_{\ell})$ (\cite{BBD},~Section 5.1.6).~As in \cite{BBD} we denote by $D^b_{\leq w}(X_0,~\bar{\Q}_{\ell})$ the full subcategory of objects in $D^b_m(X_0,\bar{\Q}_{\ell})$  of weight less than or equal to $w$.~Note that $\text{IC}_X$ as defined in (\ref{definition_intermediate_extension}) is pure of weight $\text{dim}(X)$ (\cite{BBD},~Corollaire 5.3.2).

\section{Outline of the proof}

Our proof can be divided into three distinct steps,~which are logically independent of each other.~We describe these now.

\subsubsection*{1. Preliminary reductions} 
As a first step towards proving Theorem \ref{non_uniform_estimate},~we show that we can reduce to the situation of Theorem \ref{non_uniform_estimate},~$(3)$ (Lemma \ref{basic_reduction}).~In addition we can also assume that $c_2$ is generically \'etale (Proposition \ref{reduction_generically_etale}).~The idea here is simple,~we use the absolute Frobenius to get rid of generic inseparability.~Since the desired property is generic,~we work over function fields (Lemma \ref{field_frobenius_etale}),~and then spread it out.

%While doing so we need to be careful about what the procedure does to the \textit{only} global term $\text{Fix}(c^{(n)})$ occurring in (\ref{main_theorem_estimate}).~The absolute Frobenius being a universal homeomorphism ensures that this is not an issue.

Next, we would like to show that there exists a compactification of $c$ that leaves the boundary locally invariant (see Definition \ref{locally_invariant_closed}).~This will guarantee that the boundary becomes contracting (see Definition \ref{contracting}) after twisting the correspondence by a high enough power of the Frobenius.~This, in turn, allows us to use a result of Varshavsky on the decomposition of local terms in the form of Corollary \ref{shape_of_trace_formula}.~As shown by Varshavsky (\cite{Varshavsky_Hrushovski},~Section 2) the local invariance along the boundary can always be achieved at the cost of shrinking $X$.~In conclusion,~Lemma \ref{basic_reduction} and Proposition \ref{reduction_to_locally_invariant_boundary} allow us to assume that 

\begin{enumerate}[(a)]

\item $X$ is smooth and quasi-projective over $k$.

\item $c$ is proper and defined over $\F_q$.

\item $c_2$ is \'etale.

\item There exists a compactification of $c$ that leaves the boundary locally invariant over $\F_q$.

\end{enumerate}

\subsubsection*{2. Constructing a cohomological correspondence}

We intend to prove Theorem \ref{non_uniform_estimate} using the Lefschetz-Verdier trace formula (see Theorem \ref{Lefschetz-Verdier}).~The formula takes as input a proper correspondence and a cohomological correspondence lifting it,~and produces equality between the associated local and global terms.

In the setting of Theorem \ref{non_uniform_estimate} we do have a proper correspondence (constructed in Step 1),~$\bar{C} \to \bar{X} \times_k \bar{X}$ (denoted by $\bar{c}$).~Let $j:X \hookrightarrow \bar{X}$ be the open immersion.~On $X$ we have the following natural cohomological correspondence

\begin{equation}\label{overview_1}
\xymatrix{c_1^*\bar{\Q}_{\ell} \ar[r]^-{\cong} & \bar{\Q}_{\ell} \ar[r]^-{\cong} & c_2^{!}\bar{\Q}_{\ell}},
\end{equation}

\noindent lifting $c$.~The second isomorphism in (\ref{overview_1}) is a consequence of $c_2$ being \'etale.

However (\ref{overview_1}) does not extend in a natural way to a cohomological correspondence lifting $\bar{c}$.~In fact if either $c_1$ (or $c_2$) were proper,~(\ref{overview_1}) can be extended to a cohomological correspondence of $j_!\bar{\Q}_{\ell}$ (or $j_*\bar{\Q}_{\ell}$),~but as mentioned earlier we cannot ensure this in general.~What does happen in general,~is that (\ref{overview_1}) can be extended to a morphism 

\begin{equation}\label{overview_2}
\xymatrix{\bar{c}_1^*j_!\bar{\Q}_{\ell} \to  \bar{c}_2^{!}j_{*}\bar{\Q}_{\ell}}.
\end{equation}

This \textit{duality} motivated us to consider the intermediate extension $\text{IC}_{\bar{X}}$.~Constructing pullback maps for the intersection complex has been considered by various authors (\cite{Gabber_pull_back}, \cite{Weber_1}, \cite{Weber_2}, \cite{Hanamura_Saito}). However none of these constructions are functorial for general maps.~For our purposes this lack of functoriality is not an issue.~Using the results obtained in Appendix \ref{cohomological_IC},~we construct a cohomological correspondence (Corollary \ref{cohomological_correspondence_lifting_correspondence})

\begin{equation}\label{overview_3}
u:\xymatrix{\bar{c}_1^*\text{IC}_{\bar{X}} \to  \bar{c}_2^{!}\text{IC}_{\bar{X}}},
\end{equation}

\noindent which restricts to (\ref{overview_1}) on $X$ (upto shift by $d:=\text{dim}(X)$).~Moreover our construction of (\ref{overview_3}) is such that the linear map $\bar{c}_{2*}\bar{c}_1^*:H^{d}(\bar{X},~\text{IC}_{\bar{X}}) \to H^{d}(\bar{X},~\text{IC}_{\bar{X}})$ induced by (\ref{overview_3}) is multiplication by the generic degree of $c_1$ (Proposition \ref{pushforward_pullback_intersection_cohomology}).~This can be seen as a manifestation of $\bar{c}_1$ and $\bar{c_2}$ being dominant,~and hence their images intersecting the smooth locus of $\bar{X}$.

Purity of $\text{IC}_{\bar{X}}$ then ensures that the global term associated to (\ref{overview_3}),~produces the \textit{correct} leading and error terms in the estimate (\ref{main_theorem_estimate}) (Proposition \ref{global_term_correspondence_intersection_complex})  .

Finally using Varshavsky's result on decomposition of local terms for a contracting correspondence (in the form of Corollary \ref{shape_of_trace_formula}),~we reduce the problem to computing the \textit{naive local terms} of $u^{(n)}$ on $X$,~and the local terms $\text{LT}(u^{(n)}|_Z)$,~where $Z$ is the closed complement of $X$ in $\bar{X}$.~Since $u|_X$ is the cohomological correspondence (\ref{overview_1}) up to shift,~and $c_2$ is \'etale,~the naive local terms are easily shown to be $1$ (Lemma \ref{naive_local_term_calculate}).

Thus it remains to show that $|\text{LT}(u^{(n)}|_Z)|$ asymptotically grows no faster than the error term,~which is of the order $q^{n(d-\frac{1}{2})}$.~Note that $u^{(n)}|_Z$ is a cohomological correspondence of $\text{IC}_{\bar{X}}|_Z$ lifting a possibly \textit{non-proper} correspondence (see Section \ref{local_term_subscheme_invariant_neighbourhood_fixed_points}).~The rest of the article is devoted to the proof of Theorem \ref{bound_local_term_invariant_subscheme_introduction}.

\subsubsection*{3. Local terms along a locally invariant subset}

We observe that (see Diagram \ref{Cartesian_locally_invariant}) local invariance of a closed subset $Z$ over $\F_q$,~gives rise to a correspondence $c \colon C \to Z \times_k Z$,~and a diagram of the form

\begin{equation}\label{essentially-proper_intro}
\xymatrix{ C   \ar[d]_-{\bar{c}_U} \ar[dr]^-{c} &  \\
                  U   \ar@{^{(}->}[r]^-{j}               & Z \times_k Z             }
\end{equation}

\noindent such that

\begin{enumerate}[(a)]

\item $Z$ is proper.

\item $\bar{c}_U$ is proper.

\item $U=\cup_{i=1}^{r} U_i \times_k U_i$,~where $U_i$'s are open subsets of $Z$ defined over $\F_q$ and cover $Z$.

\end{enumerate}

Such correspondences $c:C \to Z \times_k Z$ are defined to be \textit{essentially proper} over $\F_q$ (Definition \ref{good_correspondences}).~It is easy to see that the non-proper correspondences which come from locally invariant subsets are essentially proper (Lemma \ref{local_invaraince_good}).

%\begin{rmk}\label{the_set_up_overview_1}
%The $X$ in (\ref{essentially-proper_intro}) corresponds to $Z$ in our situation.~It is not to be confused with the open variety $X$ occurring in Theorem \ref{non_uniform_estimate}.
%\end{rmk}

Note that (b) implies $U$ contains $\Delta^{(n)}$ as a closed subscheme for all $n \geq 0$,~and is stable under the \textit{partial Frobenius},~$F_l \colonequals F_{Z/\F_q} \times 1_Z$.~Since $\Delta^{(n)}$ are closed subschemes of $U$,~$\text{Fix}(c^{(n)})$ is necessarily proper over $k$.~Thus given a cohomological correspondence $u$ of a Weil sheaf $\sF$ on $X$ lifting $c$,~we can make sense of $\text{LT}(u^{(n)})$.~The goal of Section \ref{local_terms_invariant_subset} is to prove Theorem \ref{local_term_growth_good_correspondence} on the growth of these local terms with respect to $n$,~which implies Theorem \ref{bound_local_term_invariant_subscheme_introduction},~and hence Theorem \ref{non_uniform_estimate} (see Section \ref{enough_to_bound_local}).

In general, there are two equivalent ways to compute local terms.~One via Varshavsky's recipe (see Section \ref{trace_maps}),~and the other via a pairing defined by Illusie.~In any case one constructs a trace map

\begin{equation}\label{overview_4}
\mathcal{T}\textit{r}_c \colon \text{Hom}(c_{2!}c^{*}_1\sF,\sF) \to H^{0}(\text{Fix}(c),K_{\text{Fix}(c)}).
\end{equation}

The local terms are then obtained by composing (\ref{overview_4}) with the adjunction $H^{0}(\text{Fix}(c),K_{\text{Fix}(c)}) \to \bar{\Q}_{\ell}$,~provided $\text{Fix}(c)$ is proper over $k$.

The local terms of $u^{(n)}$ are obtained by twisting the correspondence and applying (\ref{overview_4}),~whose target now becomes $H^{0}(\text{Fix}(c^{(n)}),K_{\text{Fix}(c^{(n)})})$. Since there are no natural maps in general between $\text{Fix}(c)$ and $\text{Fix}(c^{(n)})$,~it is not possible to \textit{compare} these maps.~To be able to compare these maps we need a common `target' for them. 

A source of this problem is that the trace map (and hence the Lefschetz-Verdier trace formula) is adapted to the diagonal.~This is unlike the Lefschetz trace formula for smooth projective varieties which allows for the intersection between arbitrary cohomology classes in the right degree.~In Section \ref{the_pairing} we give an alternate description of the trace map which has this additional flexibility.~In the notation of (\ref{essentially-proper_intro}),~we define a pairing 

\begin{equation}\label{main_pairing_outline}
\Phi:\text{Hom}(c_1^*\sF,c_2^!\sF) \otimes_{\bar{\Q}_{\ell}} H^{0}_c(U,~j^*(\sF \boxtimes \mathbb{D}_Z \sF)) \to H^{0}(Z \times_k Z,K_{Z \times_k Z}),
\end{equation}

\noindent and hence for any cohomological correspondence $u$,~ a linear functional 

\begin{equation}\label{linear_functional_outline}
\Phi_u(\beta):=\text{Tr}_{Z \times_k Z}(\Phi(u \otimes \beta)),
\end{equation}

\noindent on $H^{0}_c(U,~j^*(\sF \boxtimes \mathbb{D}_Z \sF))$.~Here $\text{Tr}_{Z \times_k Z}$ is the natural trace map on $H^{0}(Z \times_k Z,K_{Z \times_k Z})$.

Further using the fact that $\Delta^{(n)}$ are closed subschemes of $U$,~the Weil sheaf $\sF$ and the evaluation map (\ref{evaluation_trace}),~we define cohomology classes $[\Delta^{(n)}]$ in $H^{0}_c(U,~j^*(\sF \boxtimes \mathbb{D}_Z \sF))$ which satisfy (see Proposition \ref{local_terms_using_functional})

\begin{equation}\label{local_term_computed_globally}
\Phi_u([\Delta^{(n)}])=\text{LT}(u^{(n)}).
\end{equation}

The equality (\ref{local_term_computed_globally}) can be seen as a trace formula in this non-proper setting,~ where the object on the left is thought of as a \textit{global} term.

Recall that the partial Frobenius $F_l$ acts on $U$,~and since $\sF$ has the structure of a Weil sheaf,~it acts on $H^{0}_c(U,~j^*(\sF \boxtimes \mathbb{D}_Z \sF))$ too.~It essentially follows from the definition of $[\Delta^{(n)}]$ that (see Lemma \ref{Frobenius_pullback_cohomology_classes})

\begin{equation}\label{Frobenius_pullback_cohomology_classes_overview}
(F_l^{*})^n([\Delta])=[\Delta^{(n)}],
\end{equation}

\noindent for any $n \geq 0$.

When $\sF$ comes from a mixed sheaf $\sF_0 \in D^b_{\leq w}(Z_0,\bar{\Q}_{\ell})$ (on the chosen model $Z_0$ of $Z$),~the action of $F^*_l$ on $H^{0}_c(U,~j^*(\sF \boxtimes \mathbb{D}_Z \sF))$ is easy to understand using the \textit{almost product} structure of $U$ .~The key point here is that as far as the partial Frobenius is concerned,~$\sF_0 \boxtimes \mathbb{D}_{Z_0} \sF_0$ behaves as a mixed sheaf of weight less than or equal to $w$,~on a variety of $\text{dim}(Z)=\frac{\text{dim}(U)}{2}$.~Thus the weights of $F_l^*$ on $H^{0}_c(U,~j^*(\sF \boxtimes \mathbb{D}_Z \sF))$ are bounded above by $w+a+\text{dim}(Z)$ (see Lemma \ref{partial_weights_Frobenius}).

Having bounded the weights of the partial Frobenius,~Theorem \ref{local_term_growth_good_correspondence} follows immediately by combining (\ref{local_term_computed_globally}) and (\ref{Frobenius_pullback_cohomology_classes_overview}) using a linear algebra argument (Lemma \ref{linear_algebra_bound}).

\section{Varshavsky's trace formula}

In this section we recall the formalism behind Varshavsky's trace formula (\cite{Varshavsky}.)~Since we will be using the Lefschetz-Verdier trace formula (\cite{Illusie},~Corollary 4.7) we also recall the same.

\subsection{Correspondences and cohomological correspondences}In this section $k$ will denote an arbitrary algebraically closed field.

\begin{defn}[Correspondence]\label{correspondence}

A \textit{correspondence} from a scheme $X_1$ to $X_2$ is a morphism of schemes $c \colon C \to X_1 \times_k X_2$.~We will denote this by $[c]=(C,c_1,c_2)$.~When there is no risk of confusion we will also denote this simply by $c$.

\end{defn}

\subsubsection{}\label{the_trivial_correspondence}

The natural isomorphism $c_{tr}:\Spec(k) \to \Spec(k) \times_k \Spec(k)$ is a self-correspondence of $\Spec(k)$,~denoted by $[c_{tr}]=(\Spec(k),1_{\Spec(k)},1_{\Spec(k)})$.

\begin{defn}[Morphism of correspondences]\label{morphism_correspondences}

Let $[c]=(C,c_1,c_2)$ be a correspondence from $X_1$ to $X_2$ and let $[b]=(B,b_1,b_2)$ be a correspondence from $Y_1$ to $Y_2$.~A \textit{morphism of $[c]$ to $[b]$} consists of a triple of morphisms $[f] \colonequals (f_1,f^{\#},f_2)$ which fit into a commutative diagram \begin{center}

$\xymatrix{ X_1 \ar[d]^{f_1} & C \ar[l]_{c_1} \ar[r]^{c_2} \ar[d]^{f^{\#}}  & X_2 \ar[d]^{f_2} \\
                   Y_1                   & B \ar[l]^{b_1} \ar[r]_{b_2}                       &  Y_2                }$.

\end{center}

\end{defn}

\subsubsection{}\label{morphism_trivial_correspondence}
Let $c:C \to X_1 \times_k X_2$ be a correspondence from $X_1$ to $X_2$.~Then $[\pi]_c:=(\pi_{X_1}, \pi_C, \pi_{X_2})$ is a morphism from $[c]$ to $[c_{tr}]$ called the structural morphism of $[c]$.

\subsubsection{}\label{morphism_type_correspondence} We say a morphism of correspondences $[f]=(f_1, f^{\#}, f_2)$ is \textit{proper} (respectively an \textit{open immersion},~respectively a \textit{closed immersion}) if each of the $f_1$,~$f^{\#}$ and $f_2$ is proper (respectively an open immersion,~respectively a closed immersion).~We say a correspondence $[c]$ is proper,~if $C$,~$X_1$ and $X_2$ are proper are proper over $k$.

\begin{defn}[Compactification of correspondences]\label{compactification_correspondence}

A \textit{compactification of a correspondence} $c:C \to X_1 \times_k X_2$,~is an open immersion $[j]=(j_1, j^{\#}, j_2)$ of $[c]$ into a correspondence $\bar{c}: \bar{C} \to \bar{X}_1 \times_k \bar{X}_2$,~such that $[\bar{c}]$ is proper and $j_1, j^{\#}, j_2$ are dominant.

\end{defn}

We have the following Lemma.

\begin{lem}\label{compactification_Cartesian}

Let 
\begin{equation}\label{commutative_Cartesian}
\xymatrix{ C \ar@{^{(}->}[r]^-{j_C} \ar[d]^-{f} & \bar{C} \ar[d]^-{\bar{f}} \\
                 X \ar@{^{(}->}[r]^-{j} & \bar{X}}
\end{equation}

\noindent be a commutative diagram such that $f$ is a proper morphism.~Suppose that $j$ and $j_C$ are open immersions.~If $j_C$ has a dense image,~then (\ref{commutative_Cartesian}) is necessarily Cartesian.

\end{lem}

\begin{proof}

Let $j'$ be the induced morphism from $C$ to $\bar{f}^{-1}(X)$.~Since $j_C$ is a dense open immersion,~so is $j'$.~Moreover since $f$ is proper (and our schemes are assumed to be separated),~$j'$ is also proper \cite[Corollaire 5.4.3]{EGAII} and hence is an isomorphism.~Thus (\ref{commutative_Cartesian}) is necessarily Cartesian.

\end{proof}

The following corollary is an immediate consequence of Lemma \ref{compactification_Cartesian} and will be used later. 

\begin{cor}\label{compactification_effect_open_part}
Let $\bar{c} \colon  \bar{C} \to \bar{X} \times_k \bar{X}$ be a compactification of a correspondence  $c \colon C \to X \times_k X$.~If $c$ is proper then $\bar{c}^{-1}(X \times_k X)=C$.
\end{cor}

\begin{defn}[Restriction of a correspondence to an open subscheme]\label{restriction_correspondence_open}

Let $[c]=(C,c_1,c_2)$ be a correspondence from $X$ to itself.~Let $U \subseteq X$ be an open subscheme.~Then the \textit{restriction of $c$ to $U$} is the correspondence, 
$[c]|_U \colonequals (c_1^{-1}(U) \cap c_2^{-1}(U),c_1|_{c_1^{-1}(U) \cap c_2^{-1}(U)},c_2|_{c_1^{-1}(U) \cap c_2^{-1}(U)})$ from $U$ to itself.~We shall also denote this correspondence by $c|_U$.

Similarly if $W \subseteq C$ is an open subscheme of $C$,~the restriction of $c$ to $W$ is the correspondence $[c]|_W \colonequals (W,c_1|_W,c_2|_W)$.~As before $c|_W$ shall denote the induced morphism from $W$ to $X \times_k X$,~and also the correspondence $[c]|_W$.
\end{defn}

\begin{defn}[Cohomological correspondence]\label{cohomological_correspondence}

Let $[c]=(C, c_1, c_2)$ be a correspondence from $X_1$ to $X_2$.~Let $\sF_1$ and $\sF_2$ be sheaves on $X_1$ and $X_2$ respectively.~A \textit{cohomological correspondence} from $\sF_1$ to $\sF_2$ lifting $c$ is an element of $\text{Hom}(c_{2!}c_1^*\sF_1,\sF_2)$.

\end{defn}

\subsubsection{Restriction of cohomological correspondence to an open subscheme}\label{restriction_cohomological_open}

Let $c \colon C \to X \times_k X$ be a self-correspondence of $X$.~Let $C^{0} \subseteq C$ and $X^{0}_i \subseteq X,~i=1,2$ be open subschemes.~Suppose that $c$ induces a correspondence $c^0 \colon C^0 \to X_{1}^{0} \times_{k} X^{0}_2$.~Thus we have a commutative diagram

\begin{center}

$\xymatrixcolsep{4pc} \xymatrix{ C^0 \ar@{^{(}->}[r] \ar[d]^-{c_2^0} & C \ar[d]^-{c_2} \\
                   X_2^0 \ar@{^{(}->}[r]                 & X_2                  }$.

\end{center}

For any sheaf $\sF$ on $C$,~there exists a natural adjunction morphism $\sF \to c_2^{!}c_{2!}\sF$.~Restricting the above morphism to $C^0$,~and using the adjunction between $c^0_{2!}$ and $c^{0!}_2$,~we get a morphism 

\begin{equation}\label{estriction_cohomological_open_1}
\text{BC}(\sF) \colon c^{0}_{2!}(\sF|_{C^0}) \to (c_{2!}\sF)|_{X^{0}_2}.
\end{equation}

Let $\sF_1$ and $\sF_2$ be sheaves on $X_1$ and $X_2$ respectively.~Let $u \in \text{Hom}(c_{2!}c^*_1\sF_1,\sF_2)$ be a cohomological correspondence from $\sF_1$ to $\sF_2$ lifting $c$.~Then we can \textit{restrict $u$} to give a cohomological correspondence from $\sF_1|_{X^0_1}$ to $\sF_2|_{X^0_2}$ lifting $c_0$ as follows,

\begin{center}

$\xymatrixcolsep{4pc} \xymatrix{ u^0 \colon c^{0}_{2!}c^{0*}_{1}(\sF_1|_{X^0_1}) \simeq c^0_{2!}(c^{*}_1\sF_1|_{C^0}) \ar[r]^-{\text{BC}(c_1^*\sF_1)} & (c_{2!}c^*_{1}\sF_1)|_{X^0_2} \ar[r]^-{u|_{X_2^0}} &  \sF_{2}|_{X^0_2}}$.

\end{center}

In particular,~for any open subscheme $U \subseteq X$ we have a cohomological correspondence $u|_U$ lifting $c|_U$ (see Definition \ref{restriction_correspondence_open}).

\subsubsection{Action of a correspondence on cohomology}\label{action_correspondence_cohomology}

Let $c \colon C \to X_1 \times_{k} X_2$ be a correspondence.~Let $u$ be a cohomological correspondence from $\sF_1$ to $\sF_2$ lifting $c$.~Suppose that $c_1$ is a proper morphism.~Consider the following sequence of morphisms

\begin{center}
$\xymatrix{\pi_{X_1!}\sF_1 \ar[r]^-{(A)} & \pi_{X_1!}c_{1*}c_1^* \sF_1 \simeq \pi_{C!}c_1^* \sF_1 \ar[r]^-{\pi_{C!}u} & \pi_{C!}c_2^{!}\sF_2 \simeq \pi_{X_2!}c_{2!}c_2^{!}\sF \ar[r]^-{(A)} & \pi_{X_2!}\sF_2}$.
\end{center}

Here the second isomorphism follows from the properness of $c_1$.~The morphisms $(A)$ are adjunctions.~In particular when $X_1=X_2$ and $\sF_1=\sF_2=\sF$,~we get an endomorphism $\text{R}\Gamma_c(u)$ of the perfect complex $\text{R}\Gamma_c(X,\sF)$.

Assume that there exists an open subscheme $X_1^{0} \subseteq X_1$,~such that $\sF_1$ is supported on $X_1^{0}$.~Suppose that $c_1|_{c_1^{-1}(X_1^{0})} \colon C^{0} \colonequals c_1^{-1}(X_1^{0}) \to X_1^{0}$ is proper.~Let $c_1^0$ and $c_2^0$ be the induced morphism from $C^0$ to $X^0_1$ and $X_2$ respectively.

Note that we have isomorphisms 

\begin{equation}\label{action_correspondence_cohomology_1}
R\Gamma_c(X_1,\sF_1) \simeq R\Gamma_c(X_1^0,\sF_1|_{X^{0}_1})=\pi_{X^{0}_1!}(\sF_1|_{X_0})
\end{equation}

\noindent and

\begin{equation}\label{action_correspondence_cohomology_2}
 \pi_{C^{0}!}c_{2}^{0!}(\sF_2) \simeq \pi_{X_2!}c^{0}_{2!}c^{0!}_{2}(\sF_{2}).
\end{equation}

Since $c_1^{0}$ is proper one also has  

\begin{equation}\label{action_correspondence_cohomology_3}
\pi_{X^{0}_1!}c^{0}_{1*}c^{0*}_1(\sF_1|_{X_0}) \simeq \pi_{C^{0}!} c^{0*}_1(\sF_1|_{X^{0}_1}).
\end{equation}

Further, there are morphisms induced by adjunction 

\begin{equation}\label{action_correspondence_cohomology_4}
\pi_{X^{0}_1!}(\sF_1|_{X_0}) \to c^{0}_{1*}c^{0*}_1\pi_{X^0_1!}(\sF_1|_{X_0}),
\end{equation}

\noindent and

\begin{equation}\label{action_correspondence_cohomology_5}
\pi_{X_2!}c^{0}_{2!}c^{0!}_{2}(\sF_{2}) \to \pi_{X_2!}(\sF_2).
\end{equation}

Further we have a correspondence $[c^0] \colonequals (C^0, c_1|_{C^0}, c_2|_{C^0})$ between $X_1^0$ and $X_2$,~and a cohomological correspondence $u^0$ between $\sF_1|_{X^{0}_1}$ and $\sF_2$ lifting $[c^0]$ (see Section \ref{restriction_cohomological_open}).~Applying $\pi_{C^{0}!}$ to $u^{0}$,~and using (\ref{action_correspondence_cohomology_3}) and (\ref{action_correspondence_cohomology_2}) we get a morphism 

\begin{equation}\label{action_correspondence_cohomology_6}
\pi_{X^{0}_1!}c^{0}_{1*}c^{0*}_1(\sF_1|_{X_0}) \to \pi_{X_2!}c^{0}_{2!}c^{0!}_{2}(\sF_{2}).
\end{equation}

Combining (\ref{action_correspondence_cohomology_1}),~(\ref{action_correspondence_cohomology_4}),~(\ref{action_correspondence_cohomology_5}) and (\ref{action_correspondence_cohomology_6}) we get a morphism $R\Gamma_c(u) \colon R\Gamma_c(X_1,\sF_1) \to R\Gamma_c(X_2,\sF_2)$. In particular if $X_1=X_2=X$ and $\sF_1=\sF_2=\sF$,~we get an endomorphism $R\Gamma_c(u)$ of the perfect complex $R\Gamma_c(X,\sF)$.

\subsection{The Lefschetz-Verdier trace formula}

In this section, we describe a recipe to obtain the local and global terms in the Lefschetz-Verdier trace formula.~Let $k$ be an arbitrary algebraically closed field.

\subsubsection{Scheme of fixed points}\label{scheme_fixed_points}

Let $c \colon C \to X \times_{k} X$ be a correspondence.~The \textit{scheme of fixed points} of the correspondence $c$ is the closed subscheme $\text{Fix}(c) \colonequals C \times_{X \times_k X} X$ of $C$.~Here $X$ is looked at as a scheme over $X \times_k X$ via the diagonal embedding $\Delta$.

Let $\Delta'$ denote the embedding of $\text{Fix}(c)$ inside $C$.~Let $c'$ be the restriction of $c$ to $\text{Fix}(c)$.~Thus we have a Cartesian diagram

\begin{equation}\label{scheme_fixed_points_1}
\xymatrixcolsep{4pc} \xymatrix{ \text{Fix}(c) \ar[r]^-{c'}  \ar@{^{(}->}[d]^-{\Delta'} & X  \ar@{^{(}->}[d]^-{\Delta} \\
                   C \ar[r]^-{c}                      & X \times X                  }.
\end{equation}

\subsubsection{}\label{local_trace_quasi_finite}

Let $c:C \to X \times_{k} X$ be a correspondence from $X$ to itself.~Let $u$ be a cohomological correspondence of a sheaf $\sF$ to itself lifting $c$.~Further assume that $c_2$ is quasi-finite.~Proper base change implies that for any closed point $x \in X$,~the stalk at $x$ of $c_{2!}c_1^{*}\sF$ is isomorphic to $\oplus_{c_2(y)=x} \sF_{c_{1}(y)}$. 

Hence $u|_{x}$ induces a morphism $\oplus_{c_2(y)=x} \sF_{c_{1}(y)} \to \sF_{x}$.~In particular for any closed point $y \in \text{Fix}(c)$ we have an induced endomorphism (denoted by $u_y$) of $\sF_{c'(y)}$.~Here $c'$ is the map induced from $\text{Fix}(c)$ to $X$ (see Diagram \ref{scheme_fixed_points_1}).

\begin{defn}[Naive local term] \label{Naive_local_trace}

Using the assumptions and notations in Section \ref{local_trace_quasi_finite},~for any closed point $y \in \text{Fix}(c)$,~we define the \textit{naive local term} at $y$ to be the trace of the endomorphism $u_y$.

\end{defn}

The Lefschetz-Verdier trace formula can be viewed as a consequence of the commutativity of certain \textit{trace maps} with proper push forward.~Now we describe these trace maps.

\subsection{Trace maps}\label{trace_maps}\cite[Section 1.2]{Varshavsky}

Let $c \colon C \to X \times_{k} X$ be a correspondence from $X$ to itself.~Let $\sF$ be a sheaf on $X$.~Let $b \colon B \to X \times_{k} X$ be another correspondence.~Consider the Cartesian diagram

\begin{equation}\label{scheme_fixed_points_2}
\xymatrixcolsep{4pc} \xymatrix{ A \ar[r]^-{c'}  \ar[d]^-{b'} & B  \ar[d]^-{b} \\
                   C \ar[r]^-{c}                      & X \times X                  }.
\end{equation}

Further, assume that we have a map 

\begin{equation}\label{evaluation_trace_general}
\text{ev}_{\sF,b} \colon \mathbb{D}_X\sF \boxtimes \sF \to b_* K_B.
\end{equation}

One has the natural evaluation map $\mathbb{D}_X\sF \otimes \sF \to K_X$.~Since $\Delta^*(\mathbb{D}_X\sF \boxtimes \sF) \simeq \mathbb{D}_X\sF \otimes \sF$,~by adjunction one gets a morphism 

\begin{equation}\label{evaluation_trace}
\text{ev}_{\sF} \colon \mathbb{D}_X\sF \boxtimes \sF \to \Delta_* K_X.
\end{equation}

Base change applied to the Cartesian diagram \ref{scheme_fixed_points_2} implies

\begin{equation}\label{trace_maps_1}
\xymatrix{c^!(b_*K_B) \ar[r]^-{\cong}_-{(\text{BC})} & b'_* c^{'!}K_B \ar[r]^-{\cong} & b'_*K_{A}}.
\end{equation}

Thus combining (\ref{evaluation_trace}) and (\ref{trace_maps_1}) we get a morphism

\begin{equation}\label{trace_maps_2}
\xymatrix{c^{!}(\mathbb{D}_X\sF \boxtimes \sF) \ar[r]^-{c^!\text{ev}_{\sF,b}} & c^!b_*K_B \ar[r]^{(\ref{trace_maps_1})} & b'_*K_{A}}.
\end{equation}

In \cite{Illusie} (see Sections 3.1.1 and 3.2.1),~Illusie obtained a canonical isomorphism

\begin{equation}\label{illusie_isomorphism}
\mathcal{RH}\textit{om}(c_1^*\sF,c_2^!\sF) \simeq c^!(\mathbb{D}_X\sF \boxtimes \sF).
\end{equation}

\noindent Combining (\ref{illusie_isomorphism}) and (\ref{trace_maps_2}) we get a morphism

\begin{equation}\label{trace}
\underline{\mathcal{T \textit{r}}}_{B}:\mathcal{RH}\textit{om}(c_1^*\sF,c_2^!\sF) \to b'_*K_{A}.
\end{equation}

\begin{ex}
The simplest case where one can use the above formalism is when $B=X$ and $b$ is the diagonal morphism.~indeed in that case one has the natural evaluation map $\mathbb{D}_X\sF \otimes \sF \to K_X$.~Since $\Delta^*(\mathbb{D}_X\sF \boxtimes \sF) \simeq \mathbb{D}_X\sF \otimes \sF$,~by adjunction one gets a morphism 

\begin{equation}\label{evaluation_trace}
\text{ev}_{\sF} \colon \mathbb{D}_X\sF \boxtimes \sF \to \Delta_* K_X.
\end{equation}

Note that in this case $A$ is $\text{Fix}(c)$.~Moreover applying $H^{0}(C,\hspace{0.4cm})$ to (\ref{trace}) one obtains the \textit{Trace map} 

\begin{equation}\label{Varashavsky_Trace}
\mathcal{T}\textit{r}_c:\text{Hom}(c_{2!}c^{*}_1\sF,\sF) \to H^{0}(\text{Fix}(c),K_{\text{Fix}(c)}).
\end{equation}

For an open subset $\beta$ of $\text{Fix}(c)$,~let $j_{\beta}$ denote the inclusion of $\beta$ into $\text{Fix}(c)$.~The natural adjunction morphism $K_{\text{Fix}(c)} \to j_{\beta*}j_{\beta}^{*}K_{\text{Fix}(c)}$ induces a morphism 

\begin{equation}\label{trace_maps_3}
\text{Res}_{\beta}:H^{0}(\text{Fix}(c),K_{\text{Fix}(c)}) \to H^{0}(\beta,K_{\beta}).
\end{equation}

%\begin{equation}\label{trace_maps_4}

%\end{equation}

Let $\mathcal{T}\textit{r}_{\beta} \colonequals \text{Res}_{\beta} \circ \mathcal{T}\textit{r}_c$.~If $\beta$ is proper over $k$,~then the adjunction $\pi_{\beta!}\pi^{!}_{\beta} \bar{\Q}_{\ell} \to \bar{\Q}_\ell$ gives rise to a morphism $\pi_{\beta!}:H^{0}(\beta,K_{\beta}) \to \bar{\Q}_{\ell}$.~Thus we get a morphism

\begin{equation}\label{local_term_trace}
\text{LT}_{\beta} \colonequals \pi_{\beta!} \circ \mathcal{T}\textit{r}_{\beta}:\text{Hom}(c_{2!}c^{*}_1\sF,\sF) \to \bar{\Q}_{\ell}.
\end{equation}

In particular if $\text{Fix}(c)$ is proper over $k$ we get a morphism

\begin{equation}\label{local_term_proper}
\text{LT}: \text{Hom}(c_{2!}c^{*}_1\sF,\sF) \to \bar{\Q}_{\ell}.
\end{equation}

\end{ex}

\begin{defn}[Local term]\label{Local_term}

For any proper connected component $\beta$ of $\text{Fix}(c)$,~and any cohomological correspondence $u$ lifting $c$,~the \textit{local term} at $\beta$ is defined to be $\text{LT}_{\beta}(u)$.~Moreover if $\text{Fix}(c)$ is proper over $k$ we define the local term of $u$ to be $\text{LT}(u)$.

Clearly when $\text{Fix}(c)$ is proper

\begin{equation}\label{Local_term_1}
\text{LT}(u)=\sum_{\beta \in \pi_0(\text{Fix}(c))} \text{LT}_{\beta}(u).
\end{equation}

\end{defn}

\begin{rmk}

Our definition of a local term is the one in \cite{Varshavsky},~Section 1.2.~It is compatible with the definition in \cite{Illusie},~Section 4.2.5 (see \cite{Varshavsky},~Appendix A).

\end{rmk}

\begin{ex}
The following example will be relevant to us in the proof of Theorem \ref{Hrushovski_Lang_Weil}. Let $k$ be an algebraic closure of $\F_q$.~Let $c:C \to X \times_k X$ be a correspondence defined over $\F_q$.~Let $\sF$ be a Weil sheaf on $X$.~Hence $\sF$ comes equipped with an isomorphism 

\begin{equation}\label{Weil_Sheaf}
F_{\sF} \colon F^*_{X/\F_q} \sF \simeq \sF.
\end{equation}

Dualizing (\ref{Weil_Sheaf}) we get an isomorphism

\begin{equation}\label{dual_Weil_structure}
F_{\sF}^{\vee} \colon \mathbb{D}_X \sF \simeq F^!_{X/\F_q} \mathbb{D}_X \sF.
\end{equation}

For each $n$ we have a Cartesian diagram

\begin{equation}\label{graph_Frobenius_local_terms_Appendix}
\xymatrix{\text{Fix}(c^{(n)}) \ar[r] \ar@{^{(}->}[d]_-{\Delta^{(n)'}} \ar@/^2pc/[rr]^-{c^{(n)'}} & X \ar@{^{(}->}[d]_-{\Delta^{(n)}} \ar[r]^{F^n_{X/\F_q}} & X \ar@{^{(}->}[d]^-{\Delta} \\
                  C \ar[r]^-{c} \ar@/_2pc/[rr]_-{c^{(n)}} & X \times_k X                       \ar[r]^-{F^n_{X/\F_q} \times 1_X}         &  X \times_k X }.
\end{equation}

By definition,~the composition of the arrows in the lower row is $c^{(n)} \colon C \to X \times_k X$ given by $(F^n_{X/\F_q} \circ c_1,c_2)$.~Denote the composition of the arrows in the topmost row by $c^{(n)'}$.
Moreover, since $!$-pullback commutes with external products (\cite{Illusie},~(1.7.3)) we have an isomorphism

\begin{equation}\label{Frobenius_cohomology_class_1}
(F^n_{X/\F_q} \times 1_X)^{!}(\mathbb{D}_X \sF \boxtimes \sF) \simeq F^{n!}_{X/\F_q} \mathbb{D}_X \sF \boxtimes \sF \simeq \mathbb{D}_X \sF \boxtimes \sF,
\end{equation}

\noindent induced by (\ref{dual_Weil_structure}).~Further by applying the functor $(F^{n}_{X/\F_q} \times 1_X)^{!}$ to the evaluation map (\ref{evaluation_trace}),~and using base change along the right Cartesian square in (\ref{graph_Frobenius_local_terms}) we get a morphism

\begin{equation}\label{Frobenius_cohomology_class_2}
(F^n_{X/\F_q} \times 1_X)^{!}(\mathbb{D}_X \sF \boxtimes  \sF) \to \Delta^{(n)}_{*}K_{X}.
\end{equation}

Thus combining (\ref{Frobenius_cohomology_class_1}) and (\ref{Frobenius_cohomology_class_2}) we get for each $n \geq 0$,~morphism

\begin{equation}\label{Frobenius_cohomology_class_3}
\text{ev}^{(n)}_{\sF} \colon \mathbb{D}_X \sF \boxtimes \sF \to \Delta^{(n)}_{*}K_{X}.
\end{equation}

Thus the above formalism gives rise to

 \begin{equation}\label{trace_Along_graph_Frobenius_Appendix}
\mathcal{T}\textit{r}^{'(n)}_c \colon \text{Hom}(c_{2!}c_1^*\sF,\sF) \to H^0(\text{Fix}(c^{(n)}),K_{\text{Fix}(c^{(n)}})
\end{equation}

\end{ex}

Now we are in a position to state the Lefschetz-Verdier trace formula.

\begin{thm}\label{Lefschetz-Verdier}(\cite{Illusie},~\cite{Varshavsky})

Let $c:C \to X \times_{k} X$ be a correspondence with $C$ and $X$ proper over $k$.~Then for any cohomological correspondence $u$ from $\sF$ to itself lifting $c$

\begin{center}

$\text{Tr}(\text{R}\Gamma_c(u))=\sum_{\beta \in \pi_0(\text{Fix}(c))} LT_{\beta}(u)$.

\end{center}

Here $Tr(\text{R}\Gamma_c(u))$ is the trace of the endomorphism $R\Gamma_c(u)$ of the perfect complex (of $\bar{\Q}_{\ell}$ vector spaces) $\text{R}\Gamma_c(X,\sF)$ induced by $u$ (see Section \ref{action_correspondence_cohomology}).

\end{thm}

%
%\begin{proof}
%
%The result follows from Theorem \ref{trace_commutes_proper_pushforward} applied to $[\pi]_c:[c] \to [c_{\text{tr}}]$.
%
%%(\ref{pushforward_proper_correspondence}) and (\ref{trivial_trace}) imply that the term on the left (of Theorem \ref{trace_commutes_proper_pushforward}) evaluates to $\text{Tr}(R\Gamma_c(u))$.~That the term on the right is $\sum_{\beta \in \pi_0(\text{Fix}(c))} LT_{\beta}(u)$ follows from the definition of a local term (see Definition \ref{Local_term}).
%
%\end{proof}

\subsection{Locally contracting correspondences}

%Given the Lefschetz-Verdier trace formula,~ the computation of global traces (for a proper correspondence) is reduced to the problem of computing local terms along the scheme of fixed points.~Since the local terms are defined very non-explicitly,~computing them, in general, is quite difficult.~However when one is working over an algebraic closure of a finite field,~and under certain circumstances,~these local terms can be computed explicitly.
%
%In fact under these circumstances,~these local terms happen to be equal to the naive local terms (see Definition \ref{Naive_local_trace}).~Moreover the Lefschetz-Verdier trace formula can be used to compute the global traces (whenever they are defined) of cohomological correspondences lifting correspondences which are not necessarily proper.~This was conjectured by Deligne and first proved (conditionally) by Pink (\cite{Pink}) and unconditionally by Fujiwara (\cite{Fujiwara}).~Later Varshavsky obtained an effective generalisation of Fujiwara's trace formula (\cite{Varshavsky},~Theorem 2.3.2).~An important step in Varshavsky's proof is a simpler definition of a \textit{contracting correspondence} and its consequence to calculating local terms (\cite{Varshavsky},~Theorem 2.1.3).~Since we will be making use of \cite{Varshavsky},~Theorem 2.1.3,~we recall the formalism underlying it.

As before let $k$ be an arbitrary algebraically closed field.~Let $c \colon X  \to Y$ be a morphism of schemes finite type over $k$.

\begin{defn}[Ramification along a closed subscheme]\label{ramification_closed_subset}

For a reduced closed subscheme $Z \subseteq Y$ its \textit{ramification along $c$} is the smallest positive integer $n$ such that,~$c^{-1}(Z) \subseteq (c^{-1}(Z)_{red})_{n}$.~We denote this by $\text{Ram}(Z,c)$.

\end{defn}

\begin{defn}[Ramification degree of a morphism]\label{ramification_morphism}

If $c$ is quasi-finite,~then $\text{ram}(c)$ is defined as the maximum of $\text{ram}(y,c)$,~as $y$ varies over all the closed points of $Y$.~We denote this by $\text{ram}(c)$.

\end{defn}

\begin{rmk}

Note that our notation for ramification along a closed subscheme differs from the one in \cite{Varshavsky} to avoid any possibility of confusion with the notation for the ramification degree of a morphism.

\end{rmk}

Now let $c:C \to X \times_{k} X$ be a self-correspondence of $X$.

\begin{defn}[Invariant closed subset]\label{invariant_closed}

A closed subset $Z \subseteq X$ is said to be \textit{$c$-invariant} if $c_1(c_2^{-1}(Z))$ is set theoretically contained in $Z$.

\end{defn}

\begin{rmk}\label{restriction_correspondence_invariant_subscheme}
Suppose $Z \subseteq X$ be an invariant closed subset of $X$.~Then we define the \textit{restriction} of $[c]$ to $Z$ by $[c]|_Z \colonequals (c_2^{-1}(Z)_{\text{red}},c_{1,Z},c_{2,Z})$.~Here $c_{1,Z}$ and $c_{2,Z}$ are the maps induced by $c_1$ and $c_2$ respectively.~Note that $[c]|_Z$ is a self correspondence of $Z$.~We shall also denote this by $c|_Z$.
\end{rmk}

\begin{defn}[Locally invariant closed subset]\label{locally_invariant_closed}

A closed subset $Z \subseteq X$ is said to be \textit{locally $c$-invariant} if for each $x \in Z$ there exists an open neighbourhood $U$ of $x$ in $X$ such that $Z \cap U$ is $c|_{U}$-invariant (see Definition \ref{restriction_correspondence_open}) .

\end{defn}

\subsubsection{} \label{quasi-finite_local_invariant}

If $c_2$ is quasi-finite then any closed point of $X$ is locally $c$-invariant (see \cite{Varshavsky} Example 1.5.2).

\begin{defn}[Invariant in a neighbourhood of fixed points]\label{invariant_neighbourhood_fixed_points}
A closed subset $Z \subseteq X$ is said to be \textit{invariant in a neighbourhood of fixed points} if there exists an open neighbourhood $W$ of $\text{Fix}(c)$ in $C$ such that $Z$ is $c|_W$-invariant (see Definition \ref{restriction_correspondence_open}).
\end{defn}

\subsubsection{Restriction of cohomological correspondence to an invariant subset}\label{restriction_cohomological_correspondence_invariant_subscheme}(\cite{Varshavsky},~1.5.6 (a))
Let $c:C \to X \times_k X$ be a correspondence,~and $u$ a cohomological correspondence of $\sF$ lifting $c$.~Let $Z \subseteq X$ be a closed subset of $X$ left invariant by $c$.~Then we have a diagram

\begin{equation}\label{restrcition_invariant_subscheme_cohomological_correspondence}
\xymatrix{ Z \ar@{^{(}->}[d]_{i} & c_2^{-1}(Z)_{\text{red}} \ar[l]_{c_{1,Z}} \ar[r]^{c_{2,Z}} \ar@{^{(}->}[d]_{i_C}  & Z \ar@{^{(}->}[d]_{i} \\
                  X                   & C \ar[l]^{c_1} \ar[r]_{c_2}                       &  X                }
\end{equation} 
\noindent where the right-hand square is Cartesian up to nilpotents.~Thus as in Section \ref{base_change_section_2},~one has a natural transformation of functors,

%One has a natural transformation $c_2^{!} \to c_2^{!}i_*i^*$,~which by base change applied to the right hand square in (\ref{restrcition_invariant_subscheme_cohomological_correspondence}),~gives a natural transformation of $c_2^{!} \to i_{C*}c_{2,Z}^{!}i^*$.~Thus by adjunction one has a natural transformation 

\begin{equation}\label{first_tranformation_restriction_invariant_subscheme}
i_{C}^*c_2^{!} \to c_{2,Z}^{!}i^{*}.
\end{equation}

The cohomological correspondence $u$ is given by a map $c_1^*\sF \to c_2^{!}\sF$.~Applying $i_C^*$ to this,~we get a map from $c_{1,Z}^*i^*\sF \simeq i_C^*c_1^* \sF \to i_C^*c_2^!\sF$,~which using (\ref{first_tranformation_restriction_invariant_subscheme}) gives a map

\begin{equation}\label{restriction_along_invariant_subscheme}
u|_Z:c_{1,Z}^*(\sF|_{Z}) \to c_{2,Z}^{!}(\sF|_{Z}).
\end{equation}

Note that $u|_Z$ is a cohomological correspondence of $\sF|_Z$ lifting a self correspondence of $Z$ given by the top row in Diagram (\ref{restrcition_invariant_subscheme_cohomological_correspondence}).

\begin{rmk}\label{locally_invariant_subfield}(Compare \cite{Varshavsky_Hrushovski},~1.9,~(b))
Suppose $X$ is defined over a subfield $k'$ of $k$.~A closed subset $Z \subseteq X$ is said to be locally $c$-invariant \textit{over $k'$} if the open neighbourhoods $U$ in Definition \ref{locally_invariant_closed} can be chosen to be defined over $k'$.
\end{rmk}

\subsubsection{Local term along a subscheme invariant in a neighbourhood of fixed points}\label{local_term_subscheme_invariant_neighbourhood_fixed_points}(Compare \cite{Varshavsky},~1.5.6 (b))

Suppose $c:C \to X \times_k X$ is a correspondence such that $\text{Fix}(c)$ is proper over $k$.~Let $u$ be a cohomological correspondence of $\sF$ lifting $c$.~Let $Z$ be a closed subset of $X$ which is invariant in a neighbourhood of fixed points (see Definition \ref{invariant_neighbourhood_fixed_points}).~Let $W$ be an open neighbourhood of $\text{Fix}(c)$ in $C$ such that $c|_W$ leaves $Z$ invariant.~Then combining Definition \ref{restriction_correspondence_open} and Remark \ref{restriction_correspondence_invariant_subscheme} we can make sense of a self correspondence $c|_{W,Z}$ of $Z$ as the restriction of $c|_W$ to $Z$.

% follows.~One has a diagram
%
%\begin{equation}\label{restriction_correspondence_invariant}
%\xymatrixcolsep{4pc} \xymatrix{ W \cap (c_2^{-1}(Z))_{\text{red}} = W \cap (c_1^{-1}(Z) \cap c_2^{-1}(Z))_{\text{red}} \ar@{^{(}->}[r] &  c_1^{-1}(Z) \cap \bar{c}_2^{-1}(Z) \ar[d] \\
%                            & Z \times_k Z                  },
%\end{equation}
%where the equality is a consequence of $c|_W$ leaving $Z$ invariant.~The correspondence $c|_{W,Z}$ of $Z$ is the induced map from $W \cap (c_2^{-1}(Z))_{\text{red}} \to Z \times_k Z$.

Moreover there exists a cohomological correspondence (say $u|_W$) of $\sF$ lifting $c|_W$ (see Section \ref{restriction_cohomological_open}).~Since $Z$ is $c|_W$ invariant,~we can restrict $u|_W$ along $Z$ to get a cohomological correspondence (say $u|_{W,Z}$) of $\sF|_Z$ (see Section \ref{restriction_cohomological_correspondence_invariant_subscheme}) lifting $c|_{W,Z}$.

Let $\beta$ be any connected component of $Fix(c|_{W,Z})$.~Since $W$ was chosen to be a neighbourhood of fixed points and $\text{Fix}(c)$ was assumed to be proper,~$\beta$ is necessarily proper over $k$.~Thus we can make sense of the local term (say  $\text{LT}_{\beta}(u|_{W,Z}) \in \bar{\Q}_{\ell}$ ) of $u|_{W,Z}$ at $\beta$.~We also have a well defined local term $\text{LT}(u|_{W,Z})$ (see Definition \ref{Local_term}).~Since the local term at a connected component of the scheme of fixed points,~depends only on an open neighbourhood of the connected component,~$\text{LT}_{\beta}(u|_{W,Z})$ (and hence $\text{LT}(u|_{W,Z})$ )  are independent of the chosen open neighbourhood $W$ of $\text{Fix}(c)$. Hence we can remove $W$ from the notation and simply denote it by $\text{LT}_{\beta}(u|_Z)$ (and $\text{LT}(u|_Z)$).

\begin{defn}\label{stabilizes}

A closed subscheme $Z$ is said to be \textit{stabilized} by $c$ if $c_2^{-1}(Z)$ is a closed subscheme of  $c_1^{-1}(Z)$.

\end{defn}

Recall that for a closed subscheme $Z$ of $X$,~$Z_{k}$ is the closed subscheme of $X$ defined by the ideal $\sI_{Z}^{k}$.

\begin{defn}\label{contracting}

$c$ is said to be \textit{contracting near a closed subscheme} $Z \subseteq X$ if $c$ stabilizes $Z$,~and $c_2^{-1}(Z_{n+1})$ is a closed subscheme of $c_1^{-1}(Z_{n})$ for some $n \geq 1$.

\end{defn}

We have the following local variant of Definition \ref{contracting}

\begin{defn}\label{locally_contracting}
$c$ is said to be \textit{locally contracting near a closed subscheme} $Z \subseteq X$ if for each $x \in Z$ there exists an open neighbourhood $U$ of $x$ in $X$ such that $c|_{U}$ (see Definition \ref{restriction_correspondence_open}) is contracting near $Z \cap U$.
\end{defn}

\begin{rmk}\label{locally_contracting_subfield}(Compare \cite{Varshavsky_Hrushovski},~1.9 (c))
Suppose $X$ is defined over a subfield $k'$ of $k$.~$c$ is said to be locally contracting near a closed subscheme $Z \subseteq X$ \textit{over $k'$} if the open neighbourhoods $U$ in Definition \ref{locally_contracting} can be chosen to be defined over $k'$.
\end{rmk}

\begin{defn}\label{contracting_fixed_points}

$c$ is said to be contracting near closed subscheme $Z \subseteq X$ \textit{in a neighbourhood of fixed points} if there exists an open subscheme $W$ of $C$ containing $\text{Fix}(c)$ such that $c|_W$ is contracting near $Z$ (see Definition \ref{restriction_correspondence_open}).

\end{defn}

 We will need the following key result in the form of Corollary \ref{shape_of_trace_formula}.

\begin{thm}(\cite{Varshavsky},~Theorem 2.1.3)\label{local_term_boundary_vanishing}

Let $c:C \to X \times_k X$ be a correspondence contracting near a closed subscheme $Z \subseteq X$ in a neighbourhood of fixed points,~and let $\beta$ be an open connected subset of $\text{Fix}(c)$ such that $c'(\beta) \cap Z \neq \emptyset$ (see Section \ref{scheme_fixed_points}).~Then

\begin{enumerate}

\item[(1)] $\beta$ is contained set-theoretically in $c'^{-1}(Z)$,~hence $\beta$ is an open connected subset of $\text{Fix}(c|_{W,Z})$ (see Section \ref{local_term_subscheme_invariant_neighbourhood_fixed_points}).~Here $W$ is an open neighbourhood of $\text{Fix}(c)$ in $C$ such that $c|_W$ leaves $Z$ invariant.

\item[(2)] For every cohomological correspondence $u$ from $\sF$ to itself lifting $c$,~one has $\mathcal{T}\textit{r}_{\beta}(u)=\mathcal{T}\textit{r}_\beta(u|_Z)$ (see Section \ref{trace_maps}).~In particular if $\beta$ is proper over $k$ then,~$LT_{\beta}(u)=LT_{\beta}(u|_Z)$ (see (\ref{local_term_trace}) and Section \ref{local_term_subscheme_invariant_neighbourhood_fixed_points}).

\end{enumerate}

\end{thm}

\begin{rmk}

The theorem above holds over arbitrary algebraically closed fields and $c$ need \textit{not} to be proper.

\end{rmk}

\begin{cor}\label{shape_of_trace_formula}
Let $c:C \to X \times_k X$ be a correspondence such that $\mathrm{Fix}(c)$ is proper over $k$.~Let $u$ be a cohomological correspondence of $\sF$ lifting $c$.~Suppose the following conditions are satisfied.

\begin{enumerate}[(a)]

\item There exists an open subset $U \subseteq X$ such that $c_2|_{c_1^{-1}(U) \cap c_2^{-1}(U)}$ is quasi-finite.

\item $c|_U$ is contracting near every closed point of $U$ in a neighbourhood of fixed points.

\item $c$ is contracting near the closed subscheme $Z:=(X \backslash U)_{\text{red}}$ in a neighbourhood of fixed points. 

\end{enumerate}

Then $\mathrm{Fix}(c|_U)$ is finite over $k$.~Moreover

\begin{equation}\label{local_term_decomposition}
\mathrm{LT}(u)=\sum_{\beta \in \mathrm{Fix}(c|_U)} \mathrm{Tr}(u_{\beta})+\mathrm{LT}(u|_Z).
\end{equation}

Here $\mathrm{Tr}(u_{\beta})$ is the naive local term at $\beta$ (see Definition \ref{Naive_local_trace}),~which is well defined by assumption (a).

\end{cor}

\begin{proof}

We have a Cartesian diagram

\begin{equation}\label{shape_of_trace_formula_1}
\xymatrixcolsep{4pc} \xymatrix{ \text{Fix}(c) \ar[r]^-{c'}  \ar@{^{(}->}[d]^-{\Delta'} & X  \ar@{^{(}->}[d]^-{\Delta} \\
                   C \ar[r]^-{c}                      & X \times X                  }.
\end{equation}

Let $W$ be an open neighbourhood of $\text{Fix}(c)$ such that $c|_W$ leaves $Z$ invariant (such a $W$ exists by assumption (c)). 

Let $\beta \in \pi_0(\text{Fix}(c))$ be such that $c'(\beta) \cap Z \neq \emptyset$,~then condition (c) and Theorem \ref{local_term_boundary_vanishing} together imply that $\beta \in \pi_0(\text{Fix}(c|_{W,Z}))$,~and that $\text{LT}_{\beta}(u)=\text{LT}_{\beta}(u|_Z)$.~Thus (\ref{Local_term_1}) implies that 

\begin{equation}\label{local_term_decomposition_1}
\text{LT}(u)=\sum_{\beta \in \pi_0(\text{Fix}(c|_U))}\text{LT}_{\beta}(u|_U)+\text{LT}(u|_Z).
\end{equation}

Hence it remains to show that $\text{Fix}(c|_U)$ is finite over $k$,~and that for any $\beta \in \pi_0(\text{Fix}(c|_U))$, $\text{LT}_{\beta}(u|_U)=\text{Tr}(u_{\beta})$.

Replacing $X$ by $U$,~$c$ by $c|_U$ and $u$ by $u|_U$,~we can assume that $c_2$ is quasi-finite,~and $c$ is contracting near every closed point of $U$ in a neighbourhood of fixed points.

Let $x \in c'(\beta) \subseteq U$ be a closed point.~Since $c$ is contracting near every closed point of $U$ in a neighbourhood of fixed points Theorem \ref{local_term_boundary_vanishing},~(1) implies that $\beta$ is a connected component of $\text{Fix}(c|_{\{x\}})$.~Since $c_2$ is quasi-finite,~$\beta$ is necessarily a point.~Hence $\text{Fix}(c)$ is finite over $k$.

Moreover Theorem \ref{local_term_boundary_vanishing},~(2) implies that the local term $\text{LT}_{\beta}(u)$ at such a $\beta$ equals $\text{LT}_{\beta}(u|_{\{x\}})$ which in turn equals the naive local term at $\beta$ (see \cite{Varshavsky},~Section 1.5.7).
\end{proof}

\subsection{Generalities on Cohomology with support}

Let $Y$ be any closed subscheme of $X \times_k X$ contained in $U$.~As before $Z$ is the complement of $U$ in $X$,~with the reduced structure.~Let $i_{Y,U}$ and $i_Y$ denote the inclusion of $Y$ into $U$ and $X \times_k X$ respectively.~Thus one has a diagram

\begin{equation}\label{supports_first_diagram}
\xymatrix{Y  \ar@{^{(}->}[r]^{i_{Y,U}} & U \ar@{^{(}->}[r]^-{j} & X \times_k X & Z \ar@{_{(}->}[l]_-{i}}.
\end{equation}

The composite $ j \circ i_{Y,U}$ equals $i_Y$ by definition and $Y \cap Z=\emptyset$.~For any sheaf $\sG$ on $X \times_k X$,~by $H^p_Y(X \times_k X,~\sG)$ we mean the cohomology $H^{p}(Y,i^{!}_Y\sG)$.~Similarly for any sheaf $\sG$ on $U$ by $H^p_Y(U,\sG)$ we mean the cohomology $H^{p}(Y,~i_{Y,U}^{!}\sG)$.~We have a natural isomorphism

\begin{equation}\label{cohomology_with_supports_using_hom}
H^{0}_Y(X \times_k X,\sG) \simeq \text{Hom}(i_{Y*}\bar{\Q}_{\ell},\sG).
\end{equation}

Since $\text{Hom}(i_{Y*}\hspace{1mm},~i_{*} \hspace{1mm}) \simeq 0$,~applying the cohomological functor $\text{Hom}(i_{Y*}\bar{\Q}_{\ell},~\hspace{1mm})$ to the triangle

\begin{equation}\label{triangle_cohomology_supports}
\xymatrix{j_!j^* \sG \ar[r] &\sG \ar[r] & i_{*}i^*\sG \ar[r]^-{+1}&},
\end{equation}

\noindent and using (\ref{cohomology_with_supports_using_hom}) we obtain an natural isomorphism

\begin{equation}\label{first_specific_isomorphism_supports}
H^0_{Y}(X \times_k X,~\sG) \simeq H^0_Y(X \times_k X,~j_{!}j^*\sG). 
\end{equation}

\section{Preliminary reductions}\label{preliminary_reductions}

In this section, we carry out some preliminary reductions towards the proof of Theorem \ref{non_uniform_estimate}.~The crucial one is the reduction to the locally invariant boundary (see Definition \ref{locally_invariant_closed}) case using \cite{Varshavsky_Hrushovski},~Section 2.

\subsection{Reduction to $X$ smooth and $c_2$ generically \'etale}

\subsubsection*{Assume $\text{dim}(C)=0$}

The claim $(1)$ and $(3)$ in Theorem \ref{non_uniform_estimate} are trivially true when $\text{dim}(C)=0$.~Thus we only need to verify the rationality claim in $(2)$.~To do so we can assume $C$ is connected,~and hence is a closed point $c$.

If $c_2(c)$ is not in the Frobenius orbit of $c_1(c)$,~then $\text{Fix}(c^{(n)})$ is empty for all $n$,~and the rationality is trivial.~Else $\text{Fix}(c^{(n)})=1$,~for all $n \equiv r \mod s$,~for some integers $r$ and $s$.~This implies the rationality of $Z(c,t)$ in this case.~In particular Theorem \ref{non_uniform_estimate} is true when $\text{dim}(X)=0$.

Before we carry out the necessary reductions we record the following simple observation which shall be repeatedly used in this Section.

\begin{rmk}\label{basic_observation_reduction}

Since Theorem \ref{non_uniform_estimate} is true when $\text{dim}(X)=0$,~inducting on $\text{dim}(X)$,~to prove Theorem \ref{non_uniform_estimate} for $c: C \to X \times_k X$,~we can throw away an arbitrary closed subset of $X$ (defined over $\F_q$) of strictly smaller dimension.~In particular we may assume that $X$ is affine (and hence quasi-projective over $k$).

One can also induct on the dimension of $C$,~and hence to prove Theorem \ref{non_uniform_estimate} for $c: C \to X \times_k X$,~we can throw away an arbitrary closed subset of $C$ of strictly smaller dimension.~As above we may also assume that $C$ is quasi-projective over $k$.

\end{rmk}

\begin{lem}\label{basic_reduction}
It suffices to prove Theorem \ref{non_uniform_estimate} under the following additional assumptions,

\begin{enumerate}

\item $c$ is defined over $\F_q$.

\item $\text{dim}(C)=\text{dim}(X)$.

\item $C$ and $X$ are irreducible and quasi-projective.

\item $c_1$ and $c_2$ are dominant.

\item $c$ is proper.

\item $X$ is smooth.

\end{enumerate}

\end{lem}

\begin{proof}

\textit{Reduction (1):}

Assume that $c$ is defined over $\F_{q^s}$.~Any natural number $n$ can be written (uniquely) as $n=qs+r$,~for $r \in \{0,~1,~2, \cdots s-1\}$.~Clearly the assertion for $c:C \to X \times_k X$ with respect to $\F_q$,~is equivalent to the assertion for $c^{(i)}:C \to X \times_{k} X,~0 \leq i \leq s-1$ with respect to $\F_{q^s}$.~Thus one can assume $c$ is defined over $\F_q$.~Henceforth we shall assume that $c$ is defined over $\F_q$.

\textit{Reduction (2):}

Clearly $\text{dim}(C) \leq \text{dim}(X)$.~If $\text{dim}(C) < \text{dim}(X)$.~There exists a non-empty open subset $U$ of $X$ such that $c^{-1}(U \times_k U)=\emptyset$.~Hence Theorem \ref{non_uniform_estimate} is trivially true for $c^{-1}(U \times_k U) \to U \times_k U$.~Thus the reduction now follows from Remark \ref{basic_observation_reduction}.~Henceforth we shall assume that $\text{dim}(C)=\text{dim}(X)$.

\textit{Reduction (3):}

Let $X_1,~X_2,\cdots,X_t$ be the finitely many irreducible components of $X$ (over $k$).~Let $\F_{q^s}$ be a finite sub extension of $k$ containing $\F_q$,~over which $X_i$ and all the irreducible components of $c_1^{-1}(X_i) \cap c_2^{-1}(X_j)$ are defined for $1 \leq i,j \leq t$.

We now have two possible sub-cases.

\textit{Subcase 1: $s=1$}

This means all the $X_i$'s and all the components of $c_1^{-1}(X_i) \cap c_2^{-1}(X_j)$ are defined over $\F_q$ for $1 \leq i,j \leq t$.~Let $C_{ii'}$ be the distinct irreducible components of $c_1^{-1}(X_i) \cap c_2^{-1}(X_i)$,~for $1 \leq i \leq t$ with $i'$ running over a finite indexing set.~Then $c$ restricts to morphisms $c_{ii'}:C_{ii'} \to X_i \times_k X_i$,~which are quasi-finite along second projection,~and are defined over $\F_q$.

It then follows from Remark \ref{basic_observation_reduction} that Theorem \ref{non_uniform_estimate} holds for the correspondence $c$ if it holds for the finitely many correspondences $c_{ii'}$.

\textit{Subcase 2: $s >1$}

When $s>1$ we can argue as in Reduction $(1)$ to reduce to the case $s=1$.~Indeed consider the correspondences $c^{(n')}:C \to X \times_k X,~0 \leq n' \leq s-1$.~Since the geometric Frobenius $F_{X/\F_q}$ permutes the irreducible components of $X$,~for any two indices $i$ and $j$,~any irreducible component of $(c_1^{(n')})^{-1}(X_i) \cap (c_2^{(n')})^{-1}(X_j)$ is an irreducible component of  $c_1^{-1}(X_{i'}) \cap c_2^{-1}(X_{j'})$ for some indices $i'$ and $j'$. Moreover Theorem \ref{non_uniform_estimate} for $c:C \to X \times_k X$ with respect to $\F_q$,~is equivalent to the assertion for $c^{(n')}:C \to X \times_{k} X,~0 \leq i \leq s-1$ with respect to $\F_{q^s}$.~Thus one can assume $s=1$,~in which case we are reduced to Subcase $1$.

Henceforth we shall assume that $C$ and $X$ are irreducible.

\textit{Reduction (4):}

If either $c_1$ or $c_2$ is not dominant,~then there exists a non-empty open subset $U$ of $X$ such that $c^{-1}(U \times_k U)=\emptyset$.~Hence Theorem \ref{non_uniform_estimate} is trivially true for $c^{-1}(U \times_k U) \to U \times_k U$.~Thus the reduction now follows from Remark \ref{basic_observation_reduction}.

Henceforth we shall assume that $c_1$ and $c_2$ are dominant.

\textit{Reduction (5):}

Since $C$ and $X$ are quasi-projective,~we can factor $c$ as

\begin{center}

$\xymatrix{C \ar[r]^-{c}   \ar@{^{(}->}[d]^-{j_C} & X \times X                \\  
                    \bar{C} \ar[ur]_-{\bar{c}} & }$.

\end{center}

Here $\bar{C}$ is a quasi-projective variety, $j_C$ a dense open immersion and $\bar{c}$ a proper (even projective) morphism.~Let $\bar{c}_1$ and $\bar{c}_2$ be the induced projections onto $X$ from $\bar{C}$.~By assumption $c_2$ is dominant map between varieties of the same dimension,~and hence so is $\bar{c}_2$.~Thus there exists a non-empty open subset $U$ of $X$ such that,~$\bar{c}_2^{-1}(U) \to U$ is a finite morphism.~Finally using Remark \ref{basic_observation_reduction} we first shrink $X$ to $U$,~and then extend $C$ to $\bar{c}^{-1}(U \times_k U)$.

Henceforth we shall assume that $c$ is proper. 

\textit{Reduction (6):}

This is an immediate consequence of Remark \ref{basic_observation_reduction}.

\end{proof}

\begin{rmk}\label{(3)_implies(1)}
We note that under the reductions guaranteed by Lemma \ref{basic_reduction},~the bound in Theorem \ref{non_uniform_estimate} $(1)$ is a consequence of the bound in $(3)$.~Thus it suffices to prove the bound in $(3)$. 
\end{rmk}

As a next step towards proving Theorem \ref{non_uniform_estimate} (in addition to the above reductions) we reduce to the case when $c_2$ is generically \'etale.~For any scheme $X$ over $\F_p$,~we denote the absolute Frobenius morphism by $F_X$.~Let $k$ be an arbitrary field of characteristic $p>0$.~For any $p$-primary number $q=p^r$,~and any scheme $X/k$,~by $X^{(q)}$ we mean the base change $k \times_k X$,~where $k$ is considered as a $k$-scheme via  the $r^{\mathrm{th}}$ iterate of the absolute Frobenius of $k$.~Note that $X^{(q)}$ naturally gets a structure of a $k$-scheme via the first projection.~We call $X^{(q)}$ the $q^{\mathrm{th}}$ Frobenius twist of $X$ with respect to $k$.~Let $F^{(q)}_{X/k}$ be the induced relative Frobenius morphism from $X \to X^{(q)}$ (with respect to $k$).~Note that $F^{(q)}_{X/k}$ is a morphism over $k$.

We will need the following simple Lemmas which we state without proof.

\begin{lem}\label{frobenius_commutes_reduced}

Let $X$ be a scheme of finite type over a field $k$.~Then the natural inclusion of subschemes,~$((X_{\text{red}})^{(q)})_{\text{red}}  \subseteq  (X^{(q)})_{\text{red}}$ induced by the inclusion $X_{\text{red}} \subseteq X$ is an isomorphism.

\end{lem}

%\begin{proof}
%
%Since $((X_{\text{red}})^{(q)})_{\text{red}}$ is a reduced closed subscheme of $(X^{(q)})_{\text{red}}$ it suffices to show that it is supported on $|X^{(q)}|$.~Since the base change morphisms $X^{(q)} \to X$ and $(X_{\text{red}})^{(q)} \to X_{\text{red}}$ are universal homeomorphisms,~we get the required result.
%
%\end{proof}

\begin{lem}\label{relative_frobenius_etale}

Let $X=\Spec(A)$ be \'etale over a field $k$ of characteristic $p>0$.~Then the relative Frobenius $F^{(q)}_{X/k} \colon X \to X^{(q)}$ is an isomorphism for all $p$-primary numbers $q$.
\end{lem}

%\begin{proof}
%
%One is immediately reduced to the case when $X$ is connected that is $X \simeq \Spec(L)$,~where $L/k$ is a finite separable extension.~It suffices to show that $X^{(q)} \simeq \Spec(L')$ for a field $L' \supseteq k$,~such that $[L':k]=[L:k]$. 
%
%The induced morphism from $X^{(q)}$ to $\Spec(k)$ being a base change of an \'etale morphism is also \'etale.~Further $X^{(q)}$ is necessarily a connected scheme since $X$ is connected.
%
%Thus $X^{(q)} \simeq \Spec(L')$,~where $L'$ is a finite separable extension of $k$ of the same degree as $L/k$.
%\end{proof}

The following lemma is a local form of our desired result.

\begin{lem}\label{field_frobenius_etale}

Let $L/K$ be a finite extension of fields of characteristic $p>0$.~Suppose that the inseparable degree of $L/K$ is $p^n$.~Then $(\Spec(L)^{(p^n)})_{\text{red}}$ is connected and \'etale over $\Spec(K)$ of rank $[L:K]/p^n$. Here $\Spec(L)^{(p^n)}$ is the Frobenius twist of $\Spec(L)$,~considered as $K$-scheme.

\end{lem}

\begin{proof}

Let $K'$ be the separable closure of $K$ in $L$.~Thus we have a commutative diagram with Cartesian squares

\begin{equation}\label{field_Frobenius_etale_1}
\xymatrix{ \Spec(L)^{(p^n)} \ar[r] \ar[d] & \Spec(L) \ar[d] \\
                 \Spec(K')^{(p^n)} \ar[r] \ar[d] & \Spec(K') \ar[d] \\
                 \Spec(K) \ar[r]^-{F^{n}_{\Spec(K)}} & \Spec(K) }.
\end{equation}

Here $\Spec(K')^{(p^n)}$ is the $p^{n}$ Frobenius twist of $\Spec(K')$ with respect to $K$.~Lemma \ref{relative_frobenius_etale} implies that $\Spec(K')^{(p^n)}$ is isomorphic to $\Spec(K')$ under the relative Frobenius (with respect to $\Spec(K)$).~Thus the lower square in Diagram \ref{field_Frobenius_etale_1} is isomorphic to 

\begin{equation}\label{field_Frobenius_etale_2}
\xymatrix{\Spec(K') \ar[r]^-{F^{n}_{\Spec(K')}} \ar[d] & \Spec(K') \ar[d] \\
                 \Spec(K) \ar[r]^-{F^{n}_{\Spec(K)}} & \Spec(K) }.
\end{equation}

Since the inseparable degrees of $L/K$ and $L/K'$ are the same,~transitivity of base change applied to the Diagram \ref{field_Frobenius_etale_1} implies that it suffices to prove the lemma for $L/K'$.~Thus we are reduced to proving that if $L/K$ is a purely inseparable extension of degree $p^n$,~then the natural morphism $(\Spec(L)^{(p^n)})_{\text{red}} \to \Spec(K)$ is an isomorphism.~This is immediate.

%Let $\alpha_1, \alpha_2 ,\cdots , \alpha_r$  in $L$ be a set of generators of $L$ as a $K$-algebra.~Thus there exists a $K$-algebra isomorphism,~
%
%\begin{center}
%
%$\dfrac{K[X_1, X_2 ,\cdots , X_r]}{(X_1^{p^{t_1}} - \beta_1, X_2^{p^{t_2}}-\beta_2, \cdots, X^{p^{t_r}}-\beta_r)} \simeq L$
%
%\end{center}
%
%\noindent given by sending $X_{i}$ to $\alpha_{i}$.~Since $L/K$ is purely inseparable of degree $p^n$,~the $t_{i}$'s can be chosen to be bounded above by $n$.~Moreover it follows from the definition of $\Spec(L)^{(p^n)}$ that 
%
%\begin{center}
%
%$\Spec(L)^{(p^n)} \simeq \Spec(\dfrac{K[X_1, X_2 ,\cdots , X_r]}{(X_1^{p^{t_1}} - \beta_1^{p^{n}}, X_2^{p^{t_2}}-\beta_2^{p^n}, \cdots, X^{p^{t_r}}_r-\beta_r^{p^n})})$,
%
%
%\end{center}
%
%\noindent as $K$-schemes.~Hence the result.

\end{proof}

\begin{cor}\label{frobenius_etale_sufficient}

Let $L/K$ be a finite extension of fields of characteristic $p>0$.~Suppose that the inseparable degree of $L/K$ is $p^n$.~Then $(\Spec(L)^{(p^m)})_{\text{red}}$ is connected and \'etale over $\Spec(K)$ of rank $[L:K]/p^n$ for all $m \geq n$.

\end{cor}

\begin{proof}

Lemma \ref{frobenius_commutes_reduced}  implies that 

\begin{equation}\label{frobenius_etale_sufficient_1}
(\Spec(L)^{(p^m)})_{\text{red}} \simeq (((\Spec(L)^{(p^n)})_{\text{red}})^{(p^{m-n})})_{\text{red}}.
\end{equation}

Here all the Frobenius twists are with respect to $K$.~Lemma  \ref{field_frobenius_etale} implies that $(\Spec(L)^{(p^n)})_{\text{red}}$ is connected and \'etale over $\Spec(K)$ of rank $[L:K]/p^n$.~The corollary is then a consequence of Lemma \ref{relative_frobenius_etale}.

\end{proof}

We will need the following global variant of Corollary \ref{frobenius_etale_sufficient}.

\begin{lem}\label{frobenius_etale_sufficient_schemes}

Let $k$ be a field of characteristic $p>0$.~Let $f \colon Y \to X$ be a dominant morphism between varieties (over $k$) of  dimension $d$.~Let $\delta \colonequals [k(Y) \colon k(X)]$ induced by $f$ at the level of generic points.~Let $p^n$ be the inseparable degree of $k(Y)/k(X)$.~Let $F_X:X \to X$ be the absolute Frobenius of $X$ with respect to $\F_p$.~For every $m \geq 1$,~consider the Cartesian diagram

\begin{center}

$\xymatrix{Y \times_{X,F^m_X} X \ar[d]^-{f^{(m)}} \ar[r]^-{F^{m'}_X} & Y \ar[d]^-{f} \\
             X                        \ar[r]^-{F_X^m}          &  X }$.

\end{center}

Then 

\begin{enumerate}

\item the map $F^{m'}_X$ is an universal homeomorphism for any $m \geq 1$.

\item $\text{dim}(Y \times_{X,F^m_X} X)=d$.

\item $f^{(m)}$ is dominant for any $m \geq 1$.

\item For any $m \geq n$,~$(f^{(m)})_{\text{red}}$ is generically \'etale of degree $\delta/p^n$.

\end{enumerate}

\end{lem}

\begin{proof}

The morphism $F_X:X \to X$ being a universal homeomorphism immediately implies part $(1)$ and $(2)$ of the above proposition.~Since $f$ is assumed to be dominant,~$f^{(m)}$ is necessarily dominant and thus $(3)$ holds.~Hence it remains to show $(4)$.

Note that $(2)$ and $(3)$ together then imply that the natural map 

\begin{equation}\label{frobenius_etale_sufficient_schemes_2}
\Spec(k((Y \times_{X,F^m_X} X)_{\text{red}})) \to ((Y \times_{X,F^m_X} X) \times_X \Spec(k(X)))_{\text{red}}
\end{equation}

\noindent is an isomorphism.~Similarly since $X$ and $Y$ are varieties of the same dimension,~ the dominance of $f$ implies that the natural map

\begin{equation}\label{frobenius_etale_sufficient_schemes_3}
\Spec(k(Y)) \to Y \times_X \Spec(k(X)),
\end{equation}

\noindent is an isomorphism.

Let  $j_{\eta_X}:\Spec(k(X)) \to X$ be the inclusion of the generic point of $X$ into $X$.~Note that the morphism induced by the absolute Frobenius $F_X$ at the level of the generic point corresponds to the Frobenius of $k(X)$.~Thus one has a commutative diagram

\begin{equation}\label{frobenius_etale_sufficient_schemes_1}
\xymatrix{X \ar[r]^-{F^m_X} & X  \\
                  \Spec(k(X)) \ar[u]^{j_{\eta_X}} \ar[r]^{F^m_{k(X)}}  &  \Spec(k(X)) \ar[u]^{j_{\eta_X}} }.
\end{equation}

Thus the isomorphism (\ref{frobenius_etale_sufficient_schemes_3}) and transitivity of base change together imply that

\begin{equation}\label{frobenius_etale_sufficient_schemes_4}
((Y \times_{X,F^m_X} X) \times_X \Spec(k(X)))_{\text{red}} \simeq \Spec(k(Y))^{(p^m)},
\end{equation}

\noindent where the Frobenius twist on the right is with respect to the $k(X)$-scheme structure via $f$.~Combining the isomorphisms (\ref{frobenius_etale_sufficient_schemes_2}) and (\ref{frobenius_etale_sufficient_schemes_4}) we get,

\begin{equation}\label{frobenius_etale_sufficient_schemes_5}
\Spec(k((Y \times_{X,F^m_X} X)_{\text{red}})) \simeq \Spec(k(Y))^{(p^m)},
\end{equation}

\noindent as $\Spec(k(X))$ schemes.~Thus $(4)$ now follows from Corollary \ref{frobenius_etale_sufficient}.

\end{proof}

\begin{prop}\label{reduction_generically_etale}

In addition to the reductions above,~It suffices to prove Theorem \ref{non_uniform_estimate} under the assumption that $c_2$ is generically \'etale.

\end{prop}

\begin{proof}

If $c_2$ is generically \'etale we are done.~Hence assume that $c_2$ is not generically \'etale,~and let $\delta'=p^n$ be the inseparable degree of $k(X)/k(C)$ induced by $c_2$.~This in particular implies $d=\text{dim}(X) >0$.

Let $F_{X/\F_q}$ be the geometric Frobenius of $X$ with respect to $\F_q$.~Let $r$ be the smallest integer such that $q^r \geq p^n$.~Consider the Cartesian diagram

\begin{equation}\label{reduction_generically_etale_1}
\xymatrix{C \times_{X \times_k X} (X \times_k X) \ar@{^{(}->}[r] \ar[d] & X \times_k X \ar[d]^-{1 \times F^r_{X/\F_q}} \ar[r]^-{\text{pr}_2} & X \ar[d]^-{F^r_{X/\F_q}} \\
                  C    \ar@{^{(}->}[r] &  X \times_k   X \ar[r]^-{\text{pr}_2} & X    }.
\end{equation}

Let $C' \colonequals (C \times_{X \times_k X} (X \times_k X) )_{\text{red}}$ and $c':C' \subseteq X \times_k X$ be the induced correspondence.~Let $F'$ be the induced map from $C'$ to $C$.~Since $F_{X/\F_q}$ is a universal homeomorphism,~$C'$ is necessarily irreducible.

By definition $c_1'=c_1 \circ F$.~Hence 

\begin{equation}\label{reduction_generically_etale_5'}
\text{deg}(c_1')=\delta \text{deg}(F').
\end{equation}

We claim that

\begin{enumerate}
\item $c_2'$ is generically \'etale of degree $\frac{\text{deg}(c_2)}{\delta'}$.

\item To prove Theorem \ref{non_uniform_estimate} for the correspondence $c:C \subseteq X \times_k X$,~it suffices to prove Theorem \ref{non_uniform_estimate} for the correspondence $c':C' \subseteq X \times_k X$.
\end{enumerate}

$(1)$ and $(2)$ together imply the Proposition.~We now prove these claims.

\subsubsection*{Proof of claim 1}
First, note that $c'$ is also defined over $\F_q$.~Further transitivity of base change applied to the Diagram \ref{reduction_generically_etale_1} implies that the following diagram is Cartesian upto nilpotents

\begin{equation}\label{reduction_generically_etale_2}
\xymatrix{C' \ar[d]^-{c_2'} \ar[r]^-{F'} & C \ar[d]^-{c_2} \\
                X \ar[r]_-{F^{r}_{X/\F_q}} & X}.
\end{equation}

Let 

\begin{equation}\label{reduction_generically_etale_3}
\xymatrix{C_0 \times_{X_0 \times_{\F_q} X_0} (X_0 \times_{\F_q} X_0) \ar@{^{(}->}[r] \ar[d] & X_0 \times_{\F_q} X_0 \ar[d]^-{1 \times F^{[\F_q:\F_p]r}_{X_0}} \ar[r]^-{\text{pr}_2} & X_0 \ar[d]^-{F^{[\F_q:\F_p]r}_{X_0}} \\
                  C_0    \ar@{^{(}->}[r] &  X_0 \times_{\F_q}   X_0 \ar[r]^-{\text{pr}_2} & X_0    }
\end{equation}

\noindent be a model for the Diagram \ref{reduction_generically_etale_1} over $\F_q$.~Here as before $F_{X_0}$ is the absolute Frobenius of $X_0$ with respect to $\F_p$.~By assumption 

\begin{equation}\label{reduction_generically_etale_4}
[\F_q:\F_p]r \geq n.
\end{equation}

Let $C'_0:=(C_0 \times_{X_0 \times_{\F_q} X_0} (X_0 \times_{\F_q} X_0))_{\text{red}}$.~Since $\F_q$ is perfect,~$C'_0$ is a model of $C'$ over $\F_q$.~Thus the morphism $c'_0:C'_0 \subseteq X_0 \times_{\F_q} X_0$ is a model for $c'$.~Let $(c_2)_0$ and $(c'_2)_0$ be the morphism induced from $C_0$ and $C'_0$ respectively,~to $X_0$ under the second projection.~As before we get a diagram which is Cartesian upto nilpotents

\begin{equation}\label{reduction_generically_etale_5}
\xymatrix{C'_0 \ar[d]^-{(c'_2)_0} \ar[r]^-{F'_0} & C \ar[d]^-{ (c_2)_0} \\
                X_0 \ar[r]_-{F^{[\F_q:\F_p]r}_{X_0}} & X},
\end{equation}

\noindent which is a model for the Diagram \ref{reduction_generically_etale_2} over $\F_q$.

The inequality (\ref{reduction_generically_etale_4}) and Lemma \ref{frobenius_etale_sufficient_schemes} together imply that $(c'_2)_0$ is generically \'etale of degree $\frac{\text{deg}(c_2)}{\delta'}$,~and hence so is $c'_2$.

\subsubsection*{Proof of claim (2)}

Since $\text{deg}(c_2')=\frac{\text{deg}(c_2)}{\delta'}$,~the commutative diagram (\ref{reduction_generically_etale_2}) implies that $\text{deg}(F')=\frac{q^{rd}}{p^n}$.~Thus the equality (\ref{reduction_generically_etale_5'}) implies that the generic degree $\delta_1'$ of $c_1'$ is $\delta q^{rd}/p^n$.

Note that $F'$ induces a bijection (also denoted by $F'$) between $|C'|=C'(k)$ and $|C|=C(k)$.~We claim that this bijection restricts to a set theoretic bijection between $\text{Fix}(c'^{(n')})(k) \subset C'(k)$ and $\text{Fix}((c)^{(n'+r)})(k) \subset C(k)$ for all positive integers $n'$.

Indeed let $\alpha \in \text{Fix}(c'^{(n')})(k) \subset C'(k)$.~Then $F_{X/\F_q}^{n'} \circ c_1 \circ F'(\alpha)=c_2'(\alpha)$.~Hence the commutative diagram (\ref{reduction_generically_etale_2}) implies that $F_{X/\F_q}^r \circ F_{X/\F_q}^{n'} \circ c_1 \circ F'(\alpha)= F_{X/\F_q}^r \circ c_2'(\alpha)=c_2 \circ F'(\alpha)$.~Thus $F'(\alpha) \in \text{Fix}(c^{(n'+r)})(k)$. The converse is also true since $F^r_{X/\F_q}$ is a homeomorphism.

Now suppose Theorem \ref{non_uniform_estimate} has been established for $c'$.~This clearly implies $(1)$ and $(2)$ in Theorem \ref{non_uniform_estimate} for $c$.~Thus we only need to establish $(3)$ for $c$ assuming the same for $c'$.

The degree of $c_1'$ is $\delta q^{rd}/p^n$,~and $c_2'$ is generically \'etale.~Thus there exist integers $N'$ and $M'$ such that for any integer $n' \geq N'$,~$\text{Fix}(c'^{(n')})$ is finite (over $k$),~and 

\begin{equation}\label{estimate_for_generically_etale}
|\#\text{Fix}(c'^{(n')})(k)-\frac{\delta q^{rd}}{p^n} q^{n'd}| \leq M'q^{n'(d-\frac{1}{2})}.
\end{equation}

Moreover since $\#\text{Fix}(c'^{(n')})(k)$ and $\#\text{Fix}(c^{(n'+r)})(k)$ are set theoretically bijective,~the inequality (\ref{estimate_for_generically_etale}) can be rewritten as 

\begin{equation}\label{estimate_for_not_generically_etale}
|\#\text{Fix}(c^{(n'+r)})(k)-\frac{\delta}{\delta'} q^{(n'+r)d}| \leq Mq^{(n'+r)(d-\frac{1}{2})},
\end{equation}

where $M=\frac{M'}{q^{r(d-\frac{1}{2})}}$.~Note that $\delta$ is the degree of $c_1$,~and $\delta'=p^n$ the inseparable degree of $c_2$.~Thus Theorem \ref{non_uniform_estimate}.~$(3)$ holds true for $c$ with $N$ and $M$ equal to $N'+r$ and $\frac{M'}{q^{r(d-\frac{1}{2})}}$ respectively.

\end{proof}

Finally, the key reduction is analogous to the one carried out in \cite{Varshavsky_Hrushovski}.~Indeed using Varshavsky's construction (\cite{Varshavsky_Hrushovski},~Corollary 2.4 and Claim 5.2,~Step 4) we reduce to the case of locally invariant boundary.~The key construction being \cite{Varshavsky_Hrushovski},~Proposition 2.3 which shows that after replacing $X$ by an appropriate blow up we may assume that the correspondence is locally invariant (see Definition \ref{locally_invariant_closed}) along the \textit{boundary}.

\begin{prop}\label{reduction_to_locally_invariant_boundary}
It suffices to prove Theorem \ref{non_uniform_estimate} under the assumption that $c \colon C \to X \times_k X $ has a compactification $\bar{c} \colon \bar{C} \to \bar{X} \times_k \bar{X}$ (over $\F_q$)  with $\bar{C}$ and $\bar{X}$ projective and such that the boundary is locally invariant over $\F_q$.
 \end{prop}

\begin{rmk}\label{compactification_etale_also}
We note that since by assumption $c$ is proper,~Corollary \ref{compactification_effect_open_part} implies that $\bar{c}^{-1}(X \times_k X)=C$.~In particular if $c_2$ is \'etale,~so is $\bar{c}_2$ when restricted to $C$
\end{rmk}

\subsection{A bound on the global terms}
Let $k$ be an algebraic closure of $\F_q$.~Let $c \colon C \to X \times_k X$ be a correspondence of projective varieties\footnote{In this section $C$ and $X$ are projective varieties.~We deliberately avoid the `bar' decoration since we only need to deal with the compactified correspondence in this section.}.~Further assume that 
\noindent \begin{enumerate}
\item $c$ is defined over $\F_q$ (and we choose a model).
\item $\text{dim}(C)=\text{dim}(X)=d$.
\item $c_1$ and $c_2$ are dominant.
\item $c_2$ is generically \'etale.
\end{enumerate}

Let $j:U \hookrightarrow X$ be a non-empty smooth open subscheme defined over $\F_q$ such that,~$c_{2}|_{(c_1^{-1}(U) \cap c_2^{-1}(U))}$ is \'etale.~Recall that we are in such a set up (see Remark \ref{compactification_etale_also}).

Let $C' \colonequals c_1^{-1}(U) \cap c_2^{-1}(U)$.~By definition $[c]|_U$ is the correspondence $c' \colon C' \to U \times_k U$ induced by $c$.~Further since $c_2|_{C'}$ is assumed to be \'etale,~$C'$ is smooth over $k$.

There is a natural cohomological correspondence of $\bar{\Q}_{\ell}[d]|_U$ lifting $c'$ given by

\begin{equation}\label{trivial_cohomological_correspondence}
\xymatrix{c_1^{'*}\bar{\Q}_{\ell}[d] \ar[r]^-{\cong}_-{c^{'*}_1} &  \bar{\Q}_{\ell}[d] \ar[r]^-{\cong}_-{(S)} & \mathbb{D}_{C'}  \bar{\Q}_{\ell}[d] \ar[r]^-{\cong}_-{c_2^{'!}} & c_2^{'!} \mathbb{D}_{C'}  \bar{\Q}_{\ell}[d] \ar[r]^-{\cong}_-{(S)} & c_2^{'!}\bar{\Q}_{\ell}[d]}.
\end{equation}

Here the isomorphisms $(S)$ come from the smoothness of $C'$.~The morphism $c_2^{'!}$ is dual to the natural isomorphism $c_2^{'*} \colon c_2^{'*}\bar{\Q}_{\ell}[d] \simeq \bar{\Q}_{\ell}[d]$.

Since the correspondence $c|_U$ is quasi-finite along the second projection,~we can make sense of the naive local term along closed points $y$ in $\text{Fix}(c|_U)$ (see Definition \ref{Naive_local_trace}).

\begin{lem}\label{naive_local_term_calculate}
For the cohomological correspondence $u'$ in (\ref{trivial_cohomological_correspondence}),~and any closed point $y$ in $\text{Fix}(c|_U)$,~the naive local term $\text{Tr}(u'_{y})=1$.
\end{lem}

\begin{proof}
The map $c_{2!}^{'}c^{'*}_1\bar{\Q}_{\ell}[d] \to \bar{\Q}_{\ell}[d]$ induced by (\ref{trivial_cohomological_correspondence}) via adjunction between $(c^{'}_{2!},c^{'!}_2)$~under the isomorphism (\ref{trivial_cohomological_correspondence}),~corresponds to the natural adjunction map $c_{2!}^{'}c_2^{'!} \bar{\Q}_{\ell}[d] \to \bar{\Q}_{\ell}[d]$.~Since $c_{2}^{'}$ is \'etale,~this coincides with the trace map for quasi-finite flat morphism \cite[Chapter II, Lemma 1.1]{Freitag_Kiehl} and the result follows.

\end{proof}

In this set up we have the following Corollary to Proposition \ref{pullback_intersection_cohomology}.

\begin{cor}\label{cohomological_correspondence_lifting_correspondence}
There exists a cohomological correspondence (defined over $\F_q$)
\begin{equation}\label{cohomological_correspondence_lifting_correspondence_1}
u:c_1^*\text{IC}_X \to c_2^!\text{IC}_X
\end{equation}
lifting $c$ such that,~$u|_U$ (see Section \ref{restriction_cohomological_open}) is the correspondence (\ref{trivial_cohomological_correspondence}).
\end{cor}

\begin{proof}

Corollary \ref{adjunction_map_intersection_cohomology} applied to $c_1:C \to X$ gives us a map (defined over $\F_q$)

\begin{equation}\label{cohomological_correspondence_lifting_correspondence_2}
c_1^*:c_1^*\text{IC}_X \to \text{IC}_C.
\end{equation}

Corollary \ref{adjunction_dual_map_intersection_cohomology} applied to $c_2:C \to X$ gives us a map (defined over $\F_q$)

\begin{equation}\label{cohomological_correspondence_lifting_correspondence_3}
c_2^!:\text{IC}_C \to c_2^!\text{IC}_X.
\end{equation}

We define $u:c^{*}_1\text{IC}_X \to c_2^!\text{IC}_X$ to be the composition of (\ref{cohomological_correspondence_lifting_correspondence_2}) and (\ref{cohomological_correspondence_lifting_correspondence_3}).~Clearly $u$ is also defined over $\F_q$.~Since $C'$ is smooth,~the compatibility with (\ref{trivial_cohomological_correspondence}) is an immediate consequence of the compatibility in Corollaries \ref{adjunction_map_intersection_cohomology} and \ref{adjunction_dual_map_intersection_cohomology}.

\end{proof}

We now show that (\ref{cohomological_correspondence_lifting_correspondence_1}) produces the `correct' leading term and error term in (\ref{main_theorem_estimate}).

Let $\tau \colon \bar{\Q}_{\ell} \hookrightarrow \C$ be an embedding.~Let $\delta$ be the generic degree of $c_1$.

\begin{prop}\label{global_term_correspondence_intersection_complex}
There exists a positive real number $M_{\text{glo}}(\tau)$ such that for any $n \geq 0$

\begin{equation}\label{global_term_correspondence_intersection_complex_1}
|\text{Tr}(R\Gamma(u^{(n)}))-\delta q^{nd}| \leq M_{\text{glo}}(\tau)q^{n(d-\frac{1}{2})}.
\end{equation}

Here $\text{Tr}(R\Gamma(u^{(n)}))$ is the trace of the endomorphism $R\Gamma(u^{(n)})$,~of the perfect complex (of $\bar{\Q}_{\ell}$ vector spaces) $R\Gamma(X,\text{IC}_X)$,~induced by the cohomological correspondence (\ref{cohomological_correspondence_lifting_correspondence_1}).
\end{prop}

\begin{proof}
Since $H^{i}(X,\text{IC}_X)$ vanishes for $i \notin [-d,d]$ (\cite{BBD},~Section 4.2.4),~by the definition of $\text{Tr}(R\Gamma(u^{(n)}))$ we have an equality

\begin{equation}\label{global_term_correspondence_intersection_complex_2}
\text{Tr}(R\Gamma(u^{(n)}))=\sum_{i=-d}^{d}(-1)^{i} \text{Tr}((c_{2*}c^{*}_1) \circ F^{n*}_{X/\F_q}):H^i(X,\text{IC}_X) \to H^i(X,\text{IC}_X)).
\end{equation}

Here $c_1^*$ and  $c_{2*}$ are the pullback and pushforward induced by (\ref{cohomological_correspondence_lifting_correspondence_2}) and (\ref{cohomological_correspondence_lifting_correspondence_3}) respectively.~Since $c_1$,~$c_2$ and $u$ are defined over $\F_q$,~$c_{2*}c_1^*$ commutes with $F^*_{X/\F_q}$.~Thus the bound (\ref{global_term_correspondence_intersection_complex_1}) is an immediate consequence of Proposition \ref{pushforward_pullback_intersection_cohomology},~and purity of $\text{IC}_X$ (\cite{BBD},~Th\'eor\`eme 5.4.10).

\end{proof}

\section{Local terms along a locally invariant subset}\label{local_terms_invariant_subset}

This section aims to prove Theorem \ref{bound_local_term_invariant_subscheme_introduction},~and along the way obtain a rationality statement.~As mentioned in the introduction we prove a similar bound for the larger class of \textit{essentially proper} correspondences.~We begin by defining them.

Let $k$ be an algebraic closure of the finite field $\F_q$.~Through this section we will be working over schemes of finite type over the field $k$.~Some of the results can be generalized in an obvious way to arbitrary algebraically closed fields.~But we do not need these generalisations for our purposes,~hence do not pursue them here.

\subsection{Essentially proper correspondences}

Let $X/k$ be a \textit{proper} scheme defined over $\F_q$.~Let $c \colon C \to X \times_k X$ be an arbitrary correspondence.

\begin{defn}[Essentially proper correspondence]\label{good_correspondences}
We say $c$ is an \textit{essentially proper} correspondence over $\F_q$,~if it can be factored as 

\begin{equation}\label{good_correspondece_definition}
\xymatrix{ C   \ar[d]_-{\bar{c}_U} \ar[dr]^-{c} &  \\
                  U   \ar@{^{(}->}[r]^-{j}               & X \times_k X             }
\end{equation}

\noindent such that

\begin{enumerate}[(a)]

\item $\bar{c}_U$ is proper.

\item $U=\cup_{i=1}^{r} U_i \times_k U_i$,~where $U_i$'s are open subsets of $X$ defined over $\F_q$ and cover $X$.

\end{enumerate}

\end{defn}

\begin{rmk}\label{proper_good}

Proper correspondences are trivially essentially proper (over any finite subfield containing $\F_q$).

\end{rmk}

The following is our motivation for studying essentially proper correspondences.

Let $\bar{c} \colon \bar{C} \to \bar{X} \times_k \bar{X}$ be a proper correspondence.~Let $Z \subseteq \bar{X}$ be a closed subset of $\bar{X}$ defined over $\F_q$,~which is locally $\bar{c}$-invariant over $\F_q$ (see Remark \ref{locally_invariant_subfield}).~Thus there exists a finite collection of open subsets $\{U_{i}\},~1 \leq i \leq r$ of $\bar{X}$ defined over $\F_q$,~which cover $Z$,~and set theoretic inclusions

\begin{equation}\label{equation_implying_local_invariance}
\bar{c}_2^{-1}(U_i \cap Z) \cap \bar{c}_1(U_i)  \subseteq \bar{c}_1^{-1}(Z) ,
\end{equation}
for every $i$.

Let $U \colonequals \bar{X} \backslash Z$.~Note that $U$ is also defined over $\F_q$.~Clearly $U$ and $U_i$'s together form an open cover of $\bar{X}$.~Let $W \colonequals \bar{c}^{-1}\left ( \left (U \times_k U \right ) \cup \left (\cup_{i=1}^{r}(U_{i} \times_k U_{i} \right ) \right )$.

\begin{lem}\label{neighbourhood_fixed_points}
$W$ is a neighbourhood of fixed points of $\bar{c}^{(n)}$ such that,~$Z$ is $\bar{c}^{(n)}|_W$-invariant for any $n \geq 0$.
\end{lem}

\begin{proof}
Since $\left ( \left (U \times_k U \right ) \cup \left (\cup_{i=1}^{r}(U_{i} \times_k U_{i} \right ) \right )$ is an open neighbourhood of $\Delta^{(n)}_{\bar{X}}$,~$W$ by definition is an open neighbourhood of $\text{Fix}(\bar{c}^{(n)})$.

Now we show that $Z$ is $\bar{c}^{(n)}|_W$-invariant for every $n \geq 0$.~Let $x \in W$ be such that $\bar{c}_2(x) \in Z$.~We need to show that $F^n_{\bar{X}/\F_q}(\bar{c}_1(x)) \in Z$,~for every $n \geq 0$.~Since $Z$ is defined over $\F_q$ it suffices to prove that $\bar{c}_1(X) \in Z$.

Since $x \in W=\bar{c}^{-1}((U \times_k U) \cup (\cup_{i=1}^{r}(U_{i} \times_k U_{i}))$,~$\bar{c}_2(x) \in Z$ implies that $x \in \bar{c}^{-1}(U_i \times U_i)$ for some index $i$.~Thus $\bar{c}_2(x) \in U_i \cap Z$ and $\bar{c}_1(x) \in U_i$.~(\ref{equation_implying_local_invariance}) then implies that $\bar{c}_1(x) \in Z$.

\end{proof}

In light of Lemma \ref{neighbourhood_fixed_points} our discussion in Section \ref{local_term_subscheme_invariant_neighbourhood_fixed_points} applies to $Z$,~and we can make sense of the correspondences $\bar{c}|_{W,Z}$.~Consider the Cartesian diagram

\begin{equation}\label{Cartesian_locally_invariant}
\xymatrixcolsep{4pc} \xymatrix{ W \cap \bar{c}_1^{-1}(Z) \cap \bar{c}_2^{-1}(Z) \ar@{^{(}->}[r] \ar[d] &  \bar{c}_1^{-1}(Z) \cap \bar{c}_2^{-1}(Z) \ar[d] \\
                 \cup_i((U_i \cap Z) \times_k (U_i \cap Z))  \ar@{^{(}->}[r]                   & Z \times_k Z                  }.
\end{equation}

\subsubsection{}\label{correspondence_good_locally_invariant}The correspondence $\bar{c}|_{W,Z}$ is then by definition the map from $W \cap \bar{c}_1^{-1}(Z) \cap \bar{c}_2^{-1}(Z)$ to $Z \times_k Z$.~The Cartesian diagram (\ref{Cartesian_locally_invariant}) trivially implies the following Lemma.

\begin{lem}\label{local_invaraince_good}
$\bar{c}|_{W,Z}$ is an essentially proper correspondence over $\F_q$.
\end{lem}

Moreover since $Z$ is defined over $\F_q$,~we can make sense of $\bar{c}^{(n)}|_{W,Z}$,~and clearly we have

\begin{equation}\label{twist_commutes_restriction}
\bar{c}^{(n)}|_{W,Z}=\bar{c}|_{W,Z}^{(n)},
\end{equation}

\noindent as correspondences from $W \cap \bar{c}_1^{-1}(Z) \cap \bar{c}_2^{-1}(Z)$ to $Z \times_k Z$.~

\subsubsection{}\label{cohomological_correspondence_good_locally_invariant} Now suppose we are given a cohomological correspondence $u$ of a Weil sheaf $\sF$ on $\bar{X}$ lifting the correspondence $\bar{c}$.~Then Lemma \ref{neighbourhood_fixed_points} and the discussion in Section \ref{local_term_subscheme_invariant_neighbourhood_fixed_points} implies that there exists a cohomological corespondence $u^{(n)}|_{W,Z}$ of $\sF|_Z$ lifting the correspondence $c^{(n)}|_{W,Z}$.~Here $\sF|_Z$ is given the natural Weil structure coming from restricting the one on $\sF$.

In particular we can calculate $\text{LT}(u^{(n)}|_Z)$ \textit{using} $W$.~Also note that analogous to (\ref{twist_commutes_restriction}) we have

\begin{equation}\label{twist_commutes_restriction_cohomological}
u^{(n)}|_{W,Z}=(u|_{W,Z})^{(n)}
\end{equation}

\noindent as cohomological correspondences of $\sF|_Z$ lifting $\bar{c}^{(n)}|_{W,Z}=\bar{c}|_{W,Z}^{(n)}$.~In particular

\begin{equation}\label{twist_commutes_restriction_cohomological_local}
\text{LT}(u^{(n)}|_Z)=\text{LT}((u|_{W,Z})^{(n)}).
\end{equation}

Having motivated essentially proper correspondences,~we now establish some basic properties of essentially proper correspondences over $\F_q$.

As before let $c \colon C \to X \times_k X$ be an essentially proper correspondence over $\F_q$.~Recall that $X$ is assumed to be proper and defined over $\F_q$.~In particular we have a  diagram (\ref{good_correspondece_definition}).

\begin{lem}\label{all_twists_good}
$c^{(n)}$ is essentially proper over $\F_q$ for all $n \geq 1$.
\end{lem}

\begin{proof}

Let $F_{l} \colonequals F_{X/\F_q} \times 1_X$ be the partial Frobenius on $X \times_k X$.~The diagram (\ref{good_correspondece_definition}) can be enlarged to the following diagram

\begin{equation}\label{good_correspondece_all_twists_good}
\xymatrix{ C  \ar[d]^{\bar{c}_U} \ar[dr]^-{c}   \\
                  U   \ar@{^{(}->}[r]^-{j}    \ar[d]^-{F^n_{l}|_U}          & X \times_k X      \ar[d]^-{F^{n}_{X/\F_q} \times 1_X}  \\ 
                  U   \ar@{^{(}->}[r]^-{j}               & X \times_k X     }.
\end{equation}

The square is Cartesian because of the condition $(b)$ in Definition \ref{good_correspondences}.~Thus the outer square in (\ref{good_correspondece_all_twists_good}) implies that $c^{(n)}$ is also essentially proper over $\F_q$.

\end{proof}

%\begin{rmk}\label{all_twists_good}
%In the notation of (\ref{correspondence_good_locally_invariant}),~Lemma \ref{local_invaraince_good},~Lemma \ref{all_twists_good} and the equality (\ref{twist_commutes_restriction}) together imply that $\bar{c}^{(n)}|_{W,Z}$ is good for every $n \geq 0$.
%\end{rmk}

The following Lemma is analogous to Lemma \ref{neighbourhood_fixed_points}.

\begin{lem}\label{graph_transpose_contained_in_open}
For any $n \geq 0$,~the graphs of $F^{n}_{X/\F_q}$ and their transpose are contained in $U$ (see Diagram \ref{good_correspondece_definition}).~Thus $\text{Fix}(c^{(n)})$ is proper over $k$.
\end{lem}

\begin{proof}
By definition $U=\cup_{i=1}^{r} U_i \times_k U_i$,~where $U_i$'s are an open cover of $X$ defined over $\F_q$.~Hence the graphs of $F^{n}_{X/\F_q}$ and their transpose are contained in $U$.~Thus the  diagram (\ref{good_correspondece_definition}) implies that $\text{Fix}(c^{(n)})=\bar{c}_U^{-1}(\Delta^{(n)})$.~Hence $\text{Fix}(c^{(n)})$ is proper over $k$.

\end{proof}

Lemma \ref{graph_transpose_contained_in_open} implies that,~if we are given a cohomological correspondence $u$ of a Weil sheaf $\sF$ on $X$ lifting the possibly \textit{non-proper} but essentially proper correspondence $c$,~we can make sense of the local terms $\text{LT}(u^{(n)})$ (Definition \ref{Local_term}).

Now suppose  $\sF_0 \in D^{b}_{\leq w}(X_0,\bar{\Q}_{\ell})$ is a mixed sheaf of weight less than or equal to $w$ on $X_0$.~Suppose that $\sF_0$ is in $~^pD^{\leq a}(X_0,~\bar{\Q}_{\ell})$.~Here $X_0$ is the chosen model of $X$ over $\F_q$.~Let $u$ be a cohomological correspondence of $\sF$ lifting $c$,~an essentially proper correspondence over $\F_q$.

Fix a field isomorphism (say $\tau$) of $\bar{\Q}_{\ell}$ with $\C$.

%Since $c$ is a proper correspondence,~and $Z$ a locally invariant subset we can make sense of $\text{LT}(u^{(n)}|_Z)$ (see Lemma \ref{neighbourhood_fixed_points} and (\ref{local_term_subscheme_invariant_neighbourhood_fixed_points})).

This section aims to obtain the following estimate for the local terms.

\begin{thm}\label{local_term_growth_good_correspondence}

For any $\e>0$,~there exists a natural number $N(\e)$ and a positive real number $M(\tau)$ such that,~for any $n \geq N(\e)$

\begin{equation}\label{local_term_growth_good_correspondence_1}
|\text{LT}(u^{(n)})| \leq M(\tau)q^{n(\frac{(w+a+\text{dim}(X))}{2}+\e)}.
\end{equation}
Here the norm on the left is with respect to the chosen isomorphism $\tau$.
\end{thm}

Theorem \ref{local_term_growth_good_correspondence} follows easily from the Lefschetz-Verdier trace formula if one assumes $c$ is proper (and not just essentially proper).~In light of Lemma \ref{local_invaraince_good} and the equality (\ref{twist_commutes_restriction_cohomological_local}),~Theorem \ref{bound_local_term_invariant_subscheme_introduction} is an immediate consequence of Theorem \ref{local_term_growth_good_correspondence} (with $Z$ playing the role of $X$ above).

Now we show that Theorem \ref{bound_local_term_invariant_subscheme_introduction} combined with the results of the earlier sections implies $(1)$ and $(3)$ in Theorem \ref{non_uniform_estimate}.

\subsection{Theorem \ref{bound_local_term_invariant_subscheme_introduction} implies Theorem \ref{non_uniform_estimate},~$(1)$ and $(3)$}\label{enough_to_bound_local}

\begin{proof}
Recall that we have a correspondence $c \colon C \to X \times_k X$,~together with a compactification $\bar{c} \colon \bar{C} \subseteq \bar{X} \times_k \bar{X}$ of $c$.~Lemma \ref{basic_reduction},~Proposition \ref{reduction_generically_etale} and Proposition \ref{reduction_to_locally_invariant_boundary} imply that we can assume the following

\begin{enumerate}[(a)]

\item $X$ is smooth over $k$.

\item $\bar{c}_2$ restricted to $\bar{c}^{-1}(X \times_k X)=C$ is \'etale.

\item $\bar{c}$ leaves $Z \colonequals \bar{X} \backslash X$ locally invariant (see Definition \ref{locally_invariant_closed}) over $\F_q$.

\end{enumerate}

Since $Z$ is defined over $\F_q$,~$(c)$ implies that $Z$ is also locally $\bar{c}^{(n)}$-invariant over $\F_q$ for all $n \geq 1$.~Thus \cite{Varshavsky},~Corollary 2.2.4 implies that there exists an integer $N'$ such that for all $n \geq N'$,~$\bar{c}^{(n)}$ is contracting in a neighbourhood of fixed points around $Z$ (see Definition \ref{contracting_fixed_points}) and $\bar{c}^{(n)}|_X$ is contracting near every closed point of $X$ in a neighbourhood of fixed points.

Let $u \colon \bar{c}_1^*\text{IC}_{\bar{X}} \to \bar{c}_2^{!} \text{IC}_{\bar{X}}$ be the cohomological correspondence defined in Corollary \ref{cohomological_correspondence_lifting_correspondence}.~Since $\text{Fix}(\bar{c}^{(n)})$ is proper over $k$,~we can apply Corollary \ref{shape_of_trace_formula} to the cohomological correspondence $u^{(n)}$ for any $n \geq N'$.

Corollary \ref{shape_of_trace_formula} implies that for $n \geq N'$,~$\text{Fix}(c^{(n)})$ is finite.~Moreover we have for any $n \geq N'$,

\begin{equation}\label{local_term_decomposition_proof}
 \text{LT}(u^{(n)})=\sum_{\beta \in \text{Fix}(\bar{c}^{(n)}|_X) } \text{Tr}(u^{(n)}_{\beta})+ \text{LT}(u^{(n)}|_Z).
\end{equation}

Since $\bar{c}^{-1}(X \times X)=C$ and $u|_{X}$ is the correspondence (\ref{trivial_cohomological_correspondence}),~Lemma \ref{naive_local_term_calculate} implies that 

\begin{equation}\label{main_theorem_estimate_proof_1}
\sum_{\beta \in \text{Fix}(\bar{c}^{(n)}|_{X})} \text{Tr}(u^{(n)}_{\beta})=\#\text{Fix}(c^{(n)})(k).
\end{equation}

Since the restriction of $\text{IC}_{\bar{X}}$ to $Z$ is of weight less than or equal to $\text{dim}(X)$ (\cite{BBD},~5.1.14 (i)) and belongs to $~^pD^{\leq -1}$ (\cite{BBD},~Corollaire 1.4.24),~Theorem \ref{bound_local_term_invariant_subscheme_introduction} implies that there exists a positive integer $N''$ such that for all $n \geq N''$,

\begin{equation}\label{main_theorem_estimate_proof_2}
| \text{LT}(u^{(n)}|_Z))| \leq M_{\text{loc}}(\tau)q^{n(\text{dim}(X)-\frac{1}{2})}
\end{equation}

\noindent for some positive real number $M_{\text{loc}}(\tau)$ (see Remark \ref{perversity_drop_remark}).

Finally note that Lefschetz-Verdier trace formula (Corollary \ref{Lefschetz-Verdier}) and Proposition \ref{global_term_correspondence_intersection_complex} together imply that 

\begin{equation}\label{main_theorem_estimate_proof_3}
|\text{LT}(u^{(n)})- \delta q^{n \text{dim}(X)}|  \leq M_{\text{glo}}(\tau)q^{n(\text{dim}(X)-\frac{1}{2})},
\end{equation}

\noindent for some positive real number $M_{\text{glo}}(\tau)$ (depending on $\tau$).

Thus combining (\ref{local_term_decomposition_proof}),~(\ref{main_theorem_estimate_proof_1}),~(\ref{main_theorem_estimate_proof_2}) and (\ref{main_theorem_estimate_proof_3}) we obtain the bound

\begin{equation}\label{main_theorem_estimate_proof_4}
|\#\text{Fix}(c^{(n)})(k)-\delta q^{n \text{dim}(X)}| \leq M(\tau)q^{n(\text{dim}(X)-\frac{1}{2})},
\end{equation}

\noindent for all $n \geq N=\text{max}\{N',N''\}$ with $M(\tau)=M_{\text{loc}}(\tau)+M_{\text{glo}}(\tau)$.

\end{proof}

\begin{rmk}\label{perversity_drop_remark}
The bound (\ref{main_theorem_estimate_proof_2}) is a consequence of \textit{both} the dimension and perversity dropping when restricted to the boundary.~The bound in (\ref{local_term_growth_good_correspondence_1}) has an error term of $\e$ and if we did not have the perversity drop,~Theorem \ref{bound_local_term_invariant_subscheme_introduction} would give us a weaker bound with the local term growing as $q^{n(\text{dim}(X)-1/2+\e)}$,~which is clearly insufficient for our purposes.
\end{rmk}

The rest of the article is devoted to the proof of Theorem \ref{local_term_growth_good_correspondence}.

\subsection{The pairing}\label{the_pairing}

%
%We begin by reinterpreting cohomological correspondences in a way which will be useful for us.
%
%\begin{lem}\label{alternative_cohomological_correspondece}
%There exist a natural isomorphism of $\bar{\Q}_{\ell}$ vector spaces 
%\begin{equation}\label{reinterpreting_cohomological_correspondence}
%\text{Hom}(c_1^*\sF,c_2^!\sF) \simeq \text{Hom}(\sF \boxtimes \mathbb{D}_X \sF,c_* K_C).
%\end{equation}
%\end{lem}
%
%\begin{proof}
%
%Clearly we have 
%
%%(\cite[\href{https://stacks.math.columbia.edu/tag/0B6E}{Tag 0B6E}]{stacks-project})
%
%\begin{equation}\label{alternative_cohomological_correspondece_1}
%\text{Hom}(c_1^*\sF,c_2^!\sF) \simeq \text{Hom}(\bar{\Q}_{\ell,C},\mathcal{RH}\textit{om}(c_1^*\sF,c_2^!\sF)),
%\end{equation}
%
%\noindent which by Illusie's isomorphism (\ref{illusie_isomorphism}) is isomorphic to 
%
%\begin{equation}\label{alternative_cohomological_correspondece_2}
%\text{Hom}(\bar{\Q}_{\ell,C},c^!(\mathbb{D}_X\sF \boxtimes \sF)).
%\end{equation}
%
%Further using adjunction between $(c_!,c^!)$ it is clear that 
%\begin{equation}\label{alternative_cohomological_correspondece_3}
%\text{Hom}(\bar{\Q}_{\ell,C},c^!(\mathbb{D}_X\sF \boxtimes \sF)) \simeq \text{Hom} (c_!\bar{\Q}_{\ell,C},\mathbb{D}_X\sF \boxtimes \sF),
%\end{equation}
%which in turn is isomorphic to 
%
%\begin{equation}\label{alternative_cohomological_correspondece_4}
%\text{Hom}(\mathbb{D}_{X \times_k X}(\mathbb{D}_X\sF \boxtimes \sF),~c_*K_C)
%\end{equation}
%
%\noindent by Verdier duality.~The Lemma is then a consequence of the isomorphism (\ref{duality_distributes_external_products}).
% 
%\end{proof}

Let $c:C \to X \times_k X$ be an essentially proper correspondence over $\F_q$,~with the choice of a factorization as in (\ref{good_correspondece_definition}).~Let $Z:=(X \times_k X \backslash U)_{\text{red}}$ the complimentary closed subscheme and $i:Z \hookrightarrow X \times_k X$ be the corresponding closed immersion.

Choose an arbitrary compactification $\bar{c}:\bar{C} \to X \times_k X$ of $c$.~Let $\partial \bar{C}$ be the reduced complement of $C$ in $\bar{C}$.~Let $i_{\partial \bar{C}}$ and $\bar{c}_Z$ be the induced maps from $\partial \bar{C}$ to $\bar{C}$ and $Z$ respectively.

Thus we have a diagram

\begin{equation}\label{the_set_up}
\xymatrix{ C \ar@{^{(}->}[r]^-{j_C} \ar[d]^-{\bar{c}_U} & \bar{C} \ar[d]_-{\bar{c}}  & \partial \bar{C} \ar@{_{(}->}[l]_-{i_{\partial \bar{C}}} \ar[d]_-{\bar{c}_Z} \\
                  U   \ar@{^{(}->}[r]^-{j}               & X \times_k X              &  Z    \ar@{_{(}->}[l]_-{i}           }
\end{equation}

\noindent where both the squares are Cartesian (upto nilpotents) as a consequence of Lemma \ref{compactification_Cartesian}.~On $\bar{C}$ we have an exact triangle 

\begin{equation}\label{canonical_exact_above}
\xymatrix{j_{C!}\bar{\Q}_{\ell} \ar[r] & j_{C*} \bar{\Q}_{\ell} \ar[r] & i_{\partial \bar{C}*}i_{\partial \bar{C}}^*j_{C*} \bar{\Q}_{\ell} \ar[r]^-{+1}&}.
\end{equation}

Pushing forward this triangle to $X \times_k X$ via $\bar{c}_*=\bar{c}_!$,~we obtain an exact triangle

\begin{equation}\label{pushoforward_canonical_exact_above}
\xymatrix{c_{!} \bar{\Q}_{\ell} \ar[r] & c_* \bar{\Q}_{\ell} \ar[r] & i_{*}\bar{c}_{Z*}i_{\partial \bar{C}}^*j_{C*} \bar{\Q}_{\ell} \ar[r]^-{+1}&}
\end{equation}

\noindent on $X \times_k X$.

Clearly $\text{Hom}(i_* \hspace{1mm} ,~j_* \hspace{1mm}) \simeq 0$.~Thus applying the cohomological functor $\text{Hom}(\hspace{1mm}, j_{*}j^*( \mathbb{D}_X  \sF \boxtimes \sF))$ to the triangle (\ref{pushoforward_canonical_exact_above}),~we get an isomorphism

\begin{equation}\label{isomorphic_Hom_Spaces}
\text{Hom}(c_!\bar{\Q}_{\ell}, j_{*}j^*( \mathbb{D}_X  \sF \boxtimes \sF)) \simeq \text{Hom}(c_*\bar{\Q}_{\ell},j_{*}j^*( \mathbb{D}_X  \sF \boxtimes \sF)).
\end{equation}

The natural map $\sF \boxtimes \mathbb{D}_X \sF \to j_{*}j^*( \mathbb{D}_X  \sF \boxtimes \sF)$ combined with the isomorphisms (\ref{illusie_isomorphism}) and (\ref{isomorphic_Hom_Spaces}) gives a natural map

\begin{equation}\label{cohomological_correspondence_lift_equation}
\text{Hom}(c_1^*\sF,c_2^!\sF) \to \text{Hom}(c_*\bar{\Q}_{\ell}, j_{*}j^*(\mathbb{D}_X \sF \boxtimes \sF)).
\end{equation}

%The map in (\ref{cohomological_correspondence_lift_equation}) can be understood as follows.~Fix a cohomological correspondence $u$ of $\sF$ lifting $c$,~or equivalently by Lemma \ref{alternative_cohomological_correspondece} a morphism (also denoted by $u$) from $\sF \boxtimes \mathbb{D}_X \sF \to c_*K_C$.~Consider the following diagram
%
%\begin{equation}\label{unique_lift_cohomological_correspondence}
%\xymatrix{i_*i^*(\sF \boxtimes \mathbb{D}_X \sF)[-1] \ar[d]^-{+1} &  i_{*}\bar{c}_{Z*}i_{\partial \bar{C}}^*j_{C*}K_C[-1] \ar[d]^-{+1} \\
%               j_{!}j^*(\sF \boxtimes \mathbb{D}_X \sF) \ar[d]^{\alpha} \ar@{-->}[r]^-{\tilde{u}} & c_!K_C \ar[d] \\
%               \sF \boxtimes \mathbb{D}_X \sF \ar[d] \ar[r]^-{u} & c_*K_C \ar[d] \\
%               i_*i^*(\sF \boxtimes \mathbb{D}_X \sF) &  i_{*}\bar{c}_{Z*}i_{\partial \bar{C}}^*j_{C*}K_C }.
%\end{equation}
%
%The isomorphism (\ref{isomorphic_Hom_Spaces}) implies that there exists an unique arrow $\tilde{u} \colon j_{!}j^*(\sF \boxtimes \mathbb{D}_X \sF) \to c_!K_C$ making the central square in (\ref{unique_lift_cohomological_correspondence}) commute.~Moreover $\tilde{u}$ depends only on $u \circ \alpha$.~This $\tilde{u}$ is the image of $u$ under (\ref{cohomological_correspondence_lift_equation}).

By duality and properness of $X \times_k X$,~we have isomorphisms 

\begin{equation}\label{first_basic_isomorphism}
\text{Hom}( j_{*}j^*(\mathbb{D}_X \sF \boxtimes \sF),K_{X \times_k X}) \simeq \text{Hom}(\bar{\Q}_{\ell}, j_{!}j^*(\sF \boxtimes \mathbb{D}_X \sF)) \simeq H^{0}_c(U,~j^*(\sF \boxtimes \mathbb{D}_X \sF)).
\end{equation}

%\noindent and 
%
%\begin{equation}\label{second_basic_isomorphism}
%\text{Hom}(\bar{\Q}_{\ell},c_!K_C) \simeq H^{0}_c(C,K_C).
%\end{equation}

Also, there exists a  pairing

\begin{equation}\label{basic_pairing}
 \text{Hom}(c_*\bar{\Q}_{\ell}, j_{*}j^*(\mathbb{D}_X \sF \boxtimes \sF)) \otimes \text{Hom}( j_{*}j^*(\mathbb{D}_X \sF \boxtimes \sF),K_{X \times_k X}) \to \text{Hom}(c_*\bar{\Q}_{\ell},K_{X \times_k X}),
\end{equation}

\noindent given by the composition of maps.~Moreover by adjunction we have a natural map 

\begin{equation}\label{Final_isomorphism_pairing}
\text{Hom}(c_*\bar{\Q}_{\ell},K_{X \times_k X}) \to \text{Hom}(\bar{\Q}_{\ell},K_{X \times_k X}) \simeq H^0(X \times_k X, K_{X \times_k X}).
\end{equation}

Thus combining (\ref{cohomological_correspondence_lift_equation}),~(\ref{first_basic_isomorphism}),~(\ref{Final_isomorphism_pairing}) and (\ref{basic_pairing}) we get a natural pairing

\begin{equation}\label{main_pairing}
\Phi \colon \text{Hom}(c_1^*\sF,c_2^!\sF) \otimes_{\bar{\Q}_{\ell}} H^{0}_c(U,~j^*(\sF \boxtimes \mathbb{D}_X \sF)) \to H^{0}(X \times_k X, K_{X \times_k X}).
\end{equation}

In particular if we fix a cohomological correspondence $u$ of $\sF$ lifting $c$,~we get a \textit{linear functional} on $H^{0}_c(U,~j^*(\sF \boxtimes \mathbb{D}_X \sF))$ given by 

\begin{equation}\label{linear_functional}
\Phi_u(\beta) \colonequals \text{Tr}_{X \times_k X}(\Phi(u \otimes \beta)),
\end{equation}

\noindent where $\text{Tr}_{X \times_k X}$ is the natural trace map on $H^{0}(X \times_k X,K_{X \times_k X})$ (recall that $X$ is proper over $k$).

\subsection{Trace along the graphs of Frobenius}

Let $k$ be an algebraic closure of $\F_q$.~Let $c:C \to X \times_k X$ be a correspondence defined over $\F_q$.~Let $\sF$ be a Weil sheaf on $X$.~Let $u$ be a cohomological correspondence of $\sF$ lifting $c$.~For each $n$ we have a Cartesian diagram

\begin{equation}\label{graph_Frobenius_local_terms}
\xymatrix{\text{Fix}(c^{(n)}) \ar[r] \ar@{^{(}->}[d]_-{\Delta^{(n)'}} \ar@/^2pc/[rr]^-{c^{(n)'}} & X \ar@{^{(}->}[d]_-{\Delta^{(n)}} \ar[r]^{F^n_{X/\F_q}} & X \ar@{^{(}->}[d]^-{\Delta} \\
                  C \ar[r]^-{c} \ar@/_2pc/[rr]_-{c^{(n)}} & X \times_k X                       \ar[r]^-{F^n_{X/\F_q} \times 1_X}         &  X \times_k X }.
\end{equation}

We briefly recall the map  (see \ref{trace_Along_graph_Frobenius_Appendix})

\begin{equation}\label{trace_Along_graph_Frobenius}
\mathcal{T}\textit{r}^{'(n)}_c \colon \text{Hom}(c_{2!}c_1^*\sF,\sF) \to H^0(\text{Fix}(c^{(n)}),K_{\text{Fix}(c^{(n)}})
\end{equation}

\noindent constructed in the Section \ref{trace_maps} here.~Let $u$ be a cohomological correspondence of $\sF$ lifting $c$.~Thus $u$ corresponds to a map

\begin{equation}\label{two_ways_Frobenius_twist_same_Varshavsky_3}
c_!\bar{\Q}_{\ell} \to (\mathbb{D}_X \sF \boxtimes \sF),
\end{equation}

\noindent or equivalently to 

\begin{equation}\label{two_ways_Frobenius_twist_same_Varshavsky_3'}
\bar{\Q}_{\ell} \to c^!(\mathbb{D}_X \sF \boxtimes \sF),
\end{equation}

Composing (\ref{two_ways_Frobenius_twist_same_Varshavsky_3}) and (\ref{Frobenius_cohomology_class_3})  we get a map

\begin{equation}\label{two_ways_Frobenius_twist_same_Varshavsky_4}
c_!\bar{\Q}_{\ell} \to \Delta^{(n)}_{*}K_{X}.
\end{equation}

Applying $c^!$ to (\ref{Frobenius_cohomology_class_1}) and composing with (\ref{two_ways_Frobenius_twist_same_Varshavsky_3'}) we get a map (denoted by $u^{(n)}$)

\begin{equation}\label{cohomological_correspondence_twist_1}
u^{(n)} \colon \bar{\Q}_{\ell} \to c^{(n)!}(\mathbb{D}_X \sF \boxtimes \sF),
\end{equation}

Moreover, we have natural isomorphisms 

\begin{equation}\label{two_ways_Frobenius_twist_same_Varshavsky_5}
\text{Hom}(c_!\bar{\Q}_{\ell},\Delta^{(n)}_{*}K_{X}) \simeq \text{Hom}(\bar{\Q}_{\ell},c^!\Delta^{(n)}_{*}K_{X}),
\end{equation}

\noindent and 

\begin{equation}\label{two_ways_Frobenius_twist_same_Varshavsky_6}
 \text{Hom}(\bar{\Q}_{\ell},c^!\Delta^{(n)}_{*}K_{X}) \simeq \text{Hom}(\bar{\Q}_{\ell} ,\Delta^{(n)'}_*K_{\text{Fix}(c^{(n)})}) \simeq H^0(\text{Fix}(c^{(n)}),K_{\text{Fix}(c^{(n)})}).
\end{equation}

Here (\ref{two_ways_Frobenius_twist_same_Varshavsky_5}) comes from adjunction,~and (\ref{two_ways_Frobenius_twist_same_Varshavsky_6}) from base change along the left inner square of ({\ref{graph_Frobenius_local_terms}).~Combining (\ref{two_ways_Frobenius_twist_same_Varshavsky_4}), (\ref{two_ways_Frobenius_twist_same_Varshavsky_5}) and (\ref{two_ways_Frobenius_twist_same_Varshavsky_6}) we get an element in $H^0(\text{Fix}(c^{(n)}),K_{\text{Fix}(c^{(n)})})$,~which we call $\mathcal{T}\textit{r}^{'(n)}_c(u)$.

The maps $\mathcal{T}\textit{r}_c^{'(n)}$ are similar to $\mathcal{T}\textit{r}_c$ in (\ref{Varashavsky_Trace}),~but adapted to the graphs of Frobenius.~In fact these two trace maps are compatible in an obvious way.

\begin{lem}\label{two_ways_Frobenius_twist_same_Varshavsky}
$\mathcal{T}\textit{r}_c(u^{(n)})=\mathcal{T}\textit{r}_c^{'(n)}(u)$.
\end{lem}

\begin{proof}

For ease of writing we set $\sG=\mathbb{D}_X \sF \boxtimes \sF$.~ Let us recall Varshavsky's recipe (see Section \ref{trace_maps}) to compute $\mathcal{T}\textit{r}_c(u^{(n)})$.~First we apply $c^{(n)!}$ to (\ref{evaluation_trace}) to get a map

\begin{equation}\label{two_ways_Frobenius_twist_same_Varshavsky_0}
c^{(n)!}\text{ev}_{\sF} \colon c^{(n)!}(\sG) \to c^{(n)!}\Delta_*K_X.
\end{equation}

Composing (\ref{two_ways_Frobenius_twist_same_Varshavsky_0}) with the map $u^{(n)}$ in (\ref{cohomological_correspondence_twist_1}),~we get a map 

\begin{equation}\label{two_ways_Frobenius_twist_same_Varshavsky_1'}
c^{(n)!}\text{ev}_{\sF} \circ u^{(n)} \colon \bar{\Q}_{\ell} \to c^{(n)!}\Delta_*K_X.
\end{equation}

Base change applied to the right hand square of (\ref{graph_Frobenius_local_terms}) gives us a natural isomorphism

\begin{equation}\label{two_ways_Frobenius_twist_same_Varshavsky_2'}
\text{Hom}(\bar{\Q}_{\ell},~c^{(n)!}\Delta_*K_X) \simeq \text{Hom}(\bar{\Q}_{\ell},~c^{!}\Delta^{(n)}_*K_X).
\end{equation}

Finally note that  (\ref{two_ways_Frobenius_twist_same_Varshavsky_6}) gives an isomorphism of $\text{Hom}(\bar{\Q}_{\ell},~c^{!}\Delta^{(n)}_*K_X)$ with $H^0(\text{Fix}(c^{(n)}),K_{\text{Fix}(c^{(n)}})$. The element in $H^0(\text{Fix}(c^{(n)}),K_{\text{Fix}(c^{(n)}})$ corresponding to (\ref{two_ways_Frobenius_twist_same_Varshavsky_1'}) via the isomorphisms (\ref{two_ways_Frobenius_twist_same_Varshavsky_2'}) and (\ref{two_ways_Frobenius_twist_same_Varshavsky_6}) is by definition $\mathcal{T}\textit{r}_c(u^{(n)})$.~Thus $\mathcal{T}\textit{r}_c(u^{(n)})$ apriori occurs as an element in $\text{Hom}(\bar{\Q}_{\ell},~c^{(n)!}\Delta_*K_X)$ (say $u_{\Delta^{(n)}}$),~which is then considered as an element in $H^0(\text{Fix}(c^{(n)}),K_{\text{Fix}(c^{(n)}})$ via the isomorphisms (\ref{two_ways_Frobenius_twist_same_Varshavsky_2'}) and (\ref{two_ways_Frobenius_twist_same_Varshavsky_6}).

%Consider the commutative diagram
%
%\begin{equation}\label{two_ways_Frobenius_twist_same_Varshavsky_8}
%\xymatrix@C=5pt{\text{Hom}(\bar{\Q}_{\ell},c^{(n)!}\sG) \ar@<-15.0ex>[d]^-{c_!} \otimes_{\bar{\Q}_{\ell}}  \text{Hom}(c^{(n)!}\sG,c^{(n)!}\Delta_*K_X) \ar@<15.0ex>[d]^-{c_!} \ar[r]^-{\psi_1} & \text{Hom}(\bar{\Q}_{\ell},c^{(n)!}\Delta_*K_X) \ar[d]^-{c_!}  \\
%                \text{Hom}(c_!\bar{\Q}_{\ell},c_!c^{(n)!}\sG) \ar@<-15.0ex>[d]^-{(A)}  \otimes_{\bar{\Q}_{\ell}}  \text{Hom}(c_!c^{(n)!}\sG,c_!c^{(n)!}\Delta_*K_X) \ar@<15.0ex>[d]^-{(A)+\text{BC}} \ar[r]^-{\psi_2} &  \text{Hom}(c_!\bar{\Q}_{\ell},c_!c^{(n)!}\Delta_*K_X) \ar[d]^-{(A)+\text{BC}} \\
%                \text{Hom}(c_!\bar{\Q}_{\ell},(F^n_{X/\F_q} \times 1_X)^{!}\sG)  \otimes_{\bar{\Q}_{\ell}}  \text{Hom}((F^n_{X/\F_q} \times 1_X)^{!}\sG,\Delta^{(n)}_*K_X)  \ar[r]^-{\psi_3} & \text{Hom}(c_!\bar{\Q}_{\ell},,\Delta^{(n)}_*K_X)},
% \end{equation}
% 
%where the map along the rows is given by composition.~The maps 
%
%By construction of $\mathcal{T}\textit{r}_c(u^{(n)})$,~under the isomorphism (\ref{two_ways_Frobenius_twist_same_Varshavsky_6}) it corresponds to $\psi_1(u^{(n)} \otimes c^{(n)!}\text{ev}_{\sF}) \in  \text{Hom}(\bar{\Q}_{\ell},c^{!}\Delta^{(n)}_*K_X)$.
%

On the other hand $\mathcal{T}\textit{r}^{'(n)}_c(u)$ naturally exists as an element in $\text{Hom}(c_!\bar{\Q}_{\ell},\Delta^{(n)}_{*}K_{X})$ (say $u''_{\Delta^{(n)}}$),~which is then considered as an element in $H^0(\text{Fix}(c^{(n)}),K_{\text{Fix}(c^{(n)}})$ via the isomorphisms (\ref{two_ways_Frobenius_twist_same_Varshavsky_5}) and (\ref{two_ways_Frobenius_twist_same_Varshavsky_6}).

To compare these elements we consider the following commutative diagram

\begin{equation}\label{two_ways_Frobenius_twist_same_Varshavsky_2''}
\xymatrix@C=19pt{u_{\Delta^{(n)}} \in \text{Hom}(\bar{\Q}_{\ell},~c^{(n)!}\Delta_*K_X) \ar[d]_-{\psi}^-{\cong (A)} \ar[r]^-{(\ref{two_ways_Frobenius_twist_same_Varshavsky_2'})}_-{\cong (\text{BC})} & \text{Hom}(\bar{\Q}_{\ell},~c^{!}\Delta^{(n)}_*K_X) \ar[d]_-{(\ref{two_ways_Frobenius_twist_same_Varshavsky_5})}^-{\cong (A)} \ar[r]^-{\cong}_-{(\ref{two_ways_Frobenius_twist_same_Varshavsky_6})} & H^0(\text{Fix}(c^{(n)}),K_{\text{Fix}(c^{(n)}}) \\
                 \text{Hom}(c_!\bar{\Q}_{\ell},(F^n_{X/\F_q} \times 1_X)^!\Delta_*K_X) \ar[r]_-{\psi''}^-{\cong (\text{BC})} & \text{Hom}(c_!\bar{\Q}_{\ell},\Delta^{(n)}_*K_X) \ni u''_{\Delta^{(n)}}& }.
  \end{equation}

Our discussion above implies that the image of $u_{\Delta^{(n)}}$ along the top row is $\mathcal{T}\textit{r}_c(u^{(n)})$,~while the image of $u''_{\Delta^{(n)}}$ under the maps (\ref{two_ways_Frobenius_twist_same_Varshavsky_5}) and (\ref{two_ways_Frobenius_twist_same_Varshavsky_6}) in (\ref{two_ways_Frobenius_twist_same_Varshavsky_2''}) is $\mathcal{T}\textit{r}^{''}_c(u^{(n)})$.~Thus it suffices to show $\psi(u_{\Delta^{(n)}})=\psi''^{-1}(u''_{\Delta^{(n)}})$ as elements in $\text{Hom}(c_!\bar{\Q}_{\ell},(F^n_{X/\F_q} \times 1_X)^!\Delta_*K_X)$.

Consider the diagram

\begin{equation}\label{two_ways_Frobenius_twist_same_Varshavsky_7}
\xymatrix@C=50pt{ c_!\bar{\Q}_{\ell} \ar[r]^-{c_!u} \ar@/^2pc/[rr]^-{c_!u^{(n)}} & c_!c^!\sG \ar[d]^-{\alpha} \ar[r]^-{\cong}_-{(c_!c^!F^{-1}_{\text{End}(\sF)})} & c_!c^{(n)!}\sG \ar[d]^-{\beta} \ar[r]^-{c_!(c^{(n)!}\text{ev}_{\sF})}  & c_!c^{(n)!} \Delta_*K_X \ar[d]^-{\gamma} \\
                                                                 & \sG  \ar[r]^-{\cong}_-{(F^{-1}_{\text{End}(\sF)})} & (F^n_{X/\F_q} \times 1_X)^!\sG \ar[r]_-{(F^n_{X/\F_q} \times 1_X)^!\text{ev}_{\sF}} & (F^n_{X/\F_q} \times 1_X)^!\Delta_*K_X \ar[d]^-{(\text{BC})} \\
                                                                 & & & \Delta^{(n)}_*K_X }.
\end{equation}

Here $\alpha$,~$\beta$ and $\gamma$ are induced by adjunction between $(c_!,c^!)$,~$(\text{BC})$ denotes the map induced by base change along the right square in (\ref{graph_Frobenius_local_terms}),~and $F^{-1}_{\text{End}(\sF)}$ is inverse to the isomorphism (\ref{Frobenius_cohomology_class_1}).~By functoriality of adjunction $c_!c^! \to 1$,~the squares commute.

Note that by definition of $u^{(n)}$ (see (\ref{cohomological_correspondence_twist_1})),~the composition $c_!c^!F^{-1}_{\text{End}(\sF)} \circ c_!u$ is $c_!u^{(n)}$.~Thus by definition of the adjunction map $\psi$ in (\ref{two_ways_Frobenius_twist_same_Varshavsky_2''}),

\begin{equation}\label{two_ways_Frobenius_twist_same_Varshavsky_2'''}
\psi(u_{\Delta^{(n)}})=\gamma \circ c_!(c^{(n)!}\text{ev}_{\sF}) \circ c_!u^{(n)}.
\end{equation}

On the other hand by definition the composition $\alpha \circ c_!u$ is the map in (\ref{two_ways_Frobenius_twist_same_Varshavsky_3}),~and the composition $\text{BC} \circ (F^n_{X/\F_q} \times 1_X)^!\text{ev}_{\sF} \circ F^{-1}_{\text{End}(\sF)}$ is the map in (\ref{Frobenius_cohomology_class_3}).~Hence $u''_{\Delta^{(n)}}$ is the composition of these two maps.~Thus we have $\psi''^{-1}(u''_{\Delta^{(n)}})=(F^n_{X/\F_q} \times 1_X)^!\text{ev}_{\sF} \circ F^{-1}_{\text{End}(\sF)} \circ \alpha \circ c_!u$,~which by commutativity of (\ref{two_ways_Frobenius_twist_same_Varshavsky_7}) equals $\gamma \circ c_!(c^{(n)!}\text{ev}_{\sF}) \circ c_!u^{(n)}$,~ and hence equals $\psi(u_{\Delta^{(n)}})$ by (\ref{two_ways_Frobenius_twist_same_Varshavsky_2'''}).

\end{proof}

\subsection{Local terms using the pairing (\ref{main_pairing})}

In this section we shall use the pairing (\ref{main_pairing}) to give a formula for computing the local terms of a cohomological correspondence lifting an essentially proper correspondence.~Through out this section we fix an essentially proper correspondence (defined over $\F_q$),~$c: C \to X \times_k X$ together with a diagram as in (\ref{good_correspondece_definition}).

%We begin by briefly recalling some generalities on cohomology with supports.~Let $Y$ be any closed subscheme of $X \times_k X$ contained in $U$.~As before $Z$ is the complement of $U$ in $X$,~with the reduced structure.~Let $i_{Y,U}$ and $i_Y$ denote the inclusion of $Y$ into $U$ and $X \times_k X$ respectively.~Thus one has a diagram
%
%\begin{equation}\label{supports_first_diagram}
%\xymatrix{Y  \ar@{^{(}->}[r]^{i_{Y,U}} & U \ar@{^{(}->}[r]^-{j} & X \times_k X & Z \ar@{_{(}->}[l]_-{i}}.
%\end{equation}
%
%The composite $ j \circ i_{Y,U}$ equals $i_Y$ by definition and $Y \cap Z=\emptyset$.~For any sheaf $\sG$ on $X \times_k X$,~by $H^p_Y(X \times_k X,~\sG)$ we mean the cohomology $H^{p}(Y,i^{!}_Y\sG)$.~Similarly for any sheaf $\sG$ on $U$ by $H^p_Y(U,\sG)$ we mean the cohomology $H^{p}(Y,~i_{Y,U}^{!}\sG)$.~We have a natural isomorphism
%
%\begin{equation}\label{cohomology_with_supports_using_hom}
%H^{0}_Y(X \times_k X,\sG) \simeq \text{Hom}(i_{Y*}\bar{\Q}_{\ell},\sG).
%\end{equation}
%
%Since $\text{Hom}(i_{Y*}\hspace{1mm},~i_{*} \hspace{1mm}) \simeq 0$,~applying the cohomological functor $\text{Hom}(i_{Y*}\bar{\Q}_{\ell},~\hspace{1mm})$ to the triangle
%
%
%\begin{equation}\label{triangle_cohomology_supports}
%\xymatrix{j_!j^* \sG \ar[r] &\sG \ar[r] & i_{*}i^*\sG \ar[r]^-{+1}&},
%\end{equation}
%
%\noindent and using (\ref{cohomology_with_supports_using_hom}) we obtain an natural isomorphism
%
%\begin{equation}\label{first_specific_isomorphism_supports}
%H^0_{Y}(X \times_k X,~\sG) \simeq H^0_Y(X \times_k X,~j_{!}j^*\sG). 
%\end{equation}

Now consider the following natural map
%
%\begin{equation}\label{cohomology_class_Frobenius}
%[\Delta^{(n)}]:\xymatrix{\Delta^{(n)}_* \bar{\Q}_{\ell} \ar[r]^-{(\text{ev}^{(n)}_{\sF})^{\vee}} & \mathbb{D}_{X \times_k X}(\mathbb{D}_X \sF  \boxtimes \sF ) \ar[r]^-{\cong} & \sF \boxtimes \mathbb{D}_X \sF}.
%\end{equation}
%
%Here $(\text{ev}^{(n)}_{\sF})^{\vee}$ is the Verdier dual of $\text{ev}^{(n)}_{\sF}$ (See (\ref{Frobenius_cohomology_class_3})).~Clearly $[\Delta^{(0)}]=[\Delta]$ (see \ref{cohomology_class_diagonal}).
\begin{equation}\label{global_cohomology_classes}
H^0_{\Delta^{(n)}}(X \times_k X,~\sF \boxtimes \mathbb{D}_X \sF) \simeq H^0_{\Delta^{(n)}}(X \times_k X,~j_!j^*(\sF \boxtimes \mathbb{D}_X \sF)) \to H^0_c(U, j^*(\sF \boxtimes \mathbb{D}_X \sF)),
\end{equation}

\noindent here the first isomorphism follows from Lemma \ref{graph_transpose_contained_in_open} and the isomorphism (\ref{first_specific_isomorphism_supports}),~and the second map is the forget supports map.

\subsubsection{The cohomology classes}\label{The_cohomology_Classes}
Let $[\Delta^{(n)}] \in  H^{0}_{\Delta^{(n)}}(X \times_k X,\sF \boxtimes \mathbb{D}_X \sF)$ be the cohomology class obtained by taking the Verdier dual of (\ref{Frobenius_cohomology_class_3}).~By abuse of notation we shall also denote the image of $[\Delta^{(n)}]$ under the morphism (\ref{global_cohomology_classes}) by $[\Delta^{(n)}]$.~All these cohomology classes now live in the same space $H^0_c(U, j^*(\sF \boxtimes \mathbb{D}_X \sF))$.

Let $u$ be a cohomological correspondence of $\sF$ lifting $c$.~Recall that $u$ induces a linear function $\Phi_u$ on $H^0_c(U, j^*(\sF \boxtimes \mathbb{D}_X \sF))$ (see (\ref{linear_functional})).~Moreover by Lemma \ref{graph_transpose_contained_in_open} since $\text{Fix}(c^{(n)})$ is proper we can make sense of the local terms $\text{LT}(u^{(n)})$ (see Definition \ref{Local_term}).

\begin{prop}\label{local_terms_using_functional}
For any $n \geq 0$,~$\Phi_u([\Delta^{(n)}])=\text{LT}(u^{(n)})$.
\end{prop}

\begin{rmk}
For $n \geq 1$,~both $u^{(n)}$ and the cohomology classes $[\Delta^{(n)}]$ depend on the choice of the Weil structure on $\sF$.
\end{rmk}

\begin{proof}

Recall that (see (\ref{linear_functional}) and (\ref{local_term_proper})),

\begin{equation}\label{definition_global_trace}
\Phi_u([\Delta^{(n)}])=\text{Tr}_{X \times_k X}(\Phi(u \otimes [\Delta^{(n)}])),
\end{equation}

\noindent and 

\begin{equation}
\text{LT}(u^{(n)})=\text{Tr}_{\text{Fix}(c^{(n)})}(\mathcal{T}\textit{r}_c(u^{(n)})).
\end{equation}

Moreover Lemma \ref{two_ways_Frobenius_twist_same_Varshavsky} implies that it suffices to show 

\begin{equation}\label{equivalent_equality_to_be_shown}
\text{Tr}_{X \times_k X}(\Phi(u \otimes [\Delta^{(n)}]))=\text{Tr}_{\text{Fix}(c^{(n)})}(\mathcal{T}\textit{r}^{'(n)}_c(u)).
\end{equation}

The trace maps $\text{Tr}_{X \times_k X}$ and $\text{Tr}_{\text{Fix}(c^{(n)})}$ by their definitions are compatible with the natural push forward map

\begin{equation}\label{adjunction_fixed_canonical}
\psi^{(n)}:H^0(\text{Fix}(c^{(n)}),K_{\text{Fix}(c^{(n)})}) \to H^0(X \times_k X,~K_{X \times_k X}),
\end{equation}

Thus it suffices to show that the image of $\mathcal{T}\textit{r}^{'(n)}_c(u) \in H^0(\text{Fix}(c^{(n)}),K_{\text{Fix}(c^{(n)})})$ under (\ref{adjunction_fixed_canonical}) is $\Phi(u \otimes [\Delta^{(n)}])$.

Unwinding the definitions of $\Phi$ and $\mathcal{T}\textit{r}^{'(n)}_c$,~one observes that to prove the Proposition it suffices to show that the following diagram is commutative

\begin{equation}\label{local_terms_using_functional_1}
\xymatrix{\text{Hom}(c_*\bar{\Q}_{\ell},K_{X \times_k X}) \ar[d]_{(3)} & \text{Hom}(c_*\bar{\Q}_{\ell},\Delta^{(n)}_*K_X) \ar[d]^-{(1)}_-{\cong} \ar[l]^-{(2)} \\
               \text{Hom}(\bar{\Q}_{\ell},K_{X \times_k X}) \ar[d]_-{\cong}  & \text{Hom}(c_!\bar{\Q}_{\ell},\Delta^{(n)}_*K_X) \ar[d]^-{(A)}_-{\cong}\\
              H^0(X \times_k X,K_{X \times_k X})      & \text{Hom}(\bar{\Q}_{\ell},c^!\Delta^{(n)}_*K_X) \ar[d]^-{(BC)}_-{\cong}\\
                H^0(\text{Fix}(c^{(n)}),K_{\text{Fix}(c^{(n)})}) \ar[u]^{(\ref{adjunction_fixed_canonical})}  &  \text{Hom}(\bar{\Q}_{\ell},\Delta^{(n)'}_*K_{\text{Fix}(c^{(n)})}) \ar[l]^-{\cong}  }.
\end{equation}       

Here the maps $(2)$ and $(3)$ are induced by natural maps $\Delta^{(n)}_*K_X \to K_{X \times_k X}$ and $\bar{\Q}_{\ell} \to c_* \bar{\Q}_{\ell}$ respectively.~The map in $(1)$ is induced by the natural map $c_! \bar{\Q}_{\ell} \to c_* \bar{\Q}_{\ell}$.~$(1)$ is an isomorphism follows from applying the cohomological functor $\text{Hom}(\hspace{1mm},\Delta^{(n)}_*K_{X})$ to the triangle (\ref{pushoforward_canonical_exact_above}) and noting that $\text{Hom}(i_*\hspace{1mm},\Delta^{(n)}_*)$ vanishes,~thanks to Lemma \ref{graph_transpose_contained_in_open}.

Our strategy is simple.~We begin with an element $\gamma \in H^0(\text{Fix}(c^{(n)}),K_{\text{Fix}(c^{(n)})})$,~then using the isomorphisms  above we construct a map $\gamma_!:c_!\bar{\Q}_{\ell} \to \Delta^{(n)}_*K_X$,~and its unique lift $\gamma_*:c_*\bar{\Q}_{\ell} \to \Delta^{(n)}_*K_X$.~Finally we show that $\psi^{(n)}(\gamma) \in H^0(X \times_k X ,K_{X \times_k X})$, corresponds to the composition 

\begin{equation}
\xymatrix{\bar{\Q}_{\ell} \ar[r] & c_*\bar{\Q}_{\ell} \ar[r]^-{\gamma_*} & \Delta^{(n)}_*K_X \ar[r] & K_{X \times_k X}},
 \end{equation}                                                                    

\noindent and thus establishing the commutativity of (\ref{local_terms_using_functional_1}).

We have a Cartesian diagram

\begin{equation}\label{local_terms_using_functional_2}
\xymatrixcolsep{4pc} \xymatrix{ \text{Fix}(c^{(n)}) \ar[r]^{\tilde{c}^{(n)'}}  \ar@{^{(}->}[d]^{\Delta^{(n)'}} & X  \ar@{^{(}->}[d]^{\Delta^{(n)}} \\
                   C \ar[r]^{c}                    & X \times X  }.
\end{equation}

Since $\text{Fix}(c^{(n)})$ is proper over $k$ (Lemma \ref{neighbourhood_fixed_points}),~the natural maps

\begin{equation}\label{local_term_using_functional_2''}
c_!\Delta^{(n)'}_* \to c_* \Delta^{(n)'}_*
\end{equation}

\noindent and

\begin{equation}\label{local_term_using_functional_2'''}
\Delta^{(n)}_*\tilde{c}^{(n)'}_! \to \Delta^{(n)}_* \tilde{c}^{(n)'}_*
\end{equation}

are isomorphisms.~Moreover commutativity of the Digram (\ref{local_terms_using_functional_2}) implies that the maps 

\begin{equation}\label{local_term_using_functional_2''''}
c_!\Delta^{(n)'}_* \to  \Delta^{(n)}_* \tilde{c}^{(n)'}_!
\end{equation}

\noindent and

\begin{equation}\label{local_term_using_functional_2'''''}
 c_* \Delta^{(n)'}_* \to   \Delta^{(n)}_* \tilde{c}^{(n)'}_*
\end{equation}

are isomorphisms.

We have a natural adjunction map

\begin{equation}\label{local_terms_using_functional_3_1}
\alpha_n: \tilde{c}_{!}^{(n)}K_{\text{Fix}(c^{(n)})} \to K_X.
\end{equation}

Now suppose we have a global section of $K_{\text{Fix}(c^{(n)})}$ (say $\gamma$).~
Note that the the isomorphisms on the right hand side of (\ref{local_terms_using_functional_1}) imply that the element in $\text{Hom}(c_!\bar{\Q}_{\ell},~\Delta^{(n)}_*K_{X})$ (say $\gamma_!$) corresponding to $\gamma$ is obtained by composing 

\begin{equation}\label{local_terms_using_functional_7}
\gamma_! \colon \xymatrix{c_!\bar{\Q}_{\ell}  \ar[r]^-{c_!\gamma} & c_!\Delta^{(n)'}_*K_{\text{Fix}(c^{(n)})} \ar[r]^-{(\ref{local_term_using_functional_2''''})}_-{\cong} &  \Delta^{(n)}_*\tilde{c}^{(n)'}_!K_{\text{Fix}(c^{(n)})} \ar[r]^-{\Delta^{(n)}_*\alpha_n} & \Delta^{(n)}_*K_{X}} .
\end{equation}

But the isomorphism $(1)$ in Diagram \ref{local_terms_using_functional_1} (yet to be shown commutative) implies any such $\gamma_!$ has an unique lift (say $\gamma_*$) into an element in $\text{Hom}(c_*\bar{\Q}_{\ell},~\Delta^{(n)}_*K_X)$. The commutative diagram (\ref{local_terms_using_functional_8}) allows us to construct a (and hence the only) lift of $\gamma_!$.~Indeed given any such $\gamma$ we have a diagram

\begin{equation}\label{local_terms_using_functional_8}
\xymatrix{& c_!\bar{\Q}_{\ell} \ar[d] \ar[r]^-{c_!\gamma} & c_!\Delta^{(n)'}_*K_{\text{Fix}(c^{(n)})} \ar[r]^-{(\ref{local_term_using_functional_2''''})}_-{\cong} \ar[d]_{(\ref{local_term_using_functional_2''})}^-{\cong} &  \Delta^{(n)}_*\tilde{c}^{(n)'}_!K_{\text{Fix}(c^{(n)})} \ar[r]^-{\Delta^{(n)}_*\alpha_n} \ar[d]_{(\ref{local_term_using_functional_2'''})}^-{\cong} & \Delta^{(n)}_*K_{X} \ar[r] & K_{X \times_k X} \\
               \bar{\Q}_{\ell} \ar[r] & c_* \bar{\Q}_{\ell} \ar[r]^-{c_*\gamma} & c_*\Delta^{(n)'}_*K_{\text{Fix}(c^{(n)})} \ar[r]^-{(\ref{local_term_using_functional_2'''''})}_-{\cong} &  \Delta^{(n)}_*\tilde{c}^{(n)'}_*K_{\text{Fix}(c^{(n)})}  & & }.
\end{equation}

The left hand square in (\ref{local_terms_using_functional_8}) is commutative by functoriality of the map $c_! \to c_*$,~and that the right-hand square is commutative follows from the commutativity of the Diagram (\ref{local_terms_using_functional_2}).

The composite arrow from $c_*\bar{\Q}_{\ell}$ to  $\Delta^{(n)}_*K_X$ in (\ref{local_terms_using_functional_8}) is the unique lift $\gamma_*$ of $\gamma_!$.~Since applying global sections functor to the map $c_*\gamma$ in (\ref{local_terms_using_functional_8}) gives the element $\gamma$ in $H^0(\text{Fix}(c^{(n)},~K_{\text{Fix}(c^{(n)})}))$, the element $\psi^{(n)}(\gamma)$ in $H^0(X \times_k X,K_{X \times_k X})$ corresponds to the composition  

\begin{equation}
\xymatrix{\bar{\Q}_{\ell} \ar[r] & c_* \bar{\Q}_{\ell} \ar[r]^-{\gamma_*} & \Delta^{(n)}_*K_X}
 \end{equation}    

\noindent as desired.

\end{proof}

 \subsection{Action of partial Frobenius on the pairing}

Recall that we were interested in understanding the local terms $\text{LT}(u^{(n)})$ of a cohomological correspondence of a mixed sheaf,~lifting a correspondence which is essentially proper over $\F_q$.

So far from the geometric side, we have only used the properness of $X$ (and $\text{Fix}(c^{(n)})$),~and the fact that there exists an open $U \hookrightarrow X \times_k X$ containing all the graphs of Frobenius.~Also we have only required a Weil structure on $\sF$ so far.~To use Proposition \ref{local_terms_using_functional} to bound the local terms we will need further information on the geometry of $U$,~and the Weil structure on $\sF$.

Now we use the fact that $U$ is \textit{stable under the partial Frobenius} $F_l \colonequals F_{X/\F_q} \times 1_X $,~that is $F_l^{-1}(U)=U$.~As before $F_l^0$ is to be understood as the identity morphism of $X \times_k X$.

More generally let $\sP$ be the set whose elements are open subsets of $X \times_k X$ of the form $\cup_i U_i \times_k U_i$,~with $U_i$ open in $X$ and defined over $\F_q$.~We do not require that $\{U_i\}$'s cover $X$.~Clearly $\sP$ is stable under finite unions and intersections.

Let $V$ be an arbitrary element in $\sP$.~Then $V$ is also stable under the partial Frobenius,~and for any integer $n \geq 0$ we have a Cartesian diagram

\begin{equation}\label{Cartesian_partial_Frobenius}
\xymatrix{V \ar@{^{(}->}[r]^-{j_V} \ar[d]_-{F^n_l} & X \times_k X \ar[d]^-{F^n_l} \\
                V \ar@{^{(}->}[r]^-{j_V} & X \times_k X}.
\end{equation}

Here $j_V$ is the open immersion of $V$ into $X \times_k X$.

As before let $F_{\sF}$ be the structure of a Weil sheaf on $\sF$ via (\ref{Weil_Sheaf}).~The diagram (\ref{Cartesian_partial_Frobenius}) and functoriality of the adjunction $j_{V!}j_{V}^* \to 1$,~implies that there exists a commutative diagram,

\begin{equation}\label{Action_cohomology_partial_Frobenius}
\xymatrix{ F_l^{n*}j_{V!}j_V^*(\sF \boxtimes \mathbb{D}_X \sF)  \ar[r]^-{(\text{BC})}_-{\cong} \ar[d] & j_{V!}j_{V}^*F_l^{n*}(\sF \boxtimes \mathbb{D}_X \sF) \ar[r]^-{\cong}_-{(1)} \ar[d] & j_{V!}j_{V}^*(\sF \boxtimes \mathbb{D}_X \sF) \ar[d]  \\
                 F_l^{n*}(\sF \boxtimes \mathbb{D}_X \sF) \ar@{=}[r] & F_l^{n*}(\sF \boxtimes \mathbb{D}_X \sF)  \ar[r]^-{\cong}_-{(2)} & \sF \boxtimes \mathbb{D}_X \sF}
\end{equation}

\noindent where the isomorphisms $(1)$ and $(2)$ are induced by the Weil structure of $\sF$,~and $(\text{BC})$ arises from base change.

Since $F^n_l$ is finite (and hence proper) taking global sections along the top row of (\ref{Action_cohomology_partial_Frobenius}) we have an induced action of $F^n_l$ on compactly supported cohomology

\begin{equation}\label{pullback_cohomology}
F_{l}^{n*} \colon H^i_c(V,j_V^*(\sF \boxtimes \mathbb{D}_X \sF)) \to H^i_c(V,j_V^*(\sF \boxtimes \mathbb{D}_X \sF)).
\end{equation}

In particular since $U \in \sP$,~for any integer $n \geq 0$ we have an action of $F_l^{n*}$ on $H^0_c(U,j^*(\sF \boxtimes \mathbb{D}_X \sF))$.

For any sheaf $\sH$ on $X \times_k X$,~and any closed subset $Y \hookrightarrow X \times_k X$ there is a pullback map

\begin{equation}\label{partial_Frobenius_pullback_cohomology_with_supports}
H^0_{Y}(X \times_k X,\sH) \to H^0_{(F_l^{n})^{-1}(Y)} (X \times_k X,F_l^{n*}\sH),
\end{equation}

\noindent which fits into a commutative diagram

\begin{equation}\label{partial_Frobenius_pullback_cohomology_with_supports_1}
\xymatrix{H^0_{Y}(X \times_k X,\sH) \ar[r]^-{(\ref{partial_Frobenius_pullback_cohomology_with_supports})}  \ar[d] & H^0_{(F_l^{n})^{-1}(Y)} (X \times_k X,F_l^{n*}\sH) \ar[d] \\
                H^0(X \times_k X, \sH) \ar[r] & H^0(X \times_k X, F^{n*}_l\sH) }.
\end{equation}

The vertical arrows in (\ref{partial_Frobenius_pullback_cohomology_with_supports_1}) are induced by the forget supports map,~and the arrow in the bottom row is induced by pullback.

Recall that given a structure of a Weil sheaf on $\sF$ we had defined cohomology classes $[\Delta^{(n)}]$ in $H^0_c(U,j^*(\sF \boxtimes \mathbb{D}_X \sF))$ (see (\ref{global_cohomology_classes})).

\begin{lem}\label{Frobenius_pullback_cohomology_classes}
For any $n \geq 0$,~$(F_l^{*})^n([\Delta])=[\Delta^{(n)}]$.
\end{lem}

\begin{proof}

First note that as closed subschemes of $X \times_k X$,~$(F_{l}^{n})^{-1}(\Delta) = \Delta^{(n)}$.~Since $F_l^{n*}=(F_l^*)^n$ as endomorphisms of $H^0_c(U,j^*(\sF \boxtimes \mathbb{D}_X \sF))$,~it suffices to prove $F^{n*}_l([\Delta])=[\Delta^{(n)}]$.~For ease of writing let $\sG \colonequals \sF \boxtimes \mathbb{D}_X \sF$.~Consider the diagram of cohomology groups

\begin{equation}\label{Frobenius_pullback_cohomology_classes_1}
\xymatrix{H^0_{\Delta^{(n)}}(X \times_k X,\sG)  \ar[r]^{(\ref{first_specific_isomorphism_supports})}_{\cong} & H^0_{\Delta^{(n)}}(X \times_k X,j_!j^*\sG) \ar[r]^-{} & H^0_c(U,j^*\sG)\\
               H^0_{\Delta^{(n)}}(X \times_k X,F_l^{n*}\sG) \ar[u]^-{\cong} \ar[r]^{(\ref{first_specific_isomorphism_supports})}_{\cong} & H^0_{\Delta^{(n)}}(X \times_k X,F_l^{n*}j_!j^*\sG) \ar[r] \ar[u]^-{\cong} & H^0_c(U,F_l^{n*}j^*\sG) \ar[u]^-{\cong} \\
               H^0_{\Delta}(X \times_k X,\sG) \ar[u]^-{(\ref{partial_Frobenius_pullback_cohomology_with_supports})} \ar[r]^{(\ref{first_specific_isomorphism_supports})}_{\cong} & H^0_{\Delta}(X \times_k X,j_!j^*\sG) \ar[u]^-{(\ref{partial_Frobenius_pullback_cohomology_with_supports})} \ar[r]^-{} & H^0_c(U,j^*\sG) \ar[u] \ar@/_5pc/[uu]_-{F^{n*}_l}}.
\end{equation}

The isomorphisms along the vertical arrows of (\ref{Frobenius_pullback_cohomology_classes_1}) are induced by the Weil structure (see (\ref{Action_cohomology_partial_Frobenius})).~The arrow from $H^0_c(U,j^*\sG)$ to $H^0_c(U,F_l^{n*}j^*\sG)$ is induced by pullback,~and the arrow from $H^0_{\Delta^{(n)}}(X \times_k X,F_l^{n*}j_!j^*\sG)$ to $ H^0_c(U,F_l^{n*}j^*\sG)$ is the composition of base change and forget supports map.

That the square on the upper-left corner of Diagram \ref{Frobenius_pullback_cohomology_classes_1} commutes is a consequence of the outer commutative square in (\ref{Action_cohomology_partial_Frobenius}).~The square in the upper-right corner of (\ref{Frobenius_pullback_cohomology_classes_1}) commutes by functoriality of the forget supports map applied to $F_l^{n*}j_!j^*\sG \to j_!j^*\sG$.~The square in the bottom-left corner commutes by the definition of pullback map (\ref{partial_Frobenius_pullback_cohomology_with_supports}).~The square in the bottom-right corner commutes as a consequence of the commutative diagram (\ref{partial_Frobenius_pullback_cohomology_with_supports_1}).~Thus the Diagram \ref{Frobenius_pullback_cohomology_classes_1} is commutative.

Recall that the cohomology classes $[\Delta^{(n)}]$ begin their life in $H^0_{\Delta^{(n)}}(X \times_k X,\sG)$.~Thus to prove the Lemma it suffices to show that image of the class $[\Delta] \in  H^0_{\Delta}(X \times_k X,\sG)$ along the left column is $[\Delta^{(n)}] \in H^0_{\Delta^{(n)}}(X \times_k X,\sG)$.~This is immediate from the definition of $[\Delta^{(n)}]$ (see (\ref{The_cohomology_Classes})).
\end{proof}

%\begin{lem}\label{cohomological_dimension}
%For any scheme $Y/k$ and any sheaf $\sF \in D^{\leq 0}(Y, \bar{\Q}_{\ell})$, $H^i_c(Y,\sF)$ vanishes for $i>2\text{dim}(Y)$.
%\end{lem}
%
%\begin{proof}
%Since the cohomological dimension of $Y$ is $2\text{dim}(Y)$, this is an immediate consequence of $\sF \in D^{\leq 0}(Y,~\bar{\Q}_{\ell})$. 
%\end{proof}

Now suppose $\sF$ arises from a mixed sheaf $\sF_0$ on $X_0$ (the chosen model of $X$) of \textit{weight less than or equal to} $w$.~Further assume that $\sF_0$ is in $^pD^{\leq a}(X_0,~\bar{\Q}_{\ell})$.~Thus $\sF$ comes equipped with a canonical structure of a Weil sheaf.

\begin{lem}\label{partial_weights_Frobenius}
For any $V \in \sP$ and any integer $n$,~the eigenvalues of $F_l^*$ acting on $H^{n}_c(V,j_V^*(\sF \boxtimes \mathbb{D}_X \sF))$ are Weil numbers of weight less than or equal to $w+a+\text{dim}(X)$.
\end{lem}

\begin{proof}

If $V=V' \cup V''$,~where $V'$ and $V''$ are elements in $\sP$,~the Mayer-Vietoris sequence implies that the Lemma is true for $V$,~if it is true for $V',~V''$ and $V' \cap V''$.

Now suppose $V=(U_1 \times U_1) \bigcup (\cup_{i=2}^{r} (U_i \times_k U_i))$ is an union of $r \geq 2$ open sets of the form $U_i \times_k U_i$.~Note that the intersection $(U_1 \times U_1) \cap \cup_{i=2}^{r} (U_i \times_k U_i) = \cup \cup_{i=2}^{r} ((U \cap U_i) \times_k (U \cap U_i))$,~is an union of atmost $r-1$ open sets of the form $U_i \times_k U_i$.~Thus Mayer-Vietoris allows us to reduce to the case $r=1$.

Hence suppose $V=U \times_k U$ for an open subset $U$ of $X$ defined over $\F_q$.~Then Kunneth formula and Poincare duality imply that for any integer $n$ 
\begin{equation}\label{partial_weights_Frobenius_2}
H^n_c(U \times_k U,~j^*(\sF \boxtimes \mathbb{D}_X \sF)) \simeq \oplus_i (H^i_c(U ,\sF|_{U}) \otimes_{\bar{\Q}_{\ell}} (H^{i-n}(U,\sF|_{U}))^{\vee}),
\end{equation}

\noindent and that the action of $F_{l}^*$ on $H^j_c(U \times_k U,~j^*(\sF \boxtimes \mathbb{D}_X \sF))$ under the isomorphism (\ref{partial_weights_Frobenius_2}) corresponds to $F_{U/\F_q}^* \otimes 1$ on each factor  $H^i_c(U,\sF|_{U}) \otimes_{\bar{\Q}_{\ell}} (H^{i-n}(U,\sF|_{U}))^{\vee})$.~Here $F_{U/\F_q}^*$ is the Frobenius pullback on $H^i_c(U,\sF|_{U})$.

Since $\sF$ comes from a mixed sheaf $\sF_0$,~the eigenvalues of Frobenius on $H^i_c(U,\sF|_{U})$ are Weil numbers (\cite{Deligne_Weil_II},~Th\'eor\`eme 1.1).~Moreover since $\sF_0$ is of weight less than or equal to $w$,~ the Frobenius weights on $H^i_c(U,\sF|_{U})$ are bounded above by $w+i$ (\cite{BBD},~5.1.14 (i)).~Also $\sF|_{U}$ continues to be in $~^pD^{\leq a}$ (\cite{BBD},~Proposition 1.4.16 (i)).~Thus $H^i_c(U,\sF|_{U})$ vanishes for $i>\text{dim}(X)+a$ (\cite{BBD}, Section 4.2.4).

Hence (\ref{partial_weights_Frobenius_2}) implies that the eigenvalues of $F_l^*$ acting on any $H^n_c(V,j_V^*(\sF \boxtimes \mathbb{D}_X \sF))$ are Weil numbers of weight less than or equal to $w+a+\text{dim}(X)$.

\end{proof}

We will need the following elementary linear algebra lemma,~ whose proof is carried out for the sake of completion.

\begin{lem}\label{linear_algebra_bound}
Let $V$ be a finite dimensional complex vector space.~Let $T:V \to V$ be a linear map,~and $w \in \R^{\geq 1}$,~a positive real number such that all the eigenvalues $\alpha$  of $T$ satisfy $|\alpha| \leq w$.

Let $\Phi$ be a linear functional on $V$.~Then for any $v \in V$ 

\begin{enumerate}
\item there exists a positive real number $M_v$ such that,~for any $n \geq 0$

\begin{equation}\label{linear_algebra_bound_1}
|\Phi(T^nv)| \leq M_v n^{\text{dim}(V)}w^{n}.
\end{equation}

\item The power series $\sum_{n \geq 1} \Phi(T^nv)t^n \in \C[[t]] \subseteq \C((t))$,~is a rational function of $t$,~that is it belongs to $\C(t) \subseteq \C((t))$

\end{enumerate}

\end{lem}

\begin{proof}
If $T$ is semi-simple,~choosing a basis in which $T$ is diagonal we have 

\begin{equation}\label{linear_algebra_bound_2}
|\Phi(T^nv)| \leq M'_v w^{n},
\end{equation}

\noindent for some constant $M'_v$ possibly depending on $v$.~Also $(2)$ is immediate,~by choosing a basis in which $T$ is diagonal.

In general $T$ can be written as a sum of two commuting operators $T^{\text{ss}}$ and $T^{\text{nil}}$,
\begin{equation}\label{linear_algebra_bound_3}
T=T^{\text{ss}}+T^{\text{nil}},
\end{equation}
\noindent with $T^{\text{ss}}$  semi-simple and $T^{\text{nil}}$  nilpotent.~Thus for any positive integer $n$

\begin{equation}\label{linear_algebra_bound_4}
T^n=\sum_{i=0}^{\text{min}\{n,\text{dim}(V)\}} {n \choose i} (T^{\text{ss}})^{n-i}(T^{\text{nil}})^i.
\end{equation}

Since the set of eigenvalues of $T$ and $T^{\text{ss}}$ are the same,~combining (\ref{linear_algebra_bound_2}) and (\ref{linear_algebra_bound_4}) we get the necessary bound as shown below

\begin{equation}
|\Phi(T^nv)| \leq \sum_{i=0}^{\text{min}\{n,\text{dim}(V)\}}{n \choose i} M'_{(T^{\text{nil}})^iv}w^{n-i} \leq M_vn^{\text{dim}(V)}w^{n},
\end{equation}

\noindent with $M_v \colonequals (\text{dim}(V)+1)\max \limits_{0 \leq i \leq \text{dim}(V)} M'_{(T^{\text{nil}})^iv}$.

$(2)$ is also immediate in the non-semi-simple case as a consequence of the equality (\ref{linear_algebra_bound_4}),~and the result in the semi-simple case.

\end{proof}

\subsection{Proof of Theorem \ref{local_term_growth_good_correspondence}}

Combining Proposition \ref{local_terms_using_functional},~Lemma \ref{Frobenius_pullback_cohomology_classes},~Lemma \ref{partial_weights_Frobenius} and Lemma \ref{linear_algebra_bound},~$(1)$ we have for all $n \geq 0$,

\begin{equation}\label{bound_local_term_invariant_subscheme_main_section_proof_2}
|\text{LT}(u^{(n)})| \leq M(\tau)n^{r}q^{n(\frac{(w+a+\text{dim}(X))}{2})},
\end{equation}

\noindent for some positive real number $M(\tau)$.~Here $r=\text{dim}_{\bar{\Q}_{\ell}}(H^{0}_c(U,~j^*(\sF \boxtimes \mathbb{D}_X \sF)))$.

Given any $\e>0$,~let $N(\e)$ be the smallest positive integer $n$ such that $q^{n \e} \geq n^r$.~Then (\ref{bound_local_term_invariant_subscheme_main_section_proof_2}) implies that for any $n \geq N(\e)$,

\begin{equation}\label{bound_local_term_invariant_subscheme_main_section_proof_3}
|\text{LT}(u^{(n)})| \leq M(\tau)q^{n(\frac{(w+a+\text{dim}(X))}{2}+\e)},
\end{equation}

\noindent as desired.

\begin{rmk}\label{Remark_on_non_uniformity}
The proof of Lemma \ref{linear_algebra_bound} allows us to explicitly determine $M(\tau)$ in (\ref{bound_local_term_invariant_subscheme_main_section_proof_3}).~Indeed if one knew that the action of the partial Frobenius on $H^{0}_c(U,~j^*(\sF \boxtimes \mathbb{D}_X \sF))$ was semi-simple then one could have used (\ref{linear_algebra_bound_2}) instead to obtain a sharper bound as compared to (\ref{bound_local_term_invariant_subscheme_main_section_proof_3}).~Morevover as observed there,~the constant $M(\tau)$ would then only depend on the class $[\Delta]$.

\end{rmk}

\subsection{Proof of Theorem \ref{non_uniform_estimate},~$(2)$}\label{proof_part_II}

We proceed as in the proof of Theorem \ref{non_uniform_estimate},~$(1)$ and $(3)$ in Section \ref{enough_to_bound_local} and conclude that (using the same notations)

\begin{equation}\label{proof_part_II'}
\#\text{Fix}(c^{(n)}(k))=\text{LT}(u^{(n)})-\text{LT}(u^{(n)}|_Z).
\end{equation}

The result now follows from combining the Lefschetz-Verdier trace formula,~Proposition \ref{local_terms_using_functional},~Lemma \ref{linear_algebra_bound},~$(2)$,~and a standard Hankel determinant argument (see for example \cite{Del74},~Lemme 1.7).

\appendix

\section{A cohomological correspondence of the intersection complex}\label{cohomological_IC}

In this section, we will use  geometric semi-simplicity of pure perverse sheaves to construct a cohomological correspondence.~A similar result was earlier obtained in an unpublished note of Hanamura-Saito \cite[Theorem 2]{Hanamura_Saito}.~For the sake of completeness we present an alternative approach here.~Throughout this section let $k_0$ be an arbitrary finite field.~Let $k$ be an algebraic closure of $k_0$.

\subsection{Basic properties of $\text{IC}_X$}

In this section we summarize a few basic properties of $\text{IC}_X$ which will be used later.~Recall for any variety $X/k$ of dimension $d$ there is a natural map

\begin{equation}\label{natural_map_constant_to_intersection_complex}
\bar{\Q}_{\ell}[d] \to \text{IC}_X,
\end{equation}

\noindent which is an isomorphism on the regular locus of $X$.~An analogous map exists for varieties over $k_0$ also.

We can complete (\ref{natural_map_constant_to_intersection_complex}) into a triangle 

\begin{equation}\label{triangle_constant_intersection}
\xymatrix{\bar{\Q}_{\ell}[d] \ar[r] & \text{IC}_X \ar[r] & \sF \ar[r]^-{+1}&}.
\end{equation}

Clearly $\sF$ is supported on the singular locus of $X$.~Since the singular locus is of dimension at most $d-1$,~the restriction of $\bar{\Q}_{\ell}[d]$ to the singular locus is in $~^p{D}^{\leq -1}$ (\cite{BBD},~(4.0.1)').~Further $\text{IC}_X$ by construction,~when restricted to the singular locus belongs to $~^pD^{\leq -1}$.~Thus $\sF$ is in $~^pD^{\leq -1}$,~and has dimension of support at most $d-1$.~Hence $H^{i}_c(X,\sF)$ vanishes for $i \geq d-1$ (\cite{BBD},~ Section 4.2.4).~Thus the natural map

\begin{equation}\label{triangles_constant_intersection_1}
H_c^{2d}(X,\bar{\Q}_{\ell}) \to H_c^{d}(X,\text{IC}_X), 
\end{equation}

\noindent is an isomorphism

Dualizing (\ref{natural_map_constant_to_intersection_complex}) gives a natural map 

\begin{equation}\label{intersection_complex_to_dualizing_complex}
\text{IC}_X \to K_X[-d],
\end{equation}

\noindent which is also an isomorphism on the regular locus.~Dualizing (\ref{triangles_constant_intersection_1}) implies that the natural map

\begin{equation}\label{triangles_constant_intersection_1_dual}
H^{-d}(X,\text{IC}_X) \to  H^{-2d}(X,K_X), 
\end{equation}

\noindent induced by (\ref{intersection_complex_to_dualizing_complex}) is an isomorphism.

%The map in (\ref{natural_map_constant_to_intersection_complex}) corresponds to the trace map on $H^{\text{dim}(X)}_c(X,\text{IC}_X) \simeq H^{2\text{dim}(X)}_c(X^{\text{reg}},\bar{\Q}_{\ell})$, under the Poincare duality isomorphism $H^{-\text{dim}(X)}(X,\text{IC}_X) \simeq H^{\text{dim}(X)}_c(X,\text{IC}_X)^{\vee}$.

For any non-empty open subset $j:U \hookrightarrow X$,~functoriality of the adjunction $j_! j^* \to 1$ applied to (\ref{natural_map_constant_to_intersection_complex}) and (\ref{intersection_complex_to_dualizing_complex}) gives a commutative diagram 

\begin{equation}\label{functoriality_trace_map}
\xymatrix{H^{0}_c(U,K_U) \ar[r]^{\cong} & H_c^0(X,K_X) \ar[r]_-{\text{Tr}_X}^-{\cong} & \bar{\Q}_{\ell} \\
               H^{d}_c(U,\text{IC}_U) \ar[r]^-{\cong} \ar[u] & H_c^{d}(X,\text{IC}_X) \ar[u] \\
                H^{2d}_c(U, \bar{\Q}_{\ell}) \ar[u]^{\cong}_{(\ref{triangles_constant_intersection_1})} \ar[r]^{\cong} & H_c^{2d}(X,\bar{\Q}_{\ell}) \ar[u]^{\cong}_{(\ref{triangles_constant_intersection_1})} }.
\end{equation}

Here $\text{Tr}_X$ is the natural trace map on $H_c^0(X,K_X)$.~Since $X$ is assumed to be a variety and $U$ is a nonempty open subset of $X$,~the top and bottom rows of (\ref{functoriality_trace_map}) are isomorphisms.

If $U$ is contained in the regular locus of $X$ then the arrows in the left column of (\ref{functoriality_trace_map}) are trivially isomorphisms,~and thus the natural map 

\begin{equation}\label{top_cohomology_pushforward}
H_c^{d}(X,\text{IC}_X) \to H^0_c(X,K_X),
\end{equation}

\noindent is also an isomorphism.

By abuse of notation, we will also call the composite map 

\begin{equation}\label{our_definition_trace}
H^{2d}_c(X,\bar{\Q}_{\ell}) \to \bar{\Q}_{\ell}
\end{equation}

\noindent in (\ref{functoriality_trace_map}) as $\text{Tr}_X$.~Note that this coincides with the usual trace map when $X$ is regular.

We can also dualize the diagram (\ref{functoriality_trace_map})

\begin{equation}\label{functoriality_trace_map_dual}
\xymatrix{H^{0}(U,\bar{\Q}_{\ell}) \ar[d] & H^0(X,\bar{\Q}_{\ell}) \ar[l]_-{\cong} \ar[d] & \bar{\Q}_{\ell} \ar[l]_-{\cong}  \\
               H^{-d}(U,\text{IC}_U) \ar[d]_-{\cong}^-{(\ref{triangles_constant_intersection_1_dual})} & H^{-d}(X,\text{IC}_X) \ar[d]_-{\cong}^-{(\ref{triangles_constant_intersection_1_dual})} \ar[l]_-{\cong}  \\
                H^{-2d}(U, K_U) & H^{-2d}(X,K_X) \ar[l]_-{\cong} }.
\end{equation}

Also (\ref{top_cohomology_pushforward}) dualizes to give a natural isomorphism

\begin{equation}\label{top_cohomology_pushforward_dual}
H^{0}(X,\bar{\Q}_{\ell}) \simeq H^{-d}(X,\text{IC}_X).
\end{equation}

\subsection{Decomposition in the derived category and a Lemma}

Let $f_0:Y_0 \to X_0$ be a projective morphism of schemes of finite type over $k_0$. Let $\eta \in H^{2}(Y_0,~\Q_{\ell}(1))$ be the Chern class of a relatively ample line bundle on $Y_0$.~Then $\eta$  defines a map (in $D^b_c(Y_0,\bar{\Q}_{\ell})$)

\begin{equation}\label{map_derived_category_relatively_ample_1}
\eta:\bar{\Q}_{\ell} \to \bar{\Q}_{\ell}(1)[2].
\end{equation}

Tensoring (\ref{map_derived_category_relatively_ample_1}) with $\text{IC}_{Y_0}$,~we get a map

\begin{equation}\label{map_derived_category_relatively_ample_2}
\eta_{\text{IC}_{Y_0}}:\text{IC}_{Y_0} \to \text{IC}_{Y_0}(1)[2].
\end{equation}

Thus for any integer $i$,~(\ref{map_derived_category_relatively_ample_2}) induces maps on the perverse cohomologies

\begin{equation}\label{map_derived_category_relatively_ample_3}
\eta^i:~^pH^{-i}(f_{0*}IC_{Y_0}) \to ~^pH^{i}(f_{0*}IC_{Y_0})(i).
\end{equation}

Purity of $\text{IC}_{Y_0}$ (\cite{BBD},~Corollaire 5.3.2) and the relative hard Lefschetz (\cite{BBD},~Th\'eor\`eme 5.4.10) imply that the maps in (\ref{map_derived_category_relatively_ample_3}) are isomorphisms.~Using these isomorphisms Deligne obtained a canonical self dual isomorphism over $k_0$ (\cite{Deligne_Decomposition},~Section 3)\footnote{If one assume $f_0$ is only \textit{proper}, purity of $\text{IC}_{Y_0}$ \cite[Th\'eor\`eme 5.3.8]{BBD} together with \cite[Proposition 6.2.6]{Deligne_Weil_II}
implies that $f_{0*}\text{IC}_{Y_0}$ is pure and hence by \cite[Th\'eor\`eme 5.4.5]{BBD} there exists a \textit{non-canonical} isomorphism over $k$
\begin{equation}\label{Deligne_Decomposition_over_k}
f_{*}\text{IC}_{Y} \simeq \oplus_{i} ~^pH^i(f_{*}IC_{Y})[-i].
\end{equation} 
It is unclear how to descend this isomorphism to a finite subfield of $k$. I thank the referee for pointing this out.}

%\begin{equation}\label{Deligne_Decomposition_Canonical}
%f_{0*}\text{IC}_{Y_0} \simeq \oplus_{i} ~^pH^i(f_{0*}IC_{Y_0})[-i].
%\end{equation}

\begin{equation}\label{Deligne_Decomposition}
f_{0*}\text{IC}_{Y_0} \simeq \oplus_{i} ~^pH^i(f_{0*}IC_{Y_0})[-i].
\end{equation}

\subsubsection{A Lemma}
Let $j_0:U_0 \hookrightarrow X_0$ be an open immersion.~Let $K_0$ be a perverse sheaf on $X_0$.~The adjoint triple $(j_!,j^*,j_*)$ gives rise to a commutative diagram of perverse sheaves on $X_0$

\begin{equation}\label{pure_perverse_map_1}
\xymatrix{ & \text{image}(\phi_0) \ar@{^{(}->}[d] \ar@{->>}[r]^-{\tilde{\psi_0}} & j_{0!*}j_0^*K_0 \colonequals \text{image}(\psi_0 \circ \phi_0) \ar@{^{(}->}[d] \\
^pj_{0!}j^*_0K_0 \ar@{->>}[ur]^-{\phi_0} \ar[r]^-{\phi_0} & K_0 \ar[r]^-{\psi_0} & ^pj_{0*}j^*_0K_0}.
\end{equation}

\begin{lem}\label{pure_perverse_map}
If in addition, $K_0$ is pure then,~$\tilde{\psi_0}$ is an isomorphism.
\end{lem}

\begin{proof}

To check that $\tilde{\psi_0}$ is an isomorphism we can work geometrically.~Thus we have a diagram analogous to Diagram \ref{pure_perverse_map_1} but over $k$

\begin{equation}\label{pure_perverse_map_2}
\xymatrix{ & \text{image}(\phi) \ar@{^{(}->}[d] \ar@{->>}[r]^-{\tilde{\psi}} & j_{!*}j^*K \ar@{^{(}->}[d] \\
^pj_{!}j^*K \ar@{->>}[ur]^-{\phi} \ar[r]^-{\phi} & K \ar[r]^-{\psi} & ^pj_{*}j^*K},
\end{equation}

\noindent and we need to show that $\tilde{\psi}$ is an isomorphism.

Since $K_0$ is assumed to be pure,~it is geometrically semi-simple (\cite{BBD},~Th\'eor\`eme 5.3.8).~Thus we can assume that $K$ is a simple perverse sheaf.

If $j^*K$ is $0$,~then so are $\text{image}(\phi)$ and $j_{!*}j^*K$,~and $\tilde{\psi}$ is trivially an isomorphism.~Else $\text{image}(\phi)$ is necessarily nonzero (since its restriction to $U$ is nonzero),~and by the simplicity of $K$ is necessarily equal to $K$.~Thus $\text{image}(\phi)$ is also simple,~and hence the surjection $\tilde{\psi}$ is necessarily an isomorphism.

\end{proof}

Lemma \ref{pure_perverse_map} implies that when $K_0$ is a pure perverse sheaf on $X_0$,~there is a natural injection of perverse sheaves

\begin{equation}\label{map_intermediate_extension_1}
j_{0!*}j_0^*K_0 \to K_0,
\end{equation}

\noindent whose restriction to $U_0$ is the natural isomorphism

\begin{equation}\label{map_intermediate_extension_2}
j_0^*j_{0!*}j^*_0K_0 \simeq j^*_0K_0.
\end{equation}

\subsection{A pullback map on $\text{IC}_X$} 

Let $f \colon Y \to X$ be a \textit{dominant} and projective morphism of varieties of the same dimension $d$ over $k$.~Let $f_0 \colon Y_0 \to X_0$ be a choice of model (of $f \colon Y \to X$) over a finite sub field $k_0$ of $k$.

Let $j_0 \colon U_0 \hookrightarrow X_0$ be a non-empty regular open subset of $X_0$. We have a Cartesian diagram

\begin{equation}\label{pullback_intersection_cohomology_-1}
\xymatrix{f_0^{-1}(U_0) \ar[d]^-{f_0'} \ar@{^{(}->}[r]^-{j'_0} & Y_0 \ar[d]^-{f_0} \\
               U_0 \ar@{^{(}->}[r]^-{j_0} & X_0 }.
\end{equation}

By proper base change 

\begin{equation}\label{pullback_intersection_cohomology_-2}
j_0^*f_{0*}\text{IC}_{Y_0} \simeq ~f'_{0*}\text{IC}_{f_0^{-1}(U_0)}.
\end{equation}

On $f_0^{-1}(U_0)$ (\ref{natural_map_constant_to_intersection_complex}) implies that there exists a natural map 

\begin{equation}\label{pullback_intersection_cohomology_-3'}
\bar{\Q}_{\ell}[d]  \to \text{IC}_{f_0^{-1}(U_0)}.
\end{equation}

Since $U_0$ is regular,~we have morphisms

\begin{equation}\label{pullback_intersection_cohomology_-3}
\xymatrix{ j^*_0 \text{IC}_{X_0} \ar[r]^-{\cong} & \bar{\Q}_{\ell}[d] \ar[r]^-{(A)} & f'_{0*}(\bar{\Q}_{\ell}[d]) \ar[r]_-{(\ref{pullback_intersection_cohomology_-3'})} & f'_{0*}\text{IC}_{f_0^{-1}(U_0)} \ar[r]^-{\cong (\text{BC})}_-{(\ref{pullback_intersection_cohomology_-2})} & j_0^*f_{0*}\text{IC}_{Y_0}}.
\end{equation}

Here the morphism $(A)$ is induced by adjunction.~Denote by $f'_{0*}$ the composite map

\begin{equation}\label{pullback_intersection_cohomology_-3''}
f'_{0*} \colon j^*_0 \text{IC}_{X_0} \to  j_0^*f_{0*}\text{IC}_{Y_0},
\end{equation}

\noindent in (\ref{pullback_intersection_cohomology_-3}),~and by $f'_{*}$ the corresponding map 

\begin{equation}\label{pullback_intersection_cohomology_-4}
f'_* \colon  j^* \text{IC}_{X} \to  j^*f_{*}\text{IC}_{Y},
\end{equation}

\noindent obtained by base change of (\ref{pullback_intersection_cohomology_-3''}) to $k$.~By construction of (\ref{pullback_intersection_cohomology_-4}),~we have a commutative diagram of sheaves on $U$

\begin{equation}\label{pullback_intersection_cohomology_-5}
 \xymatrix{\bar{\Q}_{\ell}[d] \ar[d]^-{(A)} \ar[r]^-{\cong} & j^* \text{IC}_{X} \ar[d]^-{(\ref{pullback_intersection_cohomology_-4})}   \\
                  f'_*\bar{\Q}_{\ell}[d] \ar[d]_-{(\ref{natural_map_constant_to_intersection_complex})} \ar[r] & j^*f_{*}\text{IC}_{Y} \\
                     f'_* \text{IC}_{f^{-1}(U)} \ar[ur]_-{(\text{BC})}^-{\cong} &       },
\end{equation}

\noindent where $(A)$ is induced by adjunction.

\begin{rmk}\label{no_confusion}
The $f'_{0*}$ in (\ref{pullback_intersection_cohomology_-3''}) is not to be confused with the functor $f'_{0*}$.~Whenever either makes an appearance it will be clear from the context which one we mean.
\end{rmk}

Having made these choices there is a canonical map as shown in the following lemma.

\begin{prop}\label{pullback_intersection_cohomology}
There is a map $f_* \colon \text{IC}_X \to f_*\text{IC}_Y$ (defined over $k_0$) such that when restricted to $U$ it is the map in (\ref{pullback_intersection_cohomology_-4}).

\end{prop}

\begin{proof}
We begin by choosing a decomposition 

\begin{equation}\label{pullback_intersection_cohomology_1_1}
\phi_0 \colon f_{0*}\text{IC}_{Y_0} \simeq \oplus_{i}~^pH^i(f_{0*}IC_{Y_0})[-i].
\end{equation}

Let $\phi$ be the corresponding isomorphism over $k$.~Let 

\begin{equation}\label{pullback_intersection_cohomology_1_2}
\pi_0 \colon \oplus_{i}  ~^pH^i(f_{0*}IC_{Y_0})[-i] \to ~^pH^{0}(f_{0*}IC_{Y_0})
\end{equation}

\noindent be the projection to the zeroth graded piece.~Let $\pi$ be the corresponding projection over $k$.~Let 

\begin{equation}\label{pullback_intersection_cohomology_1_3}
u_{0}  \colon ~^pH^0(f_{0*}IC_{Y_0}) \to \oplus_{i}  ~^pH^i(f_{0*}IC_{Y_0})[-i] 
\end{equation}

\noindent be the inclusion of the zeroth graded piece,~and denote by $u$ the corresponding inclusion over $k$.

Restricting (\ref{pullback_intersection_cohomology_1_1}) to $U_0$,~and composing with (\ref{pullback_intersection_cohomology_-3''}) gives us maps

\begin{equation}\label{pullback_intersection_cohomology_1_4}
f_{0*}^{'i} \colon j_0^*\text{IC}_{X_0} \to j_0^*~^pH^i(f_{0*}IC_{Y_0})[-i],
\end{equation}

\noindent by projecting on each of the summand.~Let $f_{*}^{'i}$ be the corresponding maps over $k$.

Since $j_0^*\text{IC}_{X_0}$ and $ j_0^*~^pH^i(f_{0*}IC_{Y_0})$ are perverse sheaves on $U_0$,~the maps $f_{0*}^{'i}$ are all zero for $i>0$ (\cite{BBD},~Corollaire 2.1.21).~Further since $^pH^i(f_{0*}IC_{Y_0})$ is pure of weight $i$,~the morphisms $f_{*}^{'i}$ are all zero for $i<0$ (\cite{BBD},~Proposition 5.1.15 (iii)).~Thus

\begin{equation}\label{pullback_intersection_cohomology_1_5}
f'_*=(j^* \phi^{-1}) \circ (j^*u) \circ f^{'0}_*.
\end{equation}

Note that the maps on either side of the equality (\ref{pullback_intersection_cohomology_1_5}) are defined over $k_0$,~but the equality is over $k$.

The map $f^{'0}_{0*}$ is a map of pure perverse sheaves over $k_0$,~thus applying Lemma \ref{pure_perverse_map} in the form of (\ref{map_intermediate_extension_1}),~we get a natural map

\begin{equation}\label{pullback_intersection_cohomology_1_6}
f^{0}_{0*} \colon \text{IC}_{X_0} \to~^pH^0(f_{0*}IC_{Y_0})
\end{equation}

\noindent of perverse sheaves on $X_0$,~which restricts to $f^{'0}_{0*}$ over $U_0$.

We define

\begin{equation}\label{pullback_intersection_cohomology_1_6}
f_{0*} \colonequals \phi^{-1}_0 \circ u_0 \circ  f^{0}_{0*} \colon  \text{IC}_{X_0} \to f_{0*} \text{IC}_{Y_0}.
\end{equation}

Let $f_*$ be the corresponding map over $k$.~Since by construction $f^{0}_{0*}$ restricts to $f^{'0}_{0*}$ over $U_0$,~(\ref{pullback_intersection_cohomology_1_5}) implies that $f_*$ restricts to $f'_*$ over $U$ as desired.

\end{proof}

We also have the dual map.

\begin{cor}\label{dual_map_intersection_cohomology}
There is natural map 

\begin{equation}\label{dual_map_intersection_cohomology_1}
f^{\vee}_* \colon f_*\text{IC}_Y \to \text{IC}_X
\end{equation}

\noindent (defined over $k_0$) such that when restricted $U$ it is dual to the map (\ref{pullback_intersection_cohomology_-4}).
\end{cor}

We can also apply adjunction to $f_*$ to obtain the following.

\begin{cor}\label{adjunction_map_intersection_cohomology}
There is a natural map 

\begin{equation}\label{adjunction_map_intersection_cohomology_1}
f^* \colon f^*\text{IC}_X \to \text{IC}_Y,
\end{equation}

\noindent which when restricted to $f^{-1}(U)$ is the map in (\ref{natural_map_constant_to_intersection_complex}).
\end{cor}

Dualizing we also have the following result.

\begin{cor}\label{adjunction_dual_map_intersection_cohomology}
There is a natural map 

\begin{equation}\label{adjunction_dual_map_intersection_cohomology_1}
f^! \colon \text{IC}_Y \to f^{!}\text{IC}_X,
\end{equation}

\noindent which when restricted to $f^{-1}(U)$ is dual to the map in (\ref{natural_map_constant_to_intersection_complex}).

\end{cor}

\begin{rmk}\label{Deligne's_decomposition}
Note that compatibility of $f_*|_{U}$ with (\ref{pullback_intersection_cohomology_-4}) is independent of the choice of an isomorphism in (\ref{Deligne_Decomposition}).~Since the argument only depends on vanishing of $f_{*}^{'i}$ for $i \neq 0$,~and all the choices lead to this vanishing.\end{rmk}

In the rest of this section, we assume that $X_0$ (and hence $Y_0$) is projective over $k_0$.

The morphism of sheaves

\begin{equation}\label{action_top_cohomology_1_1}
f_* \colon \text{IC}_X \to f_* \text{IC}_Y,
\end{equation}

\noindent constructed in Proposition \ref{pullback_intersection_cohomology} induces a linear map

\begin{equation}\label{action_top_cohomology_-1}
f^* \colon H^i(X,\text{IC}_X) \to H^i(Y,\text{IC}_Y).
\end{equation}

By properness of $f$ we also have an action on the compactly supported cohomology

\begin{equation}\label{action_top_cohomology_-3}
f^{'*}_c \colon H_c^i(U,\text{IC}_U) \to H_c^i(f^{-1}(U),\text{IC}_{f^{-1}(U)}).
\end{equation}

Moreover, we have a diagram of cohomology groups

\begin{equation}\label{pullback_compactly_supported}
\xymatrix{ H^{i+d}_c(U,\bar{\Q}_{\ell}) \ar[d]^-{f^*} \ar[r]^-{\cong} & H^{i}_c(U,\text{IC}_U) \ar[r] \ar[d]^-{f^{'*}_c}_-{(\ref{action_top_cohomology_-3})} & H^i(X,\text{IC}_X) \ar[d]^-{f^*}_-{(\ref{action_top_cohomology_-1})} \\ 
      H^{i+d}_c(f^{-1}(U),\bar{\Q}_{\ell}) \ar[r]& H^{i}_c(f^{-1}(U),~\text{IC}_{f^{-1}(U)}) \ar[r] & H^i(Y,\text{IC}_Y) }.
\end{equation}

The square in the left of the Diagram \ref{pullback_compactly_supported} is commutative as a result of (\ref{pullback_intersection_cohomology_-5}).~The square on the right commutes by the functoriality of the adjunction $j_!j^* \to 1$.~When $i=d$,~the Diagram \ref{functoriality_trace_map} implies that all the row maps in the Diagram \ref{pullback_compactly_supported} are isomorphisms.
%
%Similarly the adjoint pair $(j^*,~j_*)$,~(\ref{pullback_intersection_cohomology_9}) and the commutative diagram (\ref{pullback_intersection_cohomology_-5}) implies that there exists a commutative diagram of cohomology groups
%
%%\end{equation}

We also have a commutative diagram of usual cohomology groups

\begin{equation}\label{pullback_compactly_supported_1}
\xymatrix{ \bar{\Q}_{\ell} \ar[d]_-{\text{deg}(f)}  \ar[r]^-{\cong}_-{\text{Tr}_X^{-1}} & H^{2d}(X,\bar{\Q}_{\ell})  \ar[d]^-{f^*} & H^{2d}_c(U,\bar{\Q}_{\ell}) \ar[l]_-{\cong} \ar[d]^-{f^*}  \\ 
               \bar{\Q}_{\ell} & H^{2d}(Y,\bar{\Q}_{\ell})  \ar[l]_-{\cong}^-{\text{Tr}_Y} &  H^{2d}_c(f^{-1}(U),\bar{\Q}_{\ell}) \ar[l]_-{\cong} }.
\end{equation}

Here $\text{Tr}_X$ and $\text{Tr}_Y$ are the trace maps on the top cohomology (\ref{our_definition_trace}).~Finally combining the Diagrams \ref{pullback_compactly_supported} and \ref{pullback_compactly_supported_1} we have the following result.

\begin{lem}\label{action_top_cohomology}
The following diagram is commutative 

\begin{equation}\label{action_top_cohomology_1}
\xymatrix{\bar{\Q}_{\ell} \ar[d]_-{\text{deg}(f)}  \ar[r]^-{\cong}_-{\text{Tr}_X^{-1}} & H^{2d}(X,\bar{\Q}_{\ell}) \ar[r]^-{\cong}_-{(\ref{triangles_constant_intersection_1})} \ar[d]^-{f^*} & H^d(X,\text{IC}_X) \ar[d]^-{f^*}_-{(\ref{action_top_cohomology_-1})} \ar[d]^-{f^*} \\ 
              \bar{\Q}_{\ell} & H^{2d}(Y,\bar{\Q}_{\ell})  \ar[l]_-{\cong}^-{\text{Tr}_Y} & H^d(Y,\text{IC}_Y) \ar[l]_-{\cong}^-{(\ref{triangles_constant_intersection_1})}}.
\end{equation}

\end{lem}

We can also dualize the arguments above.~The morphism (\ref{dual_map_intersection_cohomology_1}) induces a pushforward on the intersection cohomology groups

\begin{equation}\label{pushforward_cohomology_all}
f_* \colon H^{i}(Y,\text{IC}_Y) \to H^{i}(X,\text{IC}_X).
\end{equation}

As in Lemma \ref{action_top_cohomology} we would like to understand this action when $i=d$.~Note that by construction

\begin{equation}\label{pushforward_cohomology}
f_* \colon H^{d}(Y,\text{IC}_Y) \to H^{d}(X,\text{IC}_X).
\end{equation}

 \noindent is dual to the pullback map

\begin{equation}\label{pushforward_cohomology_dual'}
f^* \colon H^{-d}(X,\text{IC}_X) \to H^{-d}(Y,\text{IC}_Y),
\end{equation}

\noindent induced by (\ref{action_top_cohomology_1_1}).~Moreover restricting (\ref{action_top_cohomology_1_1}) to $U$,~and we obtain pullback maps

 \begin{equation}\label{pushforward_cohomology_dual}
f^{'*} \colon H^{i}(U,\text{IC}_U) \to H^{i}(f^{-1}(U),\text{IC}_{f^{-1}(U)}).
\end{equation}

As before the commutative diagram (\ref{pullback_intersection_cohomology_-5}) and functoriality of the adjunction $1 \to j_*j^*$ gives rise to a commutative diagram

\begin{equation}\label{pullback_usual}
\xymatrix{ H^{i+d}(X,\bar{\Q}_{\ell}) \ar[r] \ar[d]^-{f^*} & H^{i+d}(U,\bar{\Q}_{\ell}) \ar[d]^-{f^*} \ar[r]^-{\cong} & H^{i}(U,\text{IC}_U) \ar[d]^-{f^{'*}}_-{(\ref{pushforward_cohomology_dual})} & H^i(X,\text{IC}_X) \ar[d]^-{f^*}_-{(\ref{action_top_cohomology_-1})} \ar[l]  \\ 
       H^{i+d}(Y,\bar{\Q}_{\ell}) \ar[r]  & H^{i+d}(f^{-1}(U),\bar{\Q}_{\ell}) \ar[r] & H^{i}(f^{-1}(U),~\text{IC}_{f^{-1}(U)}) \ar[r] & H^i(Y,\text{IC}_Y) \ar[l] }.
\end{equation}

In particular when $i=-d$,~the Diagram \ref{functoriality_trace_map_dual} implies that all the row maps in the Diagram \ref{pullback_usual} are isomorphisms.~Thus we get a commutative diagram 

\begin{equation}\label{pullback_usual_2}
\xymatrix{\bar{\Q}_{\ell} \ar[r]^-{\cong} \ar@{=}[d] & H^{0}(X,\bar{\Q}_{\ell}) \ar[r]^-{\cong}_-{(\ref{top_cohomology_pushforward_dual})} \ar[d]^-{f^*} & H^{-d}(X,\text{IC}_X) \ar[d]^-{f^*}_-{(\ref{pushforward_cohomology_dual'})} \ar[d]^-{f^*} \\ 
              \bar{\Q}_{\ell} \ar[r]^-{\cong} & H^{0}(Y,\bar{\Q}_{\ell})  & H^{-d} (Y,\text{IC}_Y) \ar[l]_-{\cong}^-{(\ref{top_cohomology_pushforward_dual})}}.
\end{equation}

Dualizing the Diagram \ref{pullback_usual_2} we get a commutative diagram

\begin{equation}\label{pullback_usual_3}
\xymatrix{ \bar{\Q}_{\ell}  & H^{0}(X,K_X) \ar[l]_-{\cong}^-{\text{Tr}_X} & H^{d}(X,\text{IC}_X) \ar[l]_-{\cong}^-{(\ref{top_cohomology_pushforward})} \\
               \bar{\Q}_{\ell} \ar@{=}[u] &  H^{0}(Y,K_Y) \ar[l]_-{\cong}^-{\text{Tr}_Y} \ar[u]^-{f_*} &  H^{d}(Y,\text{IC}_Y) \ar[l]_-{\cong}^-{(\ref{top_cohomology_pushforward})} \ar[u]^-{f_*}_-{(\ref{pushforward_cohomology})} }.
\end{equation}

Here $f_* \colon H^0(Y,K_Y) \to H^0(X,K_X)$ is induced by the adjunction $f_*f^! \to 1$ and is dual to the pullback map $f^* \colon H^0(X,\bar{\Q}_{\ell}) \to H^0(Y,\bar{\Q}_{\ell})$.~Thus combining Diagrams \ref{pullback_usual_3} and \ref{functoriality_trace_map} we have a commutative diagram  

\begin{equation}\label{pushforward_trace_compatible}
\xymatrix{ H^{2d}(Y,\bar{\Q}_{\ell})  \ar@/^5pc/[rrd]^-{\text{Tr}_Y}_-{(\ref{our_definition_trace})} \ar[r]^-{\cong}_{(\ref{triangles_constant_intersection_1})} & H^{d}(Y,\text{IC}_Y) \ar[dd]^-{f_*}_-{(\ref{pushforward_cohomology})} \ar[dr]^-{\cong} & \\
                  &        & \bar{\Q}_{\ell} \\
                  H^{2d}(X,\bar{\Q}_{\ell}) \ar[r]^-{\cong}_-{(\ref{triangles_constant_intersection_1})} \ar@/_5pc/[rru]_-{\text{Tr}_X}^-{(\ref{our_definition_trace})} & H^d(X,\text{IC}_X) \ar[ur]^-{\cong} &} .
\end{equation}

The Diagrams \ref{action_top_cohomology_1} and \ref{pushforward_trace_compatible} immediately imply the following proposition.

\begin{prop}\label{pushforward_pullback_intersection_cohomology}
Let $f_0,g_0$ be two dominant morphisms from $X_0$ to  $Y_0$.~Let $f^* \colon H^{d}(X,\text{IC}_X) \to H^d(Y,\text{IC}_Y)$ and $g_* \colon H^{d}(Y,\text{IC}_Y) \to H^{d}(X,\text{IC}_X)$ be the linear maps as defined in (\ref{action_top_cohomology_-1}) and (\ref{pushforward_cohomology}).~Then 
\begin{equation}\label{pushforward_pullback_intersection_cohomology_1}
g_* \circ f^*=\text{deg}(f),
\end{equation}
as endomorphisms of the one dimensional $\bar{\Q}_{\ell}$ vector space $H^{d}(X,\text{IC}_X)$.
\end{prop}

\end{document}